\documentclass{article}
\usepackage{graphicx} % Required for inserting images
\usepackage{amsmath,amsfonts,amssymb}
\usepackage{algorithm}
\usepackage{algpseudocode}
\usepackage{subfigure}
\usepackage[dvips]{epsfig}
\usepackage[dvips]{graphicx}
\usepackage{layout}
\usepackage{todonotes}
\usepackage{complexity}
\usepackage{mathtools}
\usepackage{booktabs}
\usepackage{hyperref}
\usepackage{tikz}
\usepackage{colortbl}

\usepackage{arydshln}
\usepackage{enumitem}
\DeclareFontFamily{U}{mathb}{}
\DeclareFontShape{U}{mathb}{m}{n}{<-5.5> mathb5 <5.5-6.5> mathb6 
<6.5-7.5> mathb7 <7.5-8.5> mathb8 <8.5-9.5> mathb9 <9.5-11> mathb10 
<11-> mathb12}{}
\DeclareFontSubstitution{U}{mathb}{m}{n}
\DeclareRobustCommand{\blackdiamond}{\mathbin{\text{\usefont{U}{mathb}{m}{n}\symbol{"0C}}}}

\def\NE{\mbox{\tt  N\hspace{-1pt}E}}
\def\dep{=\!\!}

\renewcommand{\PL}{\mathcal{PL}}
\newcommand{\QPL}{\mathcal{QPL}}
\newcommand{\Var}{\mathrm{Var}}
\newcommand{\Prop}{\mathrm{Prop}}
\newcommand{\FV}{\mathrm{FV}}
\newcommand{\Dim}{\mathrm{D}}
\newcommand{\Dom}{\mathrm{Dom}}
\newcommand{\Ran}{\mathrm{Ran}}
\newcommand{\Eval}{\mathrm{Eval}}
\newcommand{\Crit}{\mathrm{Crit}}

\newcommand{\Min}{\mathrm{Min}}
\newcommand{\Max}{\mathrm{Max}}
\newcommand{\Fdep}{\mathcal{F}_{\hspace{-0.3em}\raisebox{-0.2ex}{$\scriptstyle\rightarrow$}}}

\newcommand{\Fano}{\mathcal{F}_{\hspace{-0.05em}\raisebox{-0.2ex}{$\scriptstyle\Upsilon$}}}

\newcommand{\Finc}{\mathcal{F}_{\hspace{-0.05em}\raisebox{-0.1ex}{$\scriptstyle\subseteq$}}}

\newcommand{\Fexc}{\mathcal{F}_{\hspace{0.15em}\raisebox{-0.2ex}{$\scriptstyle \rule{0.5pt}{.5em}$\hspace{0.1em}}}}

\newcommand{\Find}{\mathcal{F}_{\hspace{-0.15em}\raisebox{-0.2ex}{$\scriptstyle \perp$}}}

\usepackage{amsthm}
\usepackage{cleveref}
\newtheorem{theorem}{Theorem}[section] 
\newtheorem{definition}[theorem]{Definition}
\newtheorem{proposition}[theorem]{Proposition}
\newtheorem{corollary}[theorem]{Corollary}
\newtheorem{lemma}[theorem]{Lemma}

\theoremstyle{definition}

\newtheorem{example}[theorem]{Example}

\newcommand{\glor}{\mathop{\raisebox{0.3ex}{\scalebox{0.75}{\rotatebox{4}{$\setminus$}}}\hspace{-2.5pt}\scalebox{1.2}{$\lor$}}}

\usepackage{oplotsymbl}
\let\singlemight=\trianglepbdot
\let\might=\trianglepb

\title{Expressibility and inexpressibility in propositional team logics}
\author{Matilda Häggblom\\ University of Helsinki\and Minna Hirvonen\\ Leibniz University Hannover \and Jouko Väänänen\\ University of Helsinki}
\date{}

\begin{document}

\maketitle

\begin{abstract}
  We develop dimension theoretic methods for propositional team based logics.  Such quantitative methods were defined for team based first-order logic in (Hella et al., 2024) and were used to obtain strong hierarchy results in the first-order logic context. We show that in propositional logic and in several important cases, a team theoretical atom can be expressed in terms of atoms of lower arity. We estimate the `price' of such a reduction of arity, i.e. how much more complicated the new expression is.
   Our estimates involve as parameters the arity of the atoms involved, as well as the number of times the atom occurs in a formula. We also consider new variants of atoms and propositional operations, inspired by our work. We believe that our quantitative analysis leads to a deeper understanding of the scope and limits of propositional team based logic.
\end{abstract}

\section{Introduction}
Expressibility is one of the central concepts of logic. The purpose of this paper is to examine expressibility in the framework of propositional team semantics (\cite{MR3488885}). We adapt to the context of propositional logic the combinatorial method of dimension computation from \cite{Hella_Luosto_Vaananen_2024}, where it is used to prove hierarchy results in the richer context of first-order team logics. We use this method to estimate how much more can be expressed with a team semantical atom of a given arity in comparison to expressibility with atoms of lower arity. Naturally, this depends on the logical operations allowed. A special feature of our results, compared to  \cite{Hella_Luosto_Vaananen_2024}, is that we factor into the computations also the number of occurrences of particular kinds of atoms in a complex formula.

The main result of \cite{Hella_Luosto_Vaananen_2024} is, roughly speaking, that non-trivial team theoretical atoms, such as the dependence atom, the inclusion atom, the independence atom, etc, are not definable in terms of atoms of smaller arity, provided that only the usual logical operations $\wedge,\vee,\exists,\forall$ are allowed. On the other hand, if intuitionistic implication $\to$ is allowed, then the dependence atom can be expressed in terms of the constancy atom \cite{MR2480819}:
\begin{equation}\label{BI}
\dep(x_1,\ldots,x_n;y)\ \equiv\ (\dep(x_1)\wedge\ldots\land\dep(x_n))\to\ \dep(y).
\end{equation}
The explanation of this is that the logical operation $\to$ increases dimension exponentially. We observe in this paper that in the case of propositional logic, all atoms that we consider have ``reduction formulas'' i.e. equivalent forms in terms of atoms of lower arity, as in (\ref{BI}). In these reduction formulas, only the usual logical operations $\wedge,\vee$ are used, and sometimes some modifications of them, such as propositional quantifiers. We then use dimension computations to show that the reduction formulas, while lowering arity, increase necessarily---typically exponentially---the length of formulas and the number of such atoms occurring in a formula.

 For propositional inclusion and exclusion atoms, we do not obtain as simple reduction formulas as for dependence and anonymity atoms; one reason could be 
that inclusion and exclusion are not expressive enough once added to propositional logic. 
Both inclusion and anonymity atoms are union closed, but out of the two logics only propositional anonymity logic is expressively complete for all union closed team properties with the empty team property \cite{yang2022}. Similarly, both exclusion and dependence atoms are downward closed, but only propositional dependence logic is expressively complete for all downward closed team properties with the empty team property \cite{MR3488885}. We therefore consider settings where expressive completeness is realised for more expressive variants of propositional inclusion/exclusion logic. One approach is to introduce \emph{relativized} inclusion/exclusion atoms, and another is to consider quantified propositional inclusion/exclusion logic. In the context of these logics, we introduce simple reduction formulas for the atoms.

 In this paper, we also give attention to (quasi) upward closed formulas, which were mostly absent from \cite{Hella_Luosto_Vaananen_2024} due to their less expressive role in the first-order setting \cite{Galliani2015}. One example of a quasi upward closed atom is the primitive inclusion atom from \cite{yang2022} of the form $\bar{x}\subseteq\bar{p}$, where $\bar{x}$ is a sequence of constants $\top,\bot$ and $\bar{p}$ is a sequence of propositional symbols. The primitive inclusion atom plays an important role in the normal form for (extended) propositional inclusion logic \cite{yang2022}. One example of an upward closed unary operator is the epistemic might operator $\blackdiamond$, where $\blackdiamond\phi$ is satisfied by a team if there is a nonempty subteam that satisfies $\phi$. The epistemic might operator has been used in the context of convex logics in \cite{anttila2025convexteamlogics}, and for reasoning about legal concepts in \cite{Hornung}.

The structure of this paper is the following: 

 In \Cref{section 2}, we review the basic dimension theoretical concepts used in this paper and prove also some helpful auxiliary results about these concepts.

In \Cref{section 3}, we prove preservation results for different dimensions under logical operations. Respectively, we calculate the different dimensions of relevant team semantical atoms.

In \Cref{section 4}, we enter the topic of representations of atoms in terms of the same, or in some cases different atoms, of lower arity. When such a representation is iterated, the arities go down until they reach 1, or, as for dependence and anonymity atoms, 0. At the same time, the representing formula grows, potentially exponentially in comparison to what we started with. In the case of the inclusion and exclusion atoms, we obtain a representation of this kind only by allowing extended atoms.

\Cref{section 5} contains our main results Theorems \ref{opti}-\ref{opti inc exc}. These results give representations for dependence, anonymity, inclusion and exclusion atoms and show that the representations are optimal in the sense that while the number of occurrences of atoms of a lower arity grows, this growth cannot be avoided. We obtain in some cases exact bounds on how much the number of occurrences must grow, and in other cases estimates that leave room for improvement.

\Cref{section 6} contains further results about dimension with applications to modifications of the received atoms.

\section{Dimension definitions and properties}
\label{section 2}
The purely combinatorial concept of \emph{dimension} was introduced for team properties, or more generally for arbitrary families of sets, in \cite{Hella_Luosto_Vaananen_2024}. In fact, three different dimension concepts, upper dimension, dual upper dimension and cylindrical dimension, were introduced with similar but distinct properties.  The behaviour of dimensions under operators that map families of sets to new families of sets is studied in detail in \cite{Hella_Luosto_Vaananen_2024} and applied to prove arity-based hierarchy results for logics with team semantics in the context of first-order logic, including extended atoms.

The idea of our kind of dimension goes back to \cite{ciardelli09} where the dimension of a downward closed family is defined as the smallest number of
power-sets that cover the family.  A similar dimension concept in modal logic was introduced
 in \cite{HLSV} and further used in \cite{HS}. It was generalized in \cite{DBLP:journals/lmcs/LuckV19} 
from downward closed families to arbitrary families in the context of
propositional logic. Both \cite{HLSV} and  \cite{HS} apply the dimension concept to proving lower bound results in modal logic.

There are also other dimension concepts in discrete
mathematics (e.g. matroid rank and  Vapnik–Chervonenkis-dimension). These otherwise widely used dimension-concepts do not seem to be appropriate for the purpose of this paper.  Still
another dimension is the length of a disjunctive normal form in
propositional logic. It is shown in \cite{Hella_Luosto_Vaananen_2024} that this is
equivalent to the cylindrical dimension in our sense.

In \Cref{subsec prel}, we recall properties of families of sets relevant to team semantics. In particular, the definitions of the upper dimension and dual upper dimension from \cite{Hella_Luosto_Vaananen_2024} are introduced. %\q{... and we present an alternative simplified (?) definition for finite families.}

In \Cref{subsec convex}, we observe that the dimensions of convex families correspond to a natural notion of the size of their smallest cover. This result for downward closed families is readily used in \cite{Hella_Luosto_Vaananen_2024}; we state it in a more general form for convex families. We extend this connection further to \emph{quasi} upward and downward closed families. In particular, the quasi upward closed families play an important role in multiple logics with team semantics, discussed in \Cref{subsec: trs quasi}.

\subsection{Preliminaries about dimension} \label{subsec prel}

We define some basic closure properties for families of sets, and recall the definitions from \cite{Hella_Luosto_Vaananen_2024} regarding upper and dual upper dimensions.

Let $\mathcal{X}$ be a base set with at least two elements and let $\mathcal{A}\subseteq P(\mathcal{X})$ be a family of sets. 
   Define the interval
$[A,B]:=\{C\mid A\subseteq C\subseteq B\}$.%    We say that a family of sets $\mathcal{A}$ is convex if $[S,T]\subseteq\mathcal{A}$ for all $S,T\in\mathcal{A}$.  

\begin{itemize}
     \item $\mathcal{A}$ is convex if $[A,B]\subseteq\mathcal{A}$ for all $A,B\in\mathcal{A}$.
     
     \item $\mathcal{A}$ is downward closed if for all $A\in\mathcal{A}$ and $B\subseteq A$, $B\in\mathcal{A}$. Equivalently, $\mathcal{A}$ is convex and contains the empty set.
     
    \item $\mathcal{A}$ is \emph{quasi} downward closed if $\mathcal{X}\in\mathcal{A}$ and $\mathcal{A}\setminus\{\mathcal{X}\}$ is downward closed.

     \item $\mathcal{A}$ is upward closed if for all $A\in\mathcal{A}$ with $A\subseteq B\subseteq\mathcal{X}$, $B\in\mathcal{A}$. Equivalently, $\mathcal{A}$ is convex and contains $\mathcal{X}$.
     
     \item $\mathcal{A}$ is \emph{quasi} upward closed if $\emptyset\in\mathcal{A}$ and $\mathcal{A}\setminus\{\emptyset\}$ is upward closed. 
     
      \item $\mathcal{A}$ is union closed if
for every subfamily $\mathcal{B}\subseteq\mathcal{A}$, we have $\bigcup\mathcal{B}\in\mathcal{A}$.
\end{itemize}

Note that a family $\mathcal{A}$ is both quasi (upward) downward closed and (upward) downward closed only when $\mathcal{A}=P(\mathcal{X})$. 
% \ToDo{Define the quasi versions as the `strong' quasi properties from the start. }We note that upward closure entails quasi upward closure, which in turn entails union closure. Similarly, downward closure implies quasi downward closure.  
%In applications, we will primarily concern ourselves with downward closure and quasi upward closure; however, we define each of their (non-)quasi versions to obtain symmetrical dimension results. 
To avoid any ambiguity in the results of this paper, we only say that a family $\mathcal{A}$ is quasi downward (upward) closed if it is not downward (upward) closed, i.e., $\mathcal{A}\subsetneq P(\mathcal{X})$ and $\mathcal{A}$ is quasi (upward) downward closed.

% \begin{itemize}
%     \item $\mathcal{X}\in\mathcal{A}\subsetneq P(\mathcal{X})$ and $\mathcal{A}$ is quasi downward closed,

%     \item  $\emptyset\in\mathcal{A}\subsetneq P(\mathcal{X})$ and $\mathcal{A}$ is quasi upward closed.
% \end{itemize}

The central concept of this paper is the upper and dual upper dimensions of families of sets, corresponding to the smallest \emph{dominating} and \emph{supporting} subfamily, respectively.  The dimensions capture a different kind of property of a family of sets than the closure properties discussed thus far.

\begin{definition}[\cite{Hella_Luosto_Vaananen_2024}]
Let $\mathcal{A}$ be a family of sets.
\begin{description}[align=left]
    \item[\hspace*{.2cm} \normalfont (Dominating subfamily)] 
       A subfamily $\mathcal{G}\subseteq\mathcal{A}$ dominates $\mathcal{A}$ if there exists 
    %dominated
    convex families $\mathcal{D}_{G},G\in\mathcal{G}$, such that $\bigcup_{G\in\mathcal{G}}\mathcal{D}_G=\mathcal{A}$ and $\bigcup\mathcal{D}_G=G\in\mathcal{D}_{G}$ for each $G\in\mathcal{G}$.

  \item[\hspace*{.2cm} \normalfont (Upper dimension)]   The upper dimension of a family $\mathcal{A}$ is $\text{D}(\mathcal{A}):=\min\{|\mathcal{G}| \mid \mathcal{G} \text{ dominates } \mathcal{A}\}$

    \item[\hspace*{.2cm} \normalfont (Supporting subfamily)]  Let $\mathcal{A}$ be a family of sets.  A subfamily $\mathcal{G}\subseteq\mathcal{A}$ supports $\mathcal{A}$ if there exists 
     %supported 
    convex families $\mathcal{S}_{G},G\in\mathcal{G}$, such that $\bigcup_{G\in\mathcal{G}}\mathcal{S}_G=\mathcal{A}$ and $\bigcap\mathcal{S}_G=G\in\mathcal{S}_{G}$ for each $G\in\mathcal{G}$.
    \item[\hspace*{.2cm} \normalfont (Dual upper dimension)]
        The dual upper dimension of a family $\mathcal{A}$ is $\text{D}^d(\mathcal{A}):=min\{|\mathcal{G}| \mid \mathcal{G} \text{ supports } \mathcal{A}\}$.
\end{description}
\end{definition}

A dominating subfamily is thus a collection of sets $\mathcal{G}$ from the family $\mathcal{A}$, such that for every $B\in\mathcal{A}$, there is some $G\in\mathcal{G}$ for which $[B,G]\subseteq \mathcal{A}$. Similarly, a supporting subfamily is a collection $\mathcal{G}\subseteq\mathcal{A}$, such that for every $B\in\mathcal{A}$, there is some $G\in\mathcal{G}$ for which $[G,B]\subseteq \mathcal{A}$. With this in mind, some auxiliary definitions are useful when calculating the dimensions.

\begin{definition}[\cite{Hella_Luosto_Vaananen_2024}]
    Let $\mathcal{A}$ be a family of sets and $A\in\mathcal{A}$.
    
    \begin{description}[align=left]
    \item[\hspace*{.2cm} \normalfont (Convex shadow)]     The convex shadow of $A$ in $\mathcal{A}$ is the family $\partial_A(\mathcal{A})=\{B\subseteq A\mid [B,A]\subseteq\mathcal{A}\}$.

\item[\hspace*{.2cm} \normalfont (Critical set)]  A set $A\in\mathcal{A}$ is called critical in $\mathcal{A}$ if its convex shadow is maximal in the family $\{\partial_B(\mathcal{A})\mid B\in\mathcal{A}\}$. Define the notation $\Crit(\mathcal{A})=\{A\in\mathcal{A}\mid A \text{ critical in }\mathcal{A}\}$. 

 \item[\hspace*{.2cm} \normalfont (Dual convex shadow)] The dual convex shadow of $A$ in $\mathcal{A}$ is the family $\partial^d_A(\mathcal{A})=\{B\supseteq A\mid [A,B]\subseteq\mathcal{A}\}$.
 
    \item[\hspace*{.2cm} \normalfont (Dual critical set)] A set $A\in\mathcal{A}$ is called \emph{dual critical} in $\mathcal{A}$ if its dual convex shadow is maximal in the family $\{\partial^d_B(\mathcal{A})\mid B\in\mathcal{A}\}$. Define the notation $\Crit^d(\mathcal{A})=\{A\in\mathcal{A}\mid A \text{ is dual critical in }\mathcal{A}\}$. 
 \end{description}
\end{definition}

The close relationship between the (dual) upper dimension and the (dual) shadows and (dual) critical sets is established next. In particular, \Cref{Zorn lemma} \cref{Zorn item def} provides an alternative definition for the (dual) dimensions for finite families.

\begin{lemma} \label[lemma]{Zorn lemma} 
Let $\mathcal{A}$ be a finite family.
\begin{enumerate}[label=(\roman*)]
    \item The upper dimension $\Dim(\mathcal{A})$ is the smallest size of a subfamily $\mathcal{G}\subseteq \mathcal{A}$ whose members' convex shadows cover $\mathcal{A}$, that is, $\{\partial_B(\mathcal{A})\mid B\in\mathcal{G}\}=\mathcal{A}$. Similarly, the dual upper dimension $\Dim^d(\mathcal{A})$ is the smallest size of a subfamily $\mathcal{H}\subseteq \mathcal{A}$ such that $\{\partial^d_B(\mathcal{A})\mid B\in\mathcal{H}\}=\mathcal{A}$.
    
    \label{Zorn item def}
    \item (\cite{Hella_Luosto_Vaananen_2024}) There is a a family $\mathcal{G}$ dominating $\mathcal{A}$ such $\mathcal{G}\subseteq\Crit(\mathcal{A})$ and $|\mathcal{G}|=\Dim(\mathcal{A})$. Similarly, there is a a family $\mathcal{H}$ supporting $\mathcal{A}$ such $\mathcal{H}\subseteq\Crit^d(\mathcal{A})$ and $|\mathcal{H}|=\Dim^d(\mathcal{A})$.  \label{Zorn item}
\end{enumerate} 
%\{Minna: I checked the proofs, and this seems correct to me. I think that the point of the other proof is to check that each $G\in\mathcal{G_0}$ is really included in some critical set. (Each union of a chain is included in the family, so it is not possible that such critical set does not exist.)}
\end{lemma}
\begin{proof}
The definition of $\Dim(\mathcal{A})$ is the smallest number of convex sets whose union covers the family $\mathcal{A}$. For finite families, we can w.l.o.g. assume that each such convex set is maximal, i.e., that it is a convex shadow, and in particular, a convex shadow of a critical set. The dual claims in both items are analogous.
\end{proof}

We remark that the preceding lemma does not apply to some infinite families for which critical sets are missing; consider any infinite base set $\mathcal{X}$ and the infinite family $\mathcal{A}=\{A\in \mathcal{X
}\mid |A| \text{ is finite}\}$, clearly $\Crit(\mathcal{A})$ is empty. See \cite{Hella_Luosto_Vaananen_2024} for a more detailed handling of dimensions for infinite families. In this paper, we consider only finite families, and henceforth assume that all families of sets are finite. 

\Cref{tabel:intuition} provides an intuition of the upper and dual upper dimensions, as well as the cylindrical dimension $\Dim^c$ from \cite{Hella_Luosto_Vaananen_2024}. We say that a nonempty interval $[A,B]\subseteq\mathcal{A}$ is maximal in $\mathcal{A}$ if for all $A'\subseteq A$ and $B'\supseteq B$, $[A',B']\subseteq\mathcal{A}$ implies that  $A'=A$ and $B'=B$. Now for a finite family $\mathcal{A}$, the cylindrical dimension $\Dim^c(\mathcal{A})$ is the smallest number of (maximal) intervals that cover the family. We will not focus on the cylindrical dimension further in this paper, but let us briefly recall some basic results from \cite{Hella_Luosto_Vaananen_2024}: $\Dim(\mathcal{A})\leq \Dim^c(\mathcal{A})$,  $\Dim^d(\mathcal{A})\leq \Dim^c(\mathcal{A})$ and if $\mathcal{A}$ is convex, then
$\Dim^c(\mathcal{A})\leq \Dim(\mathcal{A})\Dim^d(\mathcal{A})$.

%\minna{I think that this looks good now!}
% \q{Should we say something more about the cylindrical dimension?}\ToDo{Check if we should add the result in \cite{Hella_Luosto_Vaananen_2024} about dimension estimates based on `the other' dimensions.}
% \minna{I think that currently we don't have the def of cylindrical dim in this paper? Maybe we should add it and also the dimension estimete result? Is this prop 5 in \cite{Hella_Luosto_Vaananen_2024}, or was there also some other result?}

\begin{figure}[ht] 
    \centering

\noindent

\colorbox{gray!20}{%
\begin{minipage}{0.33\textwidth}

\begin{tikzpicture}
\foreach \row in {1,2,3} {
    \foreach \col in {1,2,3,4} {
      \node[inner sep=0pt, minimum size=0pt] (n\col\row) at (\col, -\row) {};
    }
  }

\node[]  at (.9, -.85) {$\Dim(\mathcal{A})=2$};

\node[draw, circle, thick, minimum size=6pt, inner sep=0pt] at (2,-.85) {};
\node[draw, circle, thick, minimum size=6pt, inner sep=0pt] at (3,-.85) {};
\draw[<->, thick, looseness=1, out=270, in=90] (n21) to (n13);
\draw[<->, thick, looseness=1, out=270, in=90] (n21) to (n33);
\draw[<->, thick, looseness=1, out=270, in=90] (n21) to (n23);
\draw[<->, thick, looseness=1, out=279, in=90] (n31) to (n43);
 \draw[<->, thick, looseness=1, out=270, in=90] (n31) to (n23);
\end{tikzpicture}

\end{minipage}}
\hfill
\begin{minipage}{0.30\textwidth}
\begin{tikzpicture}

  \foreach \row in {1,2,3} {
    \foreach \col in {1,2,3,4} {
      \node[inner sep=0pt, minimum size=0pt] (n\col\row) at (\col, -\row) {};
    }
  }

\node[]  at (.8, -2.8) {$\Dim^d(\mathcal{B})=2$};

\node[draw, circle,fill=black,  thick, minimum size=6pt, inner sep=0pt] at (2,-3.15) {};
\node[draw, circle, fill=black, thick, minimum size=6pt, inner sep=0pt] at (3,-3.15) {};
    \draw[<->, thick, looseness=1, out=90, in=270] (n23) to (n11);
\draw[<->, thick, looseness=1, out=90, in=270] (n23) to (n31);
\draw[<->, thick, looseness=1, out=90, in=270] (n23) to (n21);
\draw[<->, thick, looseness=1, out=90, in=270] (n33) to (n41);
 \draw[<->, thick, looseness=1, out=90, in=270] (n33) to (n21);
\end{tikzpicture}
\end{minipage}
\hfill
\colorbox{gray!20}{%
\begin{minipage}{0.27\textwidth}
\begin{tikzpicture}

  \foreach \row in {1,2,3} {
    \foreach \col in {1,2,3,4} {
      \node[inner sep=0pt, minimum size=0pt] (n\col\row) at (\col, -\row) {};
    }
  }

\node[]   at (.9, -.85)  {$\Dim^c(\mathcal{C})=3$};
 \node[] (A) at (2.9, -.9) {};
   \node[] (B) at (2.1, -3.1) {};
\node[] (C) at (3, -2.3) {};
\node[] (D) at (1.1, -1.4) {};
  \draw[<->, thick] (D) to (C);
\draw[<->, thick] (n21) to (n23);
\draw[<->, thick] (A) to (B);

\end{tikzpicture}\end{minipage}}
   \caption{%An intuitive picture for the upper and dual upper dimensions, together with the cylindrical dimension $\Dim^c$ from \cite{Hella_Luosto_Vaananen_2024}.
   The double-edged arrows represent maximal intervals, the white circles critical sets, and the filled-in circles dual critical sets. Each dimension is concerned with the smallest number of some particular convex sets that cover a family; the upper dimension $\Dim$ with critical sets, the dual upper dimension $\Dim^d$ with dual critical sets, and the cylindrical dimension $\Dim^c$ with maximal intervals.}
    \label{tabel:intuition}
\end{figure}
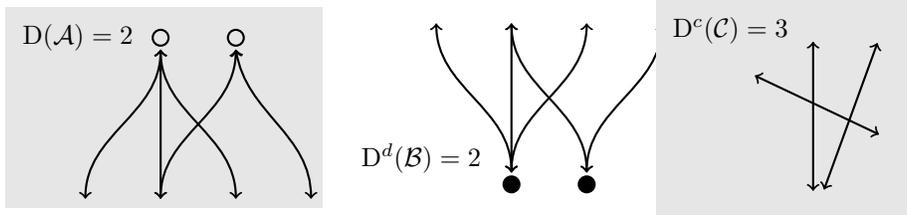

We provide concrete examples of downward, union and quasi upward closed families and their upper and dual upper dimensions in \Cref{tab:four ex}, together with an example showing that there are families for which the number of (dual) critical sets is strictly greater than the (dual) upper dimension, meaning that the subset relation in \Cref{Zorn lemma} \cref{Zorn item} can be strict.%\minna{Should this be other way around, that is, strictly greater?}\matilda{Yes, fixed.}

\begin{example} \label[example]{example1}
    Consider the family $$\mathcal{D}=\{\{c\}, \{a,c\},  \{c,d\}, \{a,b,c\}\}.$$ Now $\Crit\{\mathcal{D}\}=\{\{a,c\}, \{c,d\},\{a,b,c\}\}$, 
since the convex shadows $\partial_{\{a,b,c\}}(\mathcal{D})=\{ \{a,c\},\{a,b,c\}\}$, $\partial_{\{a,c\}}(\mathcal{D})=\{\{c\},\{a,c\} \}$ and $\partial_{\{c,d\}}(\mathcal{D})=\{\{c\},\{c,d\}\}$ are pairwise incomparable. 
We note that $\partial_{\{a,c\}}(\mathcal{D})\subseteq\partial_{\{a,b,c\}}(\mathcal{D})\cup\partial_{\{c,d\}}(\mathcal{D})=\mathcal{D}$, hence $2=\Dim(\mathcal{D})\lneq |\Crit\{\mathcal{D}\}|=3$. The family is illustrated in \Cref{tab:four ex}.
\end{example}  

\begin{figure}[ht]
\begin{center}
\begin{tabular}{cc}
% Top-left cell
\cellcolor{gray!10}
\begin{minipage}{0.45\textwidth}
\centering

\begin{tikzpicture}

  \foreach \row in {1,2,3} {
    \foreach \col in {1,2,3} {
      \node[inner sep=0pt, minimum size=0pt] (n\col\row) at (\col, -\row) {};
    }
  }

  \node[]  at (0, -.8) {$\mathcal{A}:$};
\node[]  at (1, -0.8) {$\{a,b\}$};
\node[] at (3, -0.8) {$\{b,c\}$};
\node[] at (.85, -1.65) {$\{a\}$};
\node[] at (3.15, -1.65) {$\{c\}$};
\node[] at (2, -2.1) {$\{b\}$};
\node[] at (2, -3.2) {$\emptyset$};

\draw[<->, thick, looseness=1, out=270, in=90] (n11) to (n23);
\draw[<->, thick, looseness=1, out=270, in=90] (n31) to (n23);
\end{tikzpicture}
\end{minipage}
&
% Top-right cell
\begin{minipage}{0.45\textwidth}
\centering
\begin{tikzpicture}

  \foreach \row in {1,2,3} {
    \foreach \col in {1,2,3} {
      \node[inner sep=5pt, minimum size=10pt] (n\col\row) at (\col, -\row) {};
    }
  }

\node[]  at (-.3, -.75) {$\mathcal{B}:$};
\node[]  at (1, -.9) {$\{a,b,c,d\}$};
\node[]  at (1, -2) {$\{a,b,c\}$};
\node[]  at (1, -3.1) {$\{a,b\}$};
\node[] at (2, -2) {$\{c\}$};
\node[] at (2, -3.1) {$\emptyset$};

     \draw[<->, thick] (n11) --  (n12);
      \draw[<->, thick] (n12) --  (n13);
     \draw[<->, thick] (n22) --  (n23);

\end{tikzpicture}

\end{minipage}
\\[3em]
% Bottom-left cell
\begin{minipage}{0.45\textwidth}
\centering
\begin{tikzpicture}

  \foreach \row in {1,2,3} {
    \foreach \col in {1,2,3} {
      \node[inner sep=0pt, minimum size=0pt] (n\col\row) at (\col, -\row) {};
    }
  }

\node[]  at (-.5, -.75) {$\quad\quad\mathcal{C}:$};
\node[]  at (2, -0.8) {$\{a,b,c,d\}$};
\node[]  at (2, -2.3) {$\{a,b,c\}$};
\node[]  at (1, -1.7) {$\{a,b,d\}$};
\node[] at (3, -1.7) {$\{b,c,d\}$};
\node[] at (1, -3.2) {$\{a,b\}$};
\node[] at (3, -3.2) {$\{b,c\}$};
 \fill (4, -2.95) circle (1pt);
\node[] at (4, -3.2) {$\emptyset$};

    \draw[<->, thick, looseness=1, out=270, in=90] (n21) to (n13);
\draw[<->, thick, looseness=1, out=270, in=90] (n21) to (n33);

  \node[] (A) at (4, -2.95) {};
 \draw[<->, thick, looseness=6, out=115, in=15] (A) to (A);

\end{tikzpicture}
\end{minipage}
&
% Bottom-right cell
\cellcolor{gray!10}
\begin{minipage}{0.45\textwidth}
\centering
\begin{tikzpicture}

  \foreach \row in {1,2,3} {
    \foreach \col in {1,2,3} {
      \node[inner sep=5pt, minimum size=10pt] (n\col\row) at (\col, -\row) {};
    }
  }

\node[]  at (-.3, -.75) {$\mathcal{D}:$};
\node[]  at (1, -.9) {$\{a,b,c\}$};
\node[]  at (1, -2) {$\{a,c\}$};
\node[]  at (1, -3.1) {$\{c\}$};
\node[] at (2, -2) {$\{c,d\}$};

     \draw[<->, thick] (n11) --  (n12);
      \draw[<->, thick] (n12) --  (n13);
     \draw[<->, thick] (n22) --  (n13);

\end{tikzpicture}
\end{minipage}
\end{tabular}
\end{center}

    \caption{Let $\mathcal{X}=\{a,b,c,d\}$ be a base set. The four families are illustrated with their maximal intervals shown by a double-edged arrow.  The family $\mathcal{A}:=\{\emptyset,\{a\},\{b\}, \{c\}, \{a,b\}, \{b,c\}\}$ is downward closed with 
$\Dim(\mathcal{A})=2$ and $\Dim^d({A})=1$. The family 
$\mathcal{B}:=\{\emptyset,\{c\}, \{a,b\},\{a,b,c\}, \{a,b,c,d\}\}$ is union closed with 
$\Dim(\mathcal{B})=3=\Dim^d(\mathcal{B})$. The family $\mathcal{C}:=\{\emptyset,\{a,b\},\{b,c\},\{a,b,c\}, \{a,b,d\},\{b,c,d\}, \{a,b,c,d\}\}$ is quasi upward closed with
$\Dim(\mathcal{C})=2$ and $\Dim^d(\mathcal{C})=3$. Lastly, the family $\mathcal{D}$ is from \Cref{example1}. What is its dual upper dimension $\Dim^d(\mathcal{D})$?}
    \label{tab:four ex}
\end{figure}
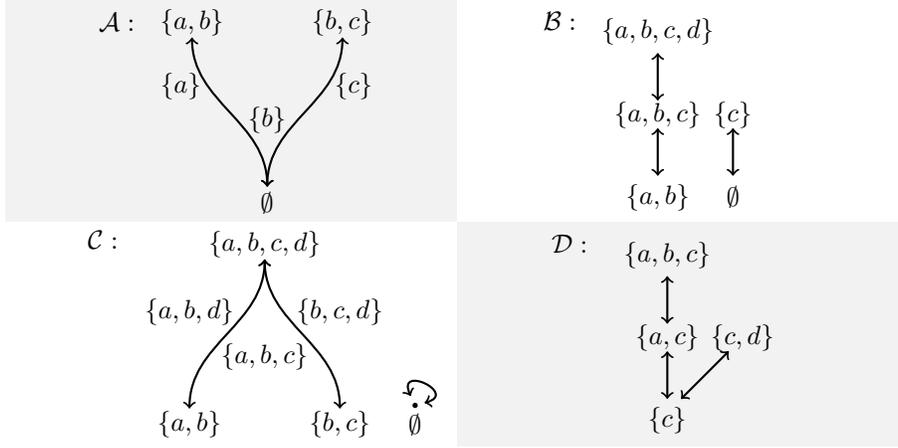

We highlight that the choice of dimension can trivialize the dimension calculations for certain classes of families. The following proposition shows that if a family $\mathcal{A}$ is either downward or quasi downward closed, then we can determine which one it is by its dual upper dimension, with an analogous argument for (quasi) upward closed families using the upper dimension. 

\begin{proposition}\label[proposition]{trivial dim} Let $\mathcal{A}\subseteq P(\mathcal{X})$ be a nonempty %\textcolor{red}{finite} 
family. 
\begin{enumerate} [label=(\roman*)]
    \item  If $\mathcal{A}$ is downward closed, then the dual upper dimension is $\Dim^d(\mathcal{A})=1$. If $\mathcal{A}$ is upward closed, then the upper dimension is $\Dim(\mathcal{A})=1$.\label{trivial dim strong}
    \item Assume additionally that $\mathcal{A}\subsetneq P(\mathcal{X})$ and $|\mathcal{A}|\geq 2$. If $\mathcal{A}$ is quasi downward closed, then the dual upper dimension is $\Dim^d(\mathcal{A})= 2$. 
 If $\mathcal{A}$ is quasi upward closed, then the upper dimension is $\Dim(\mathcal{A})=2$.\label{trivial dim quasi}
\end{enumerate}
\end{proposition}

\begin{proof}
We obtain the results using \Cref{Zorn lemma} \cref{Zorn item}.
\begin{enumerate}[label=(\roman*)]
    \item Suppose that $\mathcal{A}$ is downward closed, then $\Crit^d(\mathcal{A})=\{\emptyset\}$ since $\partial^d_\emptyset(\mathcal{A})=\{A\in\mathcal{A}\mid [\emptyset, A]\subseteq \mathcal{A}\}=\mathcal{A}$, and we conclude $\Dim(\mathcal{A})=1$. Similarly, if $\mathcal{A}$ is upward closed, $\Crit(\mathcal{A})=\{\mathcal{X}\}$, hence $\Dim(\mathcal{A})=1$.

    \item If $\mathcal{A}\subsetneq P(\mathcal{X})$ is quasi downward closed, then $\partial_\emptyset(\mathcal{A})=\{A\in\mathcal{A}\mid [\emptyset, A]\subseteq \mathcal{A}\}=\mathcal{A}\setminus\{\mathcal{X}\}$ and $\{\mathcal{X}\}\subseteq\partial_\mathcal{X}(\mathcal{A})$. Neither dual convex shadow is contained in the other, hence $\Crit^d(\mathcal{A})=\{\emptyset, \mathcal{X}\}$ and we conclude $\Dim^d(\mathcal{A})=2$.

    The case when $\mathcal{A}$ is quasi upward closed is similar since $\Crit(\mathcal{A})=\{\emptyset, \mathcal{X}\}$ due to $\partial_\mathcal{X}(\mathcal{A})=\{A\in\mathcal{A}\mid [A, \mathcal{X}]\subseteq \mathcal{A}\}=\mathcal{A}\setminus\{\emptyset\}$ by upward closure and $\{\emptyset\}\subseteq\partial_\emptyset(\mathcal{A})$, with neither convex shadow contained in the other. Hence $\Dim(\mathcal{A})=2$.
\end{enumerate}
\end{proof}

In light of the preceding proposition, we mainly focus on the upper dimension for (quasi) downward closed families, and the dual upper dimension for (quasi) upward closed families. 

\subsection{Dimensions for convex and quasi convex families} \label{subsec convex}

By definition, the upper dimension for (quasi) downward closed families has symmetrical results to the dual upper dimension for (quasi) upward closed families. We exploit this symmetry and obtain a natural correspondence between the families' dimensions and number of maximal/minimal sets. %using $\Min$ and $\Max$, respectively. 

%The following sets will be useful when calculating the dimensions. 

 We first consider downward and upward closed families, recalling that both are convex. 
We establish a relationship between the sets $\Max(\mathcal{A})$ and $\Min(\mathcal{A})$ with dominating and supporting subfamilies of $\mathcal{A}$, where 
\begin{itemize}
     \item $\Min(\mathcal{A}):=\{A\in\mathcal{A}\mid \text{ there is no } B\in\mathcal{A}\text{ such that } B\subsetneq A\},$
    % \item $\Min^q(\mathcal{A}):=\{A\in\mathcal{A}\mid \text{ there is no nonempty } B\in\mathcal{A}\text{ such that } B\subsetneq A\}$.

     \item $\Max(\mathcal{A}):=\{A\in\mathcal{A}\mid \text{there is no } B\in\mathcal{A}\text{ such that } A\subsetneq B\}$.
     % \item $\Max^q(\mathcal{A}):=\{A\in\mathcal{A}\mid \text{there is no } \mathcal{X}\neq B\in\mathcal{A}\text{ such that } A\subsetneq B\}$, where $\mathcal{X}$ is the base set. 
 \end{itemize}
 We then refine the result by making a connection to the (dual) critical sets. Finally, we conclude that the upper dimension and dual upper dimension of convex families can be calculated using the size of $\Max(\mathcal{A})$ and $\Min(\mathcal{A})$, respectively.

\begin{lemma} \label[lemma]{minmax and crit dim lemma} 
    Let $\mathcal{A}$ be a family of sets.

    \begin{enumerate}   [label=(\roman*)]
    \item \label{MinMax item1}
    If $\mathcal{G}$ dominates $\mathcal{A}$, then $\Max(\mathcal{A})\subseteq\mathcal{G}$, and if $\mathcal{H}$ supports $\mathcal{A}$, then $\Min(\mathcal{A})\subseteq\mathcal{H}$. 
    \item $\Max(\mathcal{A})\subseteq\Crit(\mathcal{A})$ and $\Min(\mathcal{A})\subseteq\Crit^d(\mathcal{A})$ for any family $\mathcal{A}$.
    \item   If $\mathcal{A}$ is a convex family, then $\Crit(\mathcal{A})\subseteq\Max(\mathcal{A})$ and $\Crit^d (\mathcal{A})\subseteq\Min(\mathcal{A})$.
    \item      If $\mathcal{A}$ is a %\textcolor{red}{finite} and 
    convex family, then $\Dim(\mathcal{A})=|\Max(\mathcal{A})|$ and $\Dim^d(\mathcal{A})=|\Min(\mathcal{A})|$.
\end{enumerate}

\end{lemma}
\begin{proof}
We prove the first claim of each item. The second is symmetrical.
\begin{enumerate}[label=(\roman*)]
\item Let $\mathcal{G}\subseteq\mathcal{A}$ be such that it dominates $\mathcal{A}$. 
%We show that then $\Max(\mathcal{A})\subseteq\mathcal{G}$. 
Since $\Max(\mathcal{A})\subseteq\mathcal{A}=\bigcup_{G\in\mathcal{G}}\mathcal{D}_{G}$ for some convex families $\mathcal{D}_{G},G\in\mathcal{G}$, we have that for every $A\in\Max(\mathcal{A})$ there is some $G\in\mathcal{G}$ such that $A\in\mathcal{D}_{G}$. 
Then $A\subseteq\bigcup\mathcal{D}_{G}=G$, and since $A$ is maximal in $\mathcal{A}$, we have $A=G$. Thus $A\in\mathcal{G}$, and $\Max(\mathcal{A})\subseteq\mathcal{G}$.  
    \item     Let $G\in\Max(\mathcal{A})$ and $G'\in\mathcal{A}$ be such that $\partial_G(\mathcal{A})\subseteq\partial_{G'}(\mathcal{A})$. Then $G\in\partial_G(\mathcal{A})\subseteq\partial_{G'}(\mathcal{A})$, so $G\subseteq G'$. Since $G$ is maximal, $G=G'$ and $\partial_G(\mathcal{A})=\partial_{G'}(\mathcal{A})$, implying that $G\in\Crit(\mathcal{A})$.

    \item  %Always $\Max(\mathcal{A})\subseteq\Crit(\mathcal{A})$, so it suffices to show the other inclusion. 
    Let $G\in\Crit(\mathcal{A})$ and $G'\in\mathcal{A}$ be such that $G\subseteq G'$. Suppose that $A\in\partial_G(\mathcal{A})$. Then by the definition of convex shadow, $[A,G]\subseteq\mathcal{A}$. 
By the convexity of $\mathcal{A}$, $[G,G']\subseteq\mathcal{A}$, so it follows that $[A,G']\subseteq\mathcal{A}$. Hence, $A\in\partial_{G'}(\mathcal{A})$ and $\partial_G(\mathcal{A})\subseteq\partial_{G'}(\mathcal{A})$. Since $G$ is critical, it must be that $G=G'$, and therefore $G$ is maximal in $\mathcal{A}$, i.e.,  $G\in\Max(\mathcal{A})$.

        \item 
   % Since $\mathcal{A}$ is \textcolor{red}{finite}, 
   By \Cref{Zorn lemma} \cref{Zorn item}, there is a family $\mathcal{G}$  dominating $\mathcal{A}$ such that $\mathcal{G}\subseteq\Crit(\mathcal{A})$ and $|\mathcal{G}|=\Dim(\mathcal{A})$.
    By \Cref{minmax and crit dim lemma} \cref{MinMax item1}, we have $\Max(\mathcal{A})\subseteq\mathcal{G}$, so $\Max(\mathcal{A})\subseteq\mathcal{G}\subseteq\Crit(\mathcal{A})$. Since $\mathcal{A}$ is convex, by the previous item, we have $\Crit(\mathcal{A})\subseteq\Max(\mathcal{A})$, from which the claim follows.

\end{enumerate}
\end{proof}

The next proposition is now immediate. 

\begin{proposition} \label[proposition]{dc and uc prop}
    If $\mathcal{A}$ is a downward closed family, then $\Dim(\mathcal{A})=|\Max(\mathcal{A})|$; and if $\mathcal{A}$ is an upward closed family, then $\Dim^d(\mathcal{A})=|\Min(\mathcal{A})|$.
\end{proposition}

We now consider the upper dimension for quasi downward closed families with symmetrical results to the dual upper dimension for quasi upward closed families. Both families are \emph{quasi convex}, i.e., convex modulo the empty set and the base set. Furthermore, recall that we assume that quasi downward (upward) closed families are not downward (upward) closed. Therefore, there is no overlap with the previous results in this section. %, using $\Max^q$ and $\Min^q$, respectively. 

%\matilda{If we want to write the results in terms of quasi convexity, we still have to split the cases according to if it's modulo the empty set or the base set, so I think it's better to keep the results as they are and just have this small introduction to quasi convexity.}

%Notice that the quasi properties do not imply convexity. 

We define quasi variants of the sets $\Min(\mathcal{A})$ and $\Max(\mathcal{A})$.

\begin{itemize}
     % \item $\Min(\mathcal{A}):=\{A\in\mathcal{A}\mid \text{ there is no } B\in\mathcal{A}\text{ such that } B\subsetneq A\}$.
    \item $\Min^q(\mathcal{A}):=\Min(\mathcal{A}\setminus\{\emptyset\})\cup\{\emptyset\}$,
    
     % \item $\Max(\mathcal{A}):=\{A\in\mathcal{A}\mid \text{there is no } B\in\mathcal{A}\text{ such that } A\subsetneq B\}$.
     \item $\Max^q(\mathcal{A}):=\Max(\mathcal{A}\setminus\{\mathcal{X}\})\cup\{\mathcal{X}\}$.%, where $\mathcal{X}$ is the base set.   
 \end{itemize}

 We note the difference between $\Min$ and $\Min^q$. 
Consider the base set $\mathcal{X}=\{a,b,c\}$. Then the family $\mathcal{A}\subseteq P(\mathcal{X})$,  $$\mathcal{A}=\{\emptyset,\{a\},\{a,b\},\{a,c\},\{a,b,c\}\}$$ is quasi upward closed with $\Min(\mathcal{A})=\{\emptyset\}$ and $\Min^q(\mathcal{A})=\{\emptyset, \{a\}\}$. Similar examples can be constructed for quasi upward closed families comparing $\Max$ and $\Max^q$.   

We now obtain a close correspondence between $\Max^q$ and dual critical sets of quasi downward closed families, and similarly, between $\Min^q$ and dual critical sets of quasi upward closed families.

%\subsection{Quasi convex family properties} 

\label{subsec quasi prop}

\begin{lemma} \label[lemma]{W minmax and crit dim lemma} 
Let $\mathcal{A}\subsetneq P(\mathcal{X})$ be a 
%\textcolor{red}{finite} 
quasi downward closed family, and let $\mathcal{B}\subsetneq P(\mathcal{X})$ be a 
%\textcolor{red}{finite} 
quasi upward closed family. Then $\Max^q(\mathcal{A})=\Crit(\mathcal{A})$ and $\Min^q(\mathcal{B})=\Crit^d(\mathcal{B})$.
\end{lemma}
\begin{proof}
We show first that $\Min^q(\mathcal{B})\subseteq\Crit^d(\mathcal{B})$.
    Let $B\in\Min^q(\mathcal{B})$ and $H\in\mathcal{B}$ such that $\partial_B^d(\mathcal{B})\subseteq\partial_{H}^d(\mathcal{B})$. Then $B\in\partial_B^d(\mathcal{B})\subseteq\partial_{H}^d(\mathcal{B})$, so $H\subseteq B$. Since $B$ is minimal, $B=H$ or $H=\emptyset$.     
    If $B=H$, then $\partial_B^d(\mathcal{B})=\partial_{H}^d(\mathcal{B})$, implying that $B\in\Crit^d(\mathcal{B})$.
    %The case of  $H=\emptyset$ and $B=\emptyset$ is similar. 
    Suppose that $H=\emptyset$ and $B\neq\emptyset$. By upward closure, $\mathcal{X}\in \partial_B^d(\mathcal{B})$. Now, $\mathcal{X}\in\partial_B^d(\mathcal{B})\subseteq\partial_{\emptyset}^d(\mathcal{B})$, hence $\mathcal{B}=P(\mathcal{X})$, a contradiction.  
%\q{I commented out the text "The case of  $H'=\emptyset$ and $B=\emptyset$ is similar", because that seems to follow already from the case $B=H'$? Should we change $H'$ to $H$ since there is no $H$?}\matilda{Correct, and changed.}
  
Let us now show that $\Crit^d (\mathcal{B})\subseteq\Min^q(\mathcal{B})$.    Let $G,G'\in\mathcal{B}$ be nonempty sets such that $G\in\Crit^d(\mathcal{B})$ and $G'\subseteq G$. It suffices to show that $G'=G$. 
    Suppose that $B\in\partial^d_G(\mathcal{B})$.
By the upward closure of $\mathcal{B}$, 
$[G',B]\subseteq\mathcal{B}$. 
Hence, $B\in\partial^d_{G'}(\mathcal{B})$ and $\partial^d_G(\mathcal{B})\subseteq\partial^d_{G'}(\mathcal{B})$. Since $G$ is dual critical, it must be that $G=G'$. We thus conclude that $\Min^q(\mathcal{B})=\Crit^d(\mathcal{B})$; the proof of  $\Max^q(\mathcal{A})=\Crit(\mathcal{A})$ is symmetrical.
\end{proof}

Consequently, we obtain a useful way to calculate the upper dimension for quasi downward closed %\textcolor{red}{finite} 
families and the dual upper dimension for quasi upward closed %\textcolor{red}{finite} 
families.
 
\begin{proposition}
\label[proposition]{quasi dc and uc prop}
Let $\mathcal{A}\subsetneq P(\mathcal{X})$ be a 
%\textcolor{red}{finite} 
quasi downward closed family, and let $\mathcal{B}\subsetneq P(\mathcal{X})$ be a %\textcolor{red}{finite} 
quasi upward closed family. Then  $\Dim(\mathcal{A})=|\Max^q(\mathcal{A})|$ and 
$\Dim^d(\mathcal{B})=|\Min^q(\mathcal{B})|$. 

\end{proposition}
\begin{proof}
      % Since $\mathcal{B}$ is \textcolor{red}{finite}, 
      By \Cref{Zorn lemma} \cref{Zorn item}, there is a family $\mathcal{H}$  supporting $\mathcal{B}$ such that $\mathcal{H}\subseteq\Crit^d(\mathcal{B})$ and $|\mathcal{H}|=\Dim^d(\mathcal{B})$.
By \Cref{W minmax and crit dim lemma}, we have that $\Crit^d (\mathcal{B})=\Min^q(\mathcal{B})$, hence 
$\mathcal{H}\subseteq \Crit^d(\mathcal{B})=\Min^q(\mathcal{B})$.
Suppose for a contradiction that there is  $B\in\Min^q(\mathcal{B})$ such that $B\not\in\mathcal{H}$.
Now for $H\in\mathcal{H}$ and quasi minimality of $B$, $B\in \partial^d_H(\mathcal{B})$ implies that $H=\emptyset$. Then $[\emptyset, B]\subseteq\mathcal{B}$, hence quasi minimality of $B$ implies that $B=\{b\}$ for some $b\in \mathcal{X}$. 
By the assumption $\mathcal{B}\subsetneq P(\mathcal{X})$ and quasi upward closure, there is another singleton $A=\{a\}\in \mathcal{X}$ such that $\{a\}\not\in\mathcal{B}$. Notice that $\{a,b\}\in\mathcal{B}$ by quasi upward closure, and the only $H'$ for which $\{a,b\}\in\partial^d_{H'}(\mathcal{B})$ is $H'=\{a,b\}$. %But $\{b\}\cup\{a\}\not\in\Min^q(\mathcal{B})$, thus $\mathcal{H}$ does not support $\mathcal{B}$, a contradiction.
Clearly, $\{a,b\}\not\in\Min^q(\mathcal{B})$ because $\{b\}\in\Min^q(\mathcal{B})$. Since $\mathcal{H}\subseteq\Min^q(\mathcal{B})$, we also have $\{a,b\}\not\in\mathcal{H}$, which means that $\mathcal{H}$ cannot support $\mathcal{B}$, a contradiction. %\minna{Would it be better if the last sentence were: "But $\{b\}\cup\{a\}\not\in\Min^q(\mathcal{B})$, and since $\mathcal{H}\subseteq\Min^q(\mathcal{B})$, $\mathcal{H}$ does not support $\mathcal{B}$, a contradiction." or something similar? I already managed to forget that $\mathcal{H}$ was chosen such that $\mathcal{H}\subseteq\Min^q(\mathcal{B})$, so I was confused by this proof and the following example that seemed to contradict each other. (But they do not, because of the way $\mathcal{H}$ was chosen.)}\matilda{I had the exact same thought yesterday, yes please change it. I'll try to clarify this in the text before the example, too.}\minna{Changed it, and also changed $\{b\}\cup\{a\}$ to $\{a,b\}$.}

The proof of $\Dim(\mathcal{A})=|\Max^q(\mathcal{A})|$ is analogous; find instead $A'= \mathcal{X}\setminus\{a\} \in\Max^q(\mathcal{A})$ and assume for a contradiction that it is not in the dominating family, and $B'=\mathcal{X}\setminus\{b\} \not\in \mathcal{A}$ such that their intersection $A'\cap B'$ is in $\mathcal{A}$ by quasi downward closure but not in any of the quasi maximal sets' convex shadows.
\end{proof}

Note that the result of \Cref{quasi dc and uc prop} is obtained without the corresponding results to \Cref{minmax and crit dim lemma} \cref{MinMax item1}.  For a quasi downward closed family $\mathcal{A}$ and quasi upward closed family $\mathcal{B}$, the proof of \Cref{quasi dc and uc prop} shows that there exists a subfamily $\mathcal{G}$ that dominates $\mathcal{A}$ and a subfamily $\mathcal{H}$ that supports $\mathcal{B}$, such that $\Max^q(\mathcal{A})=\mathcal{G}$ and $\Min^q(\mathcal{B})=\mathcal{H}$. However, there can also be other suitable choices of dominating and supporting subfamilies $\mathcal{G'}$ and $\mathcal{H'}$ for which 
%
%In fact, for a subfamily $\mathcal{G}$ that dominates $\mathcal{A}$ and subfamily $\mathcal{H}$ that supports $\mathcal{B}$, 
neither $\Max^q(\mathcal{A})\subseteq\mathcal{G'}$ nor $\Min^q(\mathcal{B})\subseteq\mathcal{H'}$ hold, exemplified next.
%\q{Are there some words missing from the first sentence of this paragraph or am I just somehow misreading it? Also doesn't the last sentence contradict Lemma 2.8? Should it be neither $\Max^q(\mathcal{A})\subseteq\mathcal{G}$ nor $\Min^q(\mathcal{B})\subseteq\mathcal{H}$ for a dominating family $\mathcal{G}$ and a supporting family $\mathcal{H}$ instead?}\matilda{The referenced item was correct but not the text, fixed now.}

\begin{example} \label[example]{example quasi min fail}
Let $\mathcal{B}\subseteq P(\{a,b\})$ be a family defined by
$$\mathcal{B}:=\{\emptyset, \{b\}, \{a,b\}\}.$$
We note that $\mathcal{B}$ is both quasi downward and quasi upward closed. 
Now $\mathcal{H}:=\{\emptyset, \{a,b\}\}$ is both a dominating and a supporting family to $\mathcal{B}$, since $\partial_\emptyset(\mathcal{B})\cup \partial_{\{a,b\}}(\mathcal{B})=\{\emptyset\}\cup \{\{b\} , \{a,b\}\}=\mathcal{B}$ and $\partial^d_\emptyset(\mathcal{B})\cup \partial^d_{\{a,b\}}(\mathcal{B})=\{\emptyset, \{b\}\}\cup \{\{a,b\}\}=\mathcal{B}$. However, $\Max^q(\mathcal{B})=\{\{b\},\{a,b\}\}$ and $\Min^q(\mathcal{B})=\{\emptyset, \{b\}\}$ hence $\Max^q(\mathcal{B})\not\subseteq\mathcal{H}$ and
$\Min^q(\mathcal{B})\not\subseteq\mathcal{H}$. Nonetheless, $D(\mathcal{A})=2=|\Max^q(\mathcal{B})|$ and $D^d(\mathcal{A})=2=|\Min^q(\mathcal{B})|$.
\end{example}
%\q{Should we have $\partial^d_\emptyset(\mathcal{B})\cup \partial^d_{\{a,b\}}(\mathcal{B})=\{\emptyset, \{b\}\}\cup \{\{a,b\}\}=\mathcal{B}$ instead of $\partial^d_\emptyset(\mathcal{B})\cup \partial^d_{\{a,b\}}(\mathcal{B})=\{\emptyset, \{b\}\}\cup \{\{b\}, \{a,b\}\}=\mathcal{B}$, because isn't $\partial^d_{\{a,b\}}(\mathcal{B})=\{\{a,b\}\}$?}\matilda{Yes, corrected now and quasi downward closed case added.}

We remark that \Cref{dc and uc prop} and \Cref{quasi dc and uc prop} show a natural correspondence between the dimensions and (quasi) downward and (quasi) upward closed families. Let $\mathcal{B}$ be a family of sets and define the downset $\downarrow\mathcal{B}=\{b\mid b\subseteq b' \text{ for some } b'\in\mathcal{B}\}$, or equivalently $\downarrow\mathcal{B}=\bigcup_{b\in\mathcal{B}}\mathcal{P}(b)$, and upset $\uparrow\mathcal{B}=\{b\subseteq\mathcal{X}\mid b'\subseteq b\text{ for some }b'\in\mathcal{B}\}$. For any downward closed family $\mathcal{A}$ with $\Max(\mathcal{A})=\{a_1,a_2,\dots, a_n\}$, we have that $\downarrow\{a_1,a_2,\dots, a_n\}=\mathcal{A}$ and $\Dim(\mathcal{A})=n$. In this sense, the upper dimension naturally captures the size of the smallest cover of a downward closed family. The case for upsets of upward closed families and their dual upper dimension is symmetrical;  For any upward closed family $\mathcal{A}$ with $\Max(\mathcal{A})=\{a_1,a_2,\dots, a_n\}$, we have that $\uparrow\{a_1,a_2,\dots, a_n\}=\mathcal{A}$ and $\Dim^d(\mathcal{A})=n$. 
Furthermore, for any quasi downward closed family $\mathcal{A}$ with $\Max^q(\mathcal{A})=\{a_1,a_2,\dots, a_n, \mathcal{X}\}$, we have that $\downarrow\{a_1,a_2,\dots, a_n\}\cup\{\mathcal{X}\}=\mathcal{A}$ and $\Dim(\mathcal{A})=n+1$. For any upward closed family $\mathcal{A}$ with $\Max^q(\mathcal{A})=\{a_1,a_2,\dots, a_n, \emptyset\}$, we have that $\uparrow\{a_1,a_2,\dots, a_n\}\cup\{\emptyset\}=\mathcal{A}$ and $\Dim^d(\mathcal{A})=n+1$.

\section{Dimension calculations} \label{section 3}

%\ToDo{Add some text.?}
In this section, we present the necessary definitions for team-based quantified propositional logic and its extensions with different dependency atoms, and calculate upper dimensions for these extensions.

\subsection{Quantified propositional logic}

We recall the syntax and semantics of quantified propositional logic and several atoms based on team semantics. We also recall the closure properties of the different extensions of quantified propositional logic we study. %\matilda{Åsa said we also should refer to the name quantified boolean logic.}\minna{Check Miika's paper}

Let $\Prop:=\{p_i\mid i\in\mathbb{N}\}$ be the set of propositional symbols. For notational convenience, we also sometimes use other letters, e.g. $q$,$r$,$t$, for propositional symbols.

\begin{definition}
The BNF syntax of quantified propositional logic (\cite{Hannula16}), or $\QPL$, is defined as follows
\[
\phi::= p_i \mid \neg p_i \mid (\phi\wedge\phi) \mid (\phi\lor\phi) \mid \exists p_i\phi \mid \forall p_i\phi,
\]
where $p_i\in \Prop$.
\end{definition}
%Note that the syntax of $\QPL$ is defined such that every formula is in \textit{negation normal form}, i.e., negation symbols can only appear in front of proposition symbols. 
The formulas of $\QPL$ are often called \emph{quantified boolean formulas}.
For a formula $\phi\in\QPL$, we use $\neg\phi$ to denote the formula obtained by pushing the negation to the front of the proposition symbols in the usual way. 
We denote by $\PL$ the fragment of $\QPL$, where existential and universal quantification is disallowed.

An evaluation $s\colon\Prop\to\{0,1\}$ is a mapping that associates each propositional symbol $p_i$ with a truth value 0 or 1. For any tuple $\bar{p}=p_{1}\dots p_{k}$ of propositional variables, we often use the shorthand $s(\bar{p})$ to denote the value sequence $s(p_1)\dots s(p_n)$. We define the set $\Eval$ of all evaluations (over $\Prop$), and its restriction to the variables in $\bar{p}$, $\Eval^{\bar{p}}$,
\begin{align*}
   \Eval:=&\{s\mid s\colon\Prop\to\{0,1\}\}, \text{ and } \\
  \Eval^{\bar{p}}:=&\{s\mid  \exists t\in\Eval \text{ s.t. } s=t\restriction\{p_{1},\dots,p_{k}\}\}.
\end{align*}

A propositional team $T$ over the propositional variables $\bar{p}$ is a set of evaluations whose domains are restricted to $\{p_{1},\dots,p_{k}\}$, i.e., $T\subseteq\Eval^{\bar{p}}$.
For $a\in\{0,1\}$, we define $s(a/p_i)$ as the evaluation $s'$ such that $s'(p_i)=a$ and $s'(x)=s(x)$, whenever $x\neq p_i$.
For any $F\colon T\to\{\{0\},\{1\},\{0,1\}\}$, define 
$T[F/p_i]:=\{s(a/p_1)\mid s\in T,a\in F(s)\}$ and $T[\{0,1\}/p_i]:=\{s(a/p_i)\mid s\in T,a\in\{0,1\}\}$.

For a formula $\phi\in \QPL$ with variables from $\bar{p}$ and a team $T$ over $\bar{p}$, we write $T\models\phi$ to mean that the team $T$ \textit{satisfies} the formula $\phi$, defined as follows
\begin{align*}
T\models p_i               \, \text{ iff }\,& \text{for } s(p_i)=1 \text{ for all } s\in T, \\
T\models \neg p_i          \, \text{ iff }\,& \text{for } s(p_i)=0 \text{ for all } s\in T, \\
T\models \phi\wedge\psi    \, \text{ iff }\,& T\models\phi \text{ and } T\models\psi, \\
T\models\phi\lor\psi       \, \text{ iff }\,& \text{there are } S,S'\subseteq T \text{ such that } S\cup S'=T,\, S\models\phi \text{ and } S'\models\psi, \\
T\models \exists p_i\phi   \, \text{ iff }\,& \text{there is } F\colon T\to\{\{0\},\{1\},\{0,1\}\} \text{ such that } T[F/p_i]\models\phi, \\
T\models \forall p_i\phi   \, \text{ iff }\,& T[\{0,1\}/p_i]\models\phi.
\end{align*}

We recall the semantic clauses for dependence, anonymity, inclusion, exclusion and conditional independence atoms.
\begin{align*}
T\models\ \dep(\bar{p};q)                     \, \text{ iff }\,& \text{for all } s,s'\in T \text{ such that } s(\bar{p})=s'(\bar{p}), \text{ we have } s(q)=s'(q), \\
T\models\bar{p}\Upsilon q                   \, \text{ iff }\,& \text{for all } s\in T, \text{ there is } s'\in T \text{ such that } s(\bar{p})=s'(\bar{p}) \text{ and } s(q)\neq s'(q), \\
T\models\bar{p}\subseteq\bar{q}             \, \text{ iff }\,& \text{for all } s\in T, \text{ there is } s'\in T \text{ such that } s(\bar{p})=s'(\bar{q}), \\
T\models\bar{p}\mid\bar{q}                  \, \text{ iff }\,& \text{for all } s,s'\in T, \text{ we have } s(\bar{p})\neq s'(\bar{q}), \\
T\models \bar{q}\perp_{\bar{p}}\bar{r}      \, \text{ iff }\,& \text{for all } s,s'\in T \text{ such that } s(\bar{p})=s'(\bar{p}), \text{ there is } s''\in T \text{ such that } \\
&\, s''(\bar{p})=s(\bar{p}),\ s''(\bar{q})=s(\bar{q}),\ \text{and } s''(\bar{r})=s'(\bar{r}).
\end{align*}

For inclusion and dependence atoms $\bar{p}\subseteq\bar{q}$ and $\bar{p}\mid\bar{q}$, we assume that the sequences $\bar{p}$ and $\bar{q}$ are of the same length, i.e, $|\bar{p}|=|\bar{q}|$.

The \emph{arity} of the dependence $\dep(\bar{p};q)$, anonymity $\bar{p}\Upsilon q$, inclusion $\bar{p}\subseteq\bar{q}$, and exclusion atom $\bar{p}\mid\bar{q}$  is $|\bar{p}|$. The arity of the conditional independence atom $\bar{q}\perp_{\bar{p}}\bar{r} $ is a triple $(|\bar{p}|,|\bar{q}|,|\bar{r}|)$. The non-conditional independence atom $\bar{q}\perp\bar{r}$ thus has arity $(0,|\bar{q}|,|\bar{r}|)$.

Let $\circledast$ be any team semantical atom, i.e., $\circledast$ can be from $\{\dep(\dots),\Upsilon, \subseteq, \mid, \perp\}$ or one of the atoms defined later in this paper.
%Let $\circledast\in\{\dep(\dots),\Upsilon, \subseteq, \mid, \perp,\dots\}$ be an atom \matilda{Not so elegant but we want this definition also for atoms we haven't yet introduced...}\minna{Maybe we could have: Let $\circledast$ be any team semantical atom, i.e., $\circledast$ can be from $\{\dep(\dots),\Upsilon, \subseteq, \mid, \perp\}$ or one of the atoms defined later in this paper?}. 
We define $\QPL(\circledast)$ as the extension of $\QPL$ with $\circledast$-atoms, i.e., $\QPL(\circledast)$ is the set of formulas over $\Prop$, which are constructed by the rules obtained by extending the BNF syntax of $\QPL$ with the atom $\circledast$.
For $\circledast$-atoms whose arity is a natural number, let $k,n\geq 0$, and define $\QPL(\circledast_k)_n$ as the set of $\QPL(\circledast)$-formulas $\phi$, where there are at most $n$ appearances of $\circledast$-atoms in $\phi$, and each of these $\circledast$-atoms are at most $k$-ary.
%\minna{Maybe this: Let $n\geq 0$ and $k$ be an arity, i.e., $k$ is either a natural number or a triple of natural numbers. Define $\QPL(\circledast_k)_n$ as the set of $\QPL(\circledast)$-formulas $\phi$, where there are at most $n$ appearances of $\circledast$-atoms in $\phi$, and each of these $\circledast$-atoms are at most $k$-ary. 
%We say that a $(k_1,k_2,k_3)$-ary independence atom is at most $(k_1',k_2',k_3')$-ary if $k_i\leq k_i'$ for all $i\in\{1,2,3\}$.?} 
The fragments $\PL(\circledast)$ and $\PL(\circledast_k)_n$ where quantification is disallowed are defined analogously.
We denote by $\Var(\bar{p})$, $\Var(\phi)$, and $\Var(T)$, the sets of variables appearing in the tuple $\bar{p}$, formula $\phi$, and team $T$, respectively. Let $\circledast\in\{\dep(\dots),\Upsilon, \subseteq, \mid, \perp\}$. The set of free variables of a formula $\phi\in \QPL(\circledast)$ is denoted by $\FV(\phi)$, and defined in the usual way. 
%For $\phi,\psi\in \QPL(\circledast)$, we write 
%\begin{align\circledast}
  %  \phi\equiv\psi, \quad \text{ if } \quad &\FV(\phi)=\FV(\psi) \text{ and we have }  T\models\phi \text{ iff } T\models\psi\\ &\text{ for any } T \text{ such that } \FV(\phi)\subseteq\Var(T).
%\end{align\circledast}

In the context of team semantics, it is common to consider a property called \emph{locality}, which essentially states that the satisfaction of a formula depends only on the evaluation of its free variables. This property is not trivially satisfied in all variants of team semantics. For example, under the so-called \emph{strict} semantics, disjunction and inclusion atoms can be used to construct formulas for which locality fails \cite{galliani12}. The following proposition was introduced in \cite{hannula17} for various extensions of quantified propositional logic. A straightforward induction proof can be used to show the following proposition for the atoms under the semantics considered in this paper. For $V\subseteq\Prop$, define $T\restriction V:=\{s\restriction V\mid s\in T\}$.  
\begin{proposition}[Locality]
    Let $\phi\in\QPL(\circledast)$, $\circledast\in\{\dep(\dots),\Upsilon, \subseteq, \mid, \perp\}$ Then for any $V\subseteq\Prop$ such that $\FV(\phi)\subseteq V\subseteq\Var(T)$, 
    \[
T\models\phi \quad\text{ iff }\quad T\restriction V\models \phi.
    \]
\end{proposition}
Due to this locality property, we can define for any $\phi\in \QPL(\circledast)$ and $\bar{p}$ such that $\FV(\phi)\subseteq\Var(\bar{p})$, the following family of teams, called a \emph{team property},
\[
%\phi^{\bar{p}}:=\{s\in\Eval^{\bar{p}}\mid s\models\phi\}\quad \text{ and } \quad
\lVert \phi\rVert^{\bar{p}}:=\{T\subseteq\Eval^{\bar{p}}\mid T\models \phi\}.
\]
%It is easy to see that 
%$\phi^{\bar{p}}\models \phi$  whenever $\phi\in \QPL(\circledast)$ with $\circledast\in\{\Upsilon,\subseteq\}$, while this does not necessarily hold for $\phi\in \QPL(\circledast)$ with $\circledast\in\{\dep(\dots), \mid, \perp\}$. 

%\ToDo{Add as a remark the closure properties of the atoms and corresponding logics.}  
We say that a formula has a property such as downward closure if the $\lVert \phi\rVert^{\bar{p}}$ of teams satisfying the formula is downward closed. Moreover, we say that a logic $\mathcal{L}$ has a property if each formula $\phi\in\mathcal{L}$ has the property. Dependence and exclusion atoms are downward closed, anonymity and inclusion atoms are union closed, and (conditional) independence atoms have none of the properties studied so far. 
%However, all  formulas $\phi\in \QPL(\circledast)$, $\circledast\in\{\dep(\dots),\Upsilon, \subseteq, \mid, \perp\}$, have 
We say that a formula $\phi$ has the empty team property if $\emptyset\in \lVert \phi\rVert^{\bar{p}}$. Now,  
\begin{itemize}
\item $\QPL$ is downward closed, union closed and has the empty team property,
    \item $\QPL(\circledast)$ with $\circledast\in\{\dep(\dots),\Upsilon, \subseteq, \mid, \perp\}$ has the empty team property,
    \item $\QPL(\dep(\dots))$ and $\QPL(\mid)$ are downward closed,
    \item  $\QPL(\Upsilon)$ and $\QPL(\subseteq)$ are union closed.
\end{itemize}

%\ToDo{Add flatness definition.}\matilda{I changed the flatness and the semantic equivalence so that they don't refer to some specific logic.} 
It is straightforward to confirm that a formula $\phi$ being downward closed, union closed and having the empty team property is equivalent to $\phi$ being \emph{flat}, i.e., $T\models\phi$ iff $s\models\phi$ for all $s\in T$. Thus $\QPL$ is flat, but none of the atoms in $\{\dep(\dots),\Upsilon, \subseteq, \mid, \perp\}$ are flat.

%\begin{proof}
   % {Check if in other paper. }\minna{The notion of locality is defined in Def 41, but it seems that this is not proved in the other dimension paper. Maybe we just find some other ref for this? I'll check the other paper about QPL.}
   % \minna{This is in \cite{hannula17} for $\QPL(\dep(\dots))$, but the proof is not included, maybe we cite this and say that the proof is on induction?} \matilda{Yes, let's do that.}
%\end{proof}
For formulas $\phi$ and $\psi$, we write 
\begin{align*}
    \phi\equiv\psi, \quad \text{ if } \quad &\FV(\phi)=\FV(\psi) \text{ and we have }  T\models\phi \text{ iff } T\models\psi\\ &\text{ for any } T \text{ such that } \FV(\phi)\subseteq\Var(T).
\end{align*}
Let $\mathcal{L}$ and $\mathcal{L}'$ be two logics with propositional team semantics. We write $\mathcal{L}\leq\mathcal{L}'$, if for every formula $\phi\in\mathcal{L}$, there is a formula $\psi\in\mathcal{L}'$ such that $\phi\equiv\psi$. The strict inequality $\mathcal{L}<\mathcal{L}'$ and the equivalence $\mathcal{L}\equiv\mathcal{L}'$ of logics are defined from $\mathcal{L}\leq\mathcal{L}'$ in the obvious way. If $\mathcal{L}\equiv\mathcal{L}'$, we say the logics $\mathcal{L}$ and $\mathcal{L}'$ are \emph{expressively equivalent}.

% \subsection{General formulas}

%\q{Minna: We need to decide what details we want to include, because I am not sure whether we want to add the whole proof here.}
%\minna{I could  try to add a proof that doesn't use the notion of Kripke operators, but I am not sure whether it is worth it because we already "reprove" a lot of stuff from \cite{Hella_Luosto_Vaananen_2024}, and I think that in this paper the main focus is on the atoms, so this result is here more of a technical thing that we require rather that something interesting in itself.}
%\q{Explain the idea of the proof.}

\subsection{Upper dimensions for extensions of $\QPL$}

We calculate the upper dimensions for quantifiers and atoms in the propositional setting by adapting the results from the first-order setting in \cite{Hella_Luosto_Vaananen_2024}.  
%\ToDo{Add some text.}

First, we consider the upper dimensions of conjunction, disjunction and the quantifiers in the propositional setting. These connectives and quantifiers are so-called point-wise Kripke-operators, covered more generally in \cite{Hella_Luosto_Vaananen_2024}.

\begin{proposition}[\cite{Hella_Luosto_Vaananen_2024}]\label[proposition]{kripkeop}
     Let $\phi,\psi\in\QPL(\circledast)$, $\circledast\in\{\dep(\dots),\Upsilon, \subseteq, \mid, \perp\}$, be such that $FV(\phi)\cup FV(\psi)\subseteq\Var(\bar{p})$.   Then the following claims hold.
    \begin{enumerate} [label=(\roman*)]
        \item $\Dim(\lVert\phi\wedge\psi\rVert^{\bar{p}})\leq\Dim(\lVert\phi\rVert^{\bar{p}})\cdot \Dim(\lVert\psi\rVert^{\bar{p}})$,
        \item $\Dim(\lVert\phi\lor\psi\rVert^{\bar{p}})\leq\Dim(\lVert\phi\rVert^{\bar{p}})\cdot \Dim(\lVert\psi\rVert^{\bar{p}})$,
        \item $\Dim(\lVert\exists p_i\phi\rVert^{\bar{p}})\leq\Dim(\lVert\phi\rVert^{\bar{p}})$,
        \item $\Dim(\lVert\forall p_i\phi\rVert^{\bar{p}})\leq\Dim(\lVert\phi\rVert^{\bar{p}})$.
    \end{enumerate}
\end{proposition}
\begin{proof} (A sketch of the proof idea from \cite{Hella_Luosto_Vaananen_2024}.)
    Let $X$ and $Y$ be nonempty sets, and $\mathcal{R}\subseteq\mathcal{P}(Y)\times\mathcal{P}(X)^n$. Define then the operator $\Delta_{\mathcal{R}}\colon\mathcal{P}(\mathcal{P}(X))^n\to\mathcal{P}(\mathcal{P}(Y))$ such that $A\in\Delta_{\mathcal{R}}(\mathcal{A}_0,\dots,\mathcal{A}_{n-1})$ iff $\exists A_0\in\mathcal{A}_0\dots\exists A_{n-1}\in\mathcal{A}_{n-1}$ such that $(A,A_0,\dots,A_{n-1})\in\mathcal{R}$. An operator $\Delta\colon\mathcal{P}(\mathcal{P}(X))^n\to\mathcal{P}(\mathcal{P}(Y))$ is a \emph{Kripke operator} if $\Delta=\Delta_{\mathcal{R}}$ for some $\mathcal{R}\subseteq\mathcal{P}(Y)\times\mathcal{P}(X)^n$. The connectives $\wedge,\lor$ and the quantifiers $\exists,\forall$ can be viewed a Kripke operators on families of teams. For example, since $\lVert\phi\wedge\psi\rVert^{\bar{p}}=\lVert\phi\rVert^{\bar{p}}\cap\lVert\psi\rVert^{\bar{p}}$, the conjunction corresponds to the Kripke operator of the relation $\mathcal{R}_{\cap}=\{(D,D,D)\mid D\in\mathcal{P}(X)\}$, i.e., the intersection of families. Moreover, these connectives and quantifiers can be shown to be \emph{point-wise} Kripke operators, which essentially means that for any $A\in\mathcal{P}(Y)$, the relation $\mathcal{R}[A]:=\{(A_0,\dots,A_{n-1})\mid(A,A_0,\dots,A_{n-1})\in\mathcal{R}\}$ is
completely determined by the relations $\mathcal{R}[\{a\}]$, $a\in A$. Point-wise Kripke operators $\Delta$ on finite $X$ and $Y$ \emph{weakly preserve dominated convexity}. This means that $\Delta(\mathcal{A}_0,\dots,\mathcal{A}_{n-1})$ is dominated and convex (or empty), if each $\mathcal{A}_i$, $0\leq i< n$, is dominated and convex. If $\Delta$ weakly preserves dominated convexity, then it can be shown that $\Dim(\Delta(\mathcal{A}_0,\dots,\mathcal{A}_{n-1}))\leq \Dim(\mathcal{A}_0)\cdot \ldots \cdot\Dim(\mathcal{A}_{n-1})$.
\end{proof}

We examine the upper dimensions of formulas in light of the locality property; only the variables in the formula affect its upper dimension. See \Cref{app1} for the proof.

\begin{proposition}[\cite{Hella_Luosto_Vaananen_2024}]\label[proposition]{dim and locality prop}
Let $\phi\in\QPL(\circledast)$, $\circledast\in\{\dep(\dots),\Upsilon, \subseteq, \mid, \perp\}$, be such that $\FV(\phi)\subseteq\Var(\bar{p})\subseteq\Var(\bar{p}')$.
\begin{enumerate}[label=(\roman*)]
    \item   $\Dim(\lVert\phi\rVert^{\bar{p}})=\Dim(\lVert\phi\rVert^{\bar{p}'})$.

    \item  %Let $\phi\in\QPL(\circledast)$ be such that $\FV(\phi)\subseteq\Var(\bar{p})\subseteq\Var(\bar{p}')$. 
$\Dim^d(\lVert\phi\rVert^{\bar{p}})\leq2^{n}\cdot\Dim^d(\lVert\phi\rVert^{\bar{p}'})$, where 
    \[
    n=\begin{cases}
        0, \text{ if } \lVert\phi\rVert^{\bar{p}} \text{ is convex,}\\
        |\Var(\bar{p}')\setminus\Var(\bar{p})|\cdot\max\{|T|\mid T\in\lVert\phi\rVert^{\bar{p}}\}, \text{ otherwise}.
    \end{cases}
    \]
   % and $\mathcal{H}$ is minimal supporting subfamily of $\lVert\phi\rVert^{\bar{p}}$
\end{enumerate}
   
\end{proposition}

%\q{some explanatory text here?}

Before calculating the upper dimensions of propositional atoms, we define some families that will be useful for the calculations. These families correspond to sets of teams that each satisfy a particular atom. The calculations of their upper dimensions are based on the results of \cite{Hella_Luosto_Vaananen_2024} for dimensions in the setting of first-order team semantics.

\begin{definition}[\cite{Hella_Luosto_Vaananen_2024}]\label{sets}
Let $X$ and $Y$ be nonempty finite sets. We define the following families
\begin{align*}
    \Fdep&:=\{f\subseteq X\times Y\mid 
 f \text{ is a mapping}\},\\
 \Fano&:=\{R\subseteq X\times Y\mid \forall x\in\Dom(R), \exists y,y'\in Y\text{ s.t. } y\neq y'\text{ and } (x,y),(x,y')\in R\},\\
 \Finc&:=\{R\subseteq X\times X\mid \Dom(R)\subseteq\Ran(R)\},\\
    \Fexc&:=\{R\subseteq X\times X\mid \Dom(R)\cap\Ran(R)=\emptyset\},\\
    \Find&:=\{A\times B\mid 
 A\subseteq X, B\subseteq Y\}.
\end{align*}
\end{definition}
% \q{Minna: I defined new commands for the names of these sets. Some of these look ugly, so it would be good to try to refine the commands in the beginning.} 
% \q{Maybe move the proof of Theorem 3.5 to appendix? Or it might be good to include these in the main text instead, because they aren't that long, and they clarify the concept of upper dimension?} \matilda{I think we should include the proof of, say, two of the items here and move the rest to the appendix.}

\begin{theorem}[\cite{Hella_Luosto_Vaananen_2024}]\label{dim_families}
Let $X$ and $Y$ be as in Definition \ref{sets}, and assume that $|X|=\ell\geq 2$ and $|Y|=n\geq 2$. Then the following claims hold for the upper dimensions of the families. 
\begin{align*}
   \Dim(\Fdep)=n^\ell, \quad&\Dim(\Fano)=2^\ell,\quad\Dim(\Finc)=2^\ell-\ell,\\\quad\Dim(\Fexc)=2^\ell-2,\quad\text{and}\quad
   &\Dim(\Find)=(2^\ell-\ell-1)(2^n-n-1)+\ell+n.
\end{align*}
 %   \begin{itemize} 
 %       \item[(i)] $\Dim(\Fdep)=n^\ell$, 
 %       \item[(ii)] $\Dim(\Fano)=2^\ell$,
  %      \item[(iii)] $\Dim(\Finc)=2^\ell-\ell$,
  %      \item[(iv)] $\Dim(\Fexc)=2^\ell-2$,
   %     \item[(v)] $\Dim(\Find)=(2^\ell-\ell-1)(2^n-n-1)+\ell+n$.
%\end{itemize}
\end{theorem}

\begin{proof}
\begin{itemize}
    \item  We show that $\Dim(\Fdep)=n^\ell$. First note that  $\Fdep$ is downward closed, %and \textcolor{red}{finite}, 
    so by \Cref{dc and uc prop}, $\Dim(\Fdep)=|\Max(\Fdep)|$. The maximal sets in $\Fdep$ are total functions $f\colon X\to Y$, so $\Dim(\Fdep)=|Y|^{|X|}=n^\ell$. 

    \item Next, we show that $\Crit(\Fano)=\{A\times Y\mid A\subseteq X\}$ is the smallest subfamily dominating $\Fano$, so $\Dim(\Fano)=|\{A\times Y\mid A\subseteq X\}|=2^\ell$.
        
        Let $R\subseteq R'\subseteq X\times Y$. Then by the definition of $\Fano$, $[R,R']\subseteq\Fano$ if and only if $R,R'\in\Fano$ and $\Dom(R)=\Dom(R')$. 
         Since $\partial_{R'}(\Fano)=\{S\subseteq R'\mid [S,R']\subseteq\Fano\}$, this means that for any $R,R'\in\Fano$, $R\in \partial_{R'}(\Fano)$ if and only if $R\subseteq R'$ and $\Dom(R)=\Dom(R')$. Since $R\in \partial_{R'}(\Fano)$ if and only if $\partial_R(\Fano)\subseteq\partial_{R'}(\Fano)$, the critical sets of $\Fano$ are of the form $A\times Y$ for some $A\subseteq X$, so $\Crit(\Fano)=\{A\times Y\mid A\subseteq X\}$.

          We now show that $\Crit(\Fano)$ is the smallest dominating subfamily of $\Fano$ Let $\mathcal{G}\subseteq\Fano$ be such that it dominates $\Fano$. 
Since $\Crit(\Fano)\subseteq\Fano=\bigcup_{G\in\mathcal{G}}\mathcal{D}_{G}$ for some convex families $\mathcal{D}_{G},G\in\mathcal{G}$, we have that for every $A\in\Crit(\Fano)$ there is some $G\in\mathcal{G}$ such that $A\in\mathcal{D}_{G}$. Since $\mathcal{D}_G$ is convex, $[A,G]\subseteq\Fano$, and hence $A\in\partial_G(\Fano)$.
Then 
$\partial_A(\Fano)\subseteq \partial_G(\Fano)$ and since $A$ is critical, 
we have $A=G$. Thus $A\in\mathcal{G}$, and $\Crit(\mathcal{A})\subseteq\mathcal{G}$.

\item The remaining calculations are in \Cref{app1}.
\end{itemize}

\end{proof}
 % \ToDo{Make clear for the previous thm and the corollary where the numbers come from.} \minna{Added the following text in blue.}
 
For dimensions in the context of propositional team semantics, the sets $X$ and $Y$ of the above theorem correspond to projections of propositional teams to the relevant variables, i.e., they contain tuples of zeroes and ones of suitable length. For example, if $X$ is the set of all $k$-length tuples of zeroes and ones, we have $|X|=2^k$. Now we obtain the upper dimensions of the propositional atoms as an easy corollary, with the case for unary dependence atoms illustrated in \Cref{figure lattice}.
\begin{corollary}
Let $\bar{p},\bar{q},q,\bar{r}$ be sequences of mutually distinct propositional variables with $|\bar{p}|=|\bar{q}|=k$ and $|\bar{r}|=m$. Then the following claims hold.
     \begin{enumerate}[label=(\roman*)]
        \item $\Dim(\lVert\phi\rVert^{\bar{p}})=1$ for any $\phi\in \QPL$,
        \item $\Dim(\lVert\dep(\bar{p};q)\rVert^{\bar{p}q})=2^{2^k}$,
        \item $\Dim(\lVert \bar{p}\Upsilon q\rVert^{\bar{p}q})=2^{2^k}$,
        \item $\Dim(\lVert \bar{p}\subseteq \bar{q}\rVert^{\bar{p}\bar{q}})=2^{2^k}-2^k$,
        \item $\Dim(\lVert \bar{p}\mid  \bar{q}\rVert^{\bar{p}\bar{q}})=2^{2^k}-2$,
        \item  $\Dim(\lVert \bar{q}\perp  \bar{r}\rVert^{\bar{q}\bar{r}})=(2^{2^k}-2^k-1)(2^{2^m}-2^m-1)+2^k+2^m$.
\end{enumerate}
\end{corollary}
\begin{proof}
    \begin{itemize}
        \item[(i)] Since any formula $\phi\in\QPL$ is downward closed, % and $|\bar{p}|$ is finite, 
        by \Cref{dc and uc prop}, we have $\Dim(\lVert\phi\rVert^{\bar{p}})=|\Max(\lVert\phi\rVert^{\bar{p}})|=|\Max([\emptyset,\phi^{\bar{p}}])|=|\{\phi^{\bar{p}}\}|=1$, where $\phi^{\bar{p}}:=\{s\in\Eval^{\bar{p}}\mid s\models\phi\}$.
       \item[(ii)-(v)] Note that by letting $X=\{0,1\}^k$ and $Y=\{0,1\}$, we have $|X|=2^k$ and $|Y|=2$. Then by Theorem \ref{dim_families}, we obtain $\Dim(\lVert\dep(\bar{p};q)\rVert^{\bar{p}q})=\Dim(\Fdep)=|Y|^{|X|}=2^{2^k}$. The other cases are analogous. 
        \item[(vi)] By letting $X=\{0,1\}^k$ and $Y=\{0,1\}^m$, we obtain the claim by Theorem \ref{dim_families} as in the previous items.
    \end{itemize}
\end{proof}

    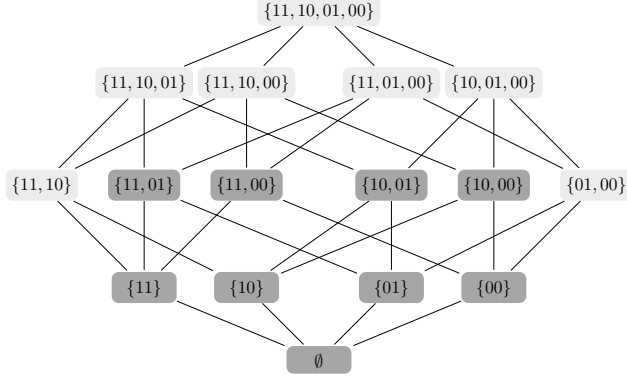
\begin{figure}\centering
\resizebox{.7\columnwidth}{!}{
       \begin{tikzpicture}[> = stealth,  shorten > = 1pt,   auto,  node distance = 2.05cm , state/.style={
 rectangle, rounded corners,
   fill=gray!15,
  minimum width=1.3cm, 
  align=center,      
  text height=.3cm  
}]

%%% LEVEL 1
       \node[state] (x) {$\{11,10,01,00\}$};
%%% LEVEL 2
        \node[state] [below left of=x] (B) {$\{ 11,10,00 \}$};
        \node[state][below right of=x]  (C) {$\{ 11,01,00 \}$};
        \node[state] [left  of=B] (A) {$\{ 11,10,01 \}$};
        \node[state][right of=C]  (D) {$\{ 10,01,00 \}$};
%%%% LEVEL 3
        \node[state, fill=gray!70][below of=B]  (AD) {$\{ 11,00 \}$};
        \node[state, fill=gray!70][below of=C]  (BC) {$\{ 10,01 \}$};
        \node[state, fill=gray!70][below of=D]  (BD) {$\{ 10,00 \}$};
        \node[state, fill=gray!70][below of=A]  (AC) {$\{ 11,01 \}$};
        \node[state]  [left  of=AC] (AB){$\{ 11,10 \}$};
        \node[state]  [right of=BD] (CD){$\{ 01,00 \}$};

%%  LEVEL   4
        \node[state, fill=gray!70][below of=AD]  (ABD) {$\{ 10 \}$};
        \node[state, fill=gray!70][below of=BC]  (ACD) {$\{ 01 \}$};
        \node[state, fill=gray!70][left of=ABD]  (ABC) {$\{ 11 \}$};
        \node[state, fill=gray!70][right of=ACD]  (BCD) {$\{ 00\}$};

%%%% LEVEL 5
\node[state, fill=gray!70][below right of=ABD]  (ABCD) {$\emptyset$};

%% LEVEL 1 to LEVEL 2
		\path[-] (x)  edge node {} (A);
		\path[-] (x)  edge node {} (B);
		\path[-] (x)  edge node {} (C);
		\path[-] (x)  edge node {} (D);

%% LEVEL 2 to LEVEL 3
		\path[-] (A)  edge node {} (AB);
		\path[-] (A)  edge node {} (AC);
		\path[-] (A)  edge node {} (BC);
		\path[-] (B)  edge node {} (AB);
		\path[-] (B)  edge node {} (AD);
		\path[-] (B)  edge node {} (BD);
		\path[-] (C)  edge node {} (AC);%
		\path[-] (C)  edge node {} (AD);
		\path[-] (C)  edge node {} (CD);
		\path[-] (D)  edge node {} (BC);
		\path[-] (D)  edge node {} (BD);
		\path[-] (D)  edge node {} (CD);
%%% LEVEL 3 TO 4

		\path[-] (AB)  edge node {} (ABC);
		\path[-] (AB)  edge node {} (ABD);
		\path[-] (AC)  edge node {} (ABC);
		\path[-] (AC)  edge node {} (ACD);
		\path[-] (BC)  edge node {} (ABD);
		\path[-] (BC)  edge node {} (ACD);
		\path[-] (AD)  edge node {} (ABC);
		\path[-] (AD)  edge node {} (BCD);
		\path[-] (BD)  edge node {} (ABD);
		\path[-] (BD)  edge node {} (BCD);
		\path[-] (CD)  edge node {} (ACD);
		\path[-] (CD)  edge node {} (BCD);
%%%% LEVEL 4 to 5
		\path[-] (ABC)  edge node {} (ABCD);
		\path[-] (ABD)  edge node {} (ABCD);
		\path[-] (ACD)  edge node {} (ABCD);
		\path[-] (BCD)  edge node {} (ABCD);
%%%% Dashed lines
      \end{tikzpicture}
}
\caption{The dark gray teams in the picture belong to $\lVert\dep(p_1;q)\rVert^{p_1q}$, where $10(p_1)=1$ and $10(q)=0$ etc. There are four maximal teams, hence $\Dim(\lVert\dep(p_1;q)\rVert^{p_1q})=4$. \label{figure lattice}}
\end{figure}

\section{Translations}
\label{section 4}

%\ToDo{Define "reduction formula" and "optimal" Minna: Added somewhat informal definitions (blue), do these seem okay?}

We study the translation of different atoms into equivalent formulas of a simpler form, meaning that all atoms have a reduced arity. We call these \emph{reduction formulas} since they allow us to reduce the arity of the atoms in a formula by replacing them with equivalent expressions. For dependence atoms and anonymity atoms, the translations are succinct, with the number of atoms only increasing by a factor of $2$. 
%
%We will later in \Cref{subsec inexpressibility} use dimensions to prove that this is optimal.
%, in the sense that a $k+1$-ary dependence or anonymity atom cannot be expressed with a formula that has fewer than two $k$-ary dependence or anonymity atoms, respectively. 
The independence atom also has a simple reduction formula.

For exclusion and inclusion atoms, the task is more involved. We examine two approaches to obtain reduction formulas, one with quantifiers and one by introducing relativized variants of inclusion and exclusion atoms. The former approach takes inspiration from translations in the first-order team-based setting, while the latter is novel.  We end the section by showing that $\QPL(\subseteq)$ and $\PL$ extended with relativized inclusion atoms are expressively complete for all union closed team properties with the empty team, and that 
$\QPL(\mid)$ and $\PL$ extended with relativized exclusion atoms are expressively complete for all nonempty downward closed team properties. We also discuss that the respective expressive completeness results cannot be obtained for $\PL$ extended with the usual propositional inclusion and exclusion atoms.

\subsection{Dependence, anonymity, and independence atoms}

%For propositional dependence atoms:

%$\dep(p_1\dots p_{k+1};q)\equiv ((p_{k+1}\wedge \dep(p_1\dots p_{k}; q))\lor(\neg p_{k+1}\wedge \dep(p_1\dots p_{k}; q)))$

%For propositional anonymity atoms:

%$p_1\dots p_{k+1}\Upsilon q\equiv ((p_{k+1}\wedge p_1\dots p_{k}\Upsilon q)\lor(\neg p_{k+1}\wedge p_1\dots p_{k}\Upsilon q))$

%$\bar{q}\perp_p \bar{r}\equiv (p\wedge \bar{q}\perp \bar{r})\lor(\neg p\wedge \bar{q}\perp \bar{r})$

%$p\bar{q}\perp \bar{r}\equiv p\perp \bar{r}\wedge((p\wedge \bar{q}\perp \bar{r})\lor(\neg p\wedge \bar{q}\perp \bar{r}))$

Propositional dependence, anonymity, and conditional independence atoms all have simple reduction formulas. In the case of dependence and anonymity, one application of the reduction formulas reduces the maximal atom arity by one. Also conditional independence atoms with arity $(1,m,k)$ can be reduced to non-conditional independence atoms $(0,m,k)$ with one application of the translation, and the final translation formula reduces the arity to $(0,m-1,k)$ or $(0,m,k-1)$ with one application, remembering that the atom is symmetrical.

\begin{proposition}\label[proposition]{reductionfmlas}
    The following equivalences hold.
     \begin{enumerate}[label=(\roman*)]
        \item $\dep(p_1\dots p_{k+1};q)\equiv (p_{k+1}\wedge \dep(p_1\dots p_{k}; q))\lor(\neg p_{k+1}\wedge \dep(p_1\dots p_{k}; q))$. \label{reductionfmlas item dep}
        \item $p_1\dots p_{k+1}\Upsilon q\equiv ((p_{k+1}\wedge p_1\dots p_{k}\Upsilon q)\lor(\neg p_{k+1}\wedge p_1\dots p_{k}\Upsilon q))$. \label{reductionfmlas item anon}
        \item $p_1\dots p_k\perp_{p_{k+1}} q_1\dots q_n\equiv (p_{k+1}\wedge p_1\dots p_k\perp q_1\dots q_n)\lor(\neg p_{k+1}\wedge p_1\dots p_k\perp q_1\dots q_n)$. \label{reductionfmlas item cond ind}
      \item $p_1\dots p_{k+1}\perp q_1\dots q_n\equiv p_{k+1}\perp q_1\dots q_n\wedge((p_{k+1}\wedge p_1\dots p_{k}\perp q_1\dots q_n)\lor(\neg p_{k+1}\wedge p_1\dots p_{k}\perp q_1\dots q_n))$.\footnote{The reduction formulas in items \ref{reductionfmlas item cond ind} and \ref{reductionfmlas item ind} for independence are joint work with Fan Yang and included here with her permission.} \label{reductionfmlas item ind}
     \end{enumerate}
\end{proposition}
\begin{proof}
    Items \ref{reductionfmlas item dep}-\ref{reductionfmlas item cond ind} are easy to verify. We prove \cref{reductionfmlas item ind}, defining $\bar{p}=p_1\dots p_k$ and $\bar{q}=q_1\dots q_n$. Let $T\models \bar{p}p_{k+1}\perp \bar{q}$, then  $T\models p_{k+1}\perp \bar{q}$ follows. 
    %\minna{Should this be $T\models p_{k+1}\perp \bar{q}$?}\matilda{Yes, fixed.} 
    Let $S_0,S_1\subseteq T$ be such that $S_0\models \neg p_{k+1}$, $S_1\models  p_{k+1}$ and $S_0\cup S_1=T$. Now for any $v,v'\in S_0$, $v(p_{k+1})=0$, hence also the witness $v''\in T$ for which $v''(\bar{p}p_{k+1})=v(\bar{p}p_{k+1})$ and $v''(\bar{q})=v'(\bar{q})$ is in $S_0$. Hence $S_0\models \bar{p}p_{k+1}\perp \bar{q}$, and thus also $S_0\models \bar{p}\perp \bar{q}$. The case for $S_1$ is similar, and we conclude $T\models p_{k+1}\perp \bar{q}\wedge((p_{k+1}\wedge \bar{p}\perp \bar{q})\lor(\neg p_{k+1}\wedge \bar{p}\perp \bar{q}))$. 

    For the other direction, let $T\models p_{k+1}\perp \bar{q}\wedge((p_{k+1}\wedge \bar{p}\perp \bar{q})\lor(\neg p_{k+1}\wedge \bar{p}\perp \bar{q}))$. Now $T\models p_{k+1}\perp \bar{q}$ and there are $S_0,S_1\subseteq T$ such that $S_0\models \neg p_{k+1}\land \bar{p}\perp \bar{q}$, $S_1\models p_{k+1}\land \bar{p}\perp \bar{q}$ and $S_0\cup S_1=T$. Let $v,v'\in T$. By $T\models p_{k+1}\perp \bar{q}$, there is a $v''\in T$ such that $v''(p_{k+1})=v(p_{k+1})$ and $v''(\bar{q})=v'(\bar{q})$. If $v(p_{k+1})=0$, then $v''\in S_0$. Since $p_{k+1}$ has a constant value in $S_0$, $S_0\models \bar{p}p_{k+1}\perp \bar{q}$ and we find $v_0\in S_0$ such that $v_0(\bar{p}p_{k+1})=v''(\bar{p}p_{k+1})=v(\bar{p}p_{k+1})$ and $v_0(\bar{q})=v''(\bar{q})=v'(\bar{q})$.  
The case when $v(p_{k+1})=0$ is symmetrical using the other disjunct and the subteam $S_1$. Hence $T\models \bar{p}p_{k+1}\perp \bar{q}$.
\end{proof}

It is known in the literature \cite{MR3488885} that propositional dependence logic $\PL(\dep(\dots))$ is expressively equivalent to propositional constancy logic $\PL(\dep(\dots)_0)$, and we note that the reduction formula in \Cref{reductionfmlas} \cref{reductionfmlas item dep} when iterated results in a formula with only constancy atoms (equivalently, dependence atoms of arity $0$). Similarly, the reduction formula in \cref{reductionfmlas item anon} allows us to reduce the arity of anonymity atoms to zero, resulting in non-constancy atoms ($\not\dep(\cdot)$). In \cite{yang2022}, the logics are shown to be expressively equivalent, i.e., $\PL(\Upsilon)\equiv \PL(\Upsilon_0)\equiv \PL(\not\dep(\cdot))$. %, where $T\models \not\dep(p)$ if and only if $T=\emptyset$ or there are $s_1,s_2\in T$ such that $s_1(p)\neq s_2(T)$. 

\subsection{Quantifiers with inclusion and exclusion atoms}
\label{subsec inc exc}

%\minna{This subsection is quite long. Should we have two subsubsections (quantifiers/relativized) here or would that be too many "subs"?}\matilda{Yes, we can try to split it.}

In this section, we use translation results from the setting of first-order team semantics and the reduction formulas of Theorem \ref{reductionfmlas} to obtain reduction formulas for propositional inclusion and exclusion atoms. These translations use propositional quantifiers, unlike the reduction formulas of Theorem \ref{reductionfmlas}. 

% We first use translation results from the setting of first-order team semantics and the reduction formulas of Theorem \ref{reductionfmlas} to obtain reduction formulas for propositional inclusion and exclusion. Note that the formulas that we define in this section use propositional quantifiers, unlike the reduction formulas of Theorem \ref{reductionfmlas}. \ToDo{Combine previous two paragraphs}

In the setting of first-order team semantics, it is known that inclusion atoms can be expressed using the anonymity atoms \cite{ronnholm18}.
Because quantified propositional team semantics can be seen as a restriction of first-order team semantics to two-element models with one unary relation, propositional inclusion atoms can also be expressed using propositional anonymity atoms by an analogous formula. To utilize the first-order translation that uses identities between variables, we define 
%The only difference is that the identity between variables is not an atomic formula in propositional logic. However, since it is expressible in propositional logic, 
the notations $p_i=p_j$ and $p_i\neq p_j$ %can be viewed 
as the abbreviations of the formulas $(p_i\wedge p_j)\lor(\neg p_i\wedge \neg p_j)$ and $(p_i\wedge \neg p_j)\lor(\neg p_i\wedge  p_j)$, respectively. We will also use analogous shorthand notations for tuples of propositional variables, i.e., we write $\bar{p}=\bar{p}'$ for $\bigwedge_{1\leq i\leq k} p_i=p'_i$ and $\bar{p}\neq\bar{p}'$ for $\bigvee_{1\leq i\leq k}p_i\neq p'_i$. Now we can use the translation from \cite{ronnholm18} to express inclusion atoms with anonymity atoms as in \Cref{eq1}. Similarly, the anonymity atom can be expressed using the inclusion atom in the first-order team semantics \cite{galliani12}, and an analogous formula can also be used in the quantified propositional setting shown in \Cref{eq2}.
%
% follows
% $\bar{t}\subseteq\bar{t}'\equiv\forall p_1\forall p_2\exists\bar{q}\exists z(((p_1=p_2\wedge\bar{q}=\bar{t}\wedge q_1=z)\lor(p_1\neq p_2\wedge\bar{q}=\bar{t}'))\wedge\bar{q}\Upsilon z)$. Similarly, the anonymity atom can be expressed using the inclusion atom in the first-order team semantics \cite{galliani12}, and an analogous formula can also be used in the quantified propositional setting. Hence, propositional anonymity atoms can be expressed using propositional inclusion atom as follows
% $\bar{p}\Upsilon q\equiv \exists u(u\neq q\wedge\bar{p}u\subseteq\bar{p}q)$.
%
%\begin{align} \label{eq1}
 %   \bar{t}\subseteq\bar{t}'\equiv &\forall p_1\forall p_2\exists\bar{q}\exists z(((p_1=p_2\wedge\bar{q}=\bar{t}\wedge q_1=z)\\
  %  &\lor(p_1\neq p_2\wedge\bar{q}=\bar{t}'))\wedge\bar{q}\Upsilon z),
   % \notag\\
    %
   %\bar{p}\Upsilon q\equiv &\exists u(u\neq q\wedge\bar{p}u\subseteq\bar{p}q).\label{eq2}
%\end{align}
\begin{align} \label{eq1}
    \bar{t}\subseteq\bar{t}'\equiv &\forall z_1\forall z_2\exists\bar{p}\exists q(((z_1=z_2\wedge p_1=q\wedge\bar{p}=\bar{t})\\
    &\lor(z_1\neq z_2\wedge\bar{p}=\bar{t}'))\wedge\bar{p}\Upsilon q),
    \notag\\
   \bar{p}\Upsilon q\equiv &\exists u(u\neq q\wedge\bar{p}u\subseteq\bar{p}q).\label{eq2}
\end{align}
We use these translations to obtain reduction formulas for inclusion atoms in $\QPL(\subseteq)$.

%{Check what number of occurrences would give inexpressibility results for inclusion and exclusion.}
% \minna{It seems that for $\subseteq_2$-atoms the translation using four $\subseteq_1$ atoms is "optimal", i.e., we have $\QPL(\subseteq_{2})_n\not\leq\QPL(\subseteq_1)_m$ for all $n,m$, s.t., $0<n\leq m< 4n$. For $k\geq 2$ and $\circledast\in\{\subseteq,|\}$, we have $\QPL(\circledast_{k+1})_n\not\leq\QPL(\circledast_k)_m$ for all $n,m$, s.t., $0<n\leq m< 3n$, but I have to double check my computations to be sure. So the translation for inclusion is close to optimal, but exclusion is quite far away.} \matilda{Wow, that's better than I thought!}
\begin{proposition}\label[proposition]{theoremk+1tok4}
    Every $k+1$-ary propositional inclusion atom has an equivalent formula  $\phi\in \QPL(\subseteq_k)_4$.
\end{proposition}

\begin{proof}
The propositional inclusion atom $t_1\dots t_{k+1}\subseteq t'_1\dots t'_{k+1}$ has a reduction formula that can be obtained as follows. First, express the atom with a propositional anonymity formula using \Cref{eq1}.
Then use the reduction formula $p_1\dots p_{k+1}\Upsilon q\equiv ((p_{k+1}\wedge p_1\dots p_{k}\Upsilon q)\lor(\neg p_{k+1}\wedge p_1\dots p_{k}\Upsilon q))$
 to the anonymity atoms twice, to first obtain
\begin{align*} 
    \forall z_1\forall z_2\exists\bar{p}\exists q(((z_1=z_2&\wedge p_1=q\wedge\bar{p}=\bar{t}\lor(z_1\neq z_2\wedge\bar{p}=\bar{t}'))\wedge (\theta_k^+\lor \theta_k^-)
\end{align*}
and then
\begin{align*} 
    &\forall z_1\forall z_2\exists\bar{p}\exists q(((z_1=z_2\wedge p_1=q\wedge\bar{p}=\bar{t})\lor(z_1\neq z_2\wedge\bar{p}=\bar{t}'))\\
    &\wedge ((p_{k+1}\wedge (\theta_{k-1}^+\lor\theta_{k-1}^-)\lor(\neg p_{k+1}\wedge (\theta_{k-1}^+\lor\theta_{k-1}^-)),
\end{align*}
where $\theta_{j}^+:=(p_{j+1}\wedge p_1\dots p_{j}\Upsilon q)$ and $\theta_{j}^-:=(\neg p_{j}\wedge p_1\dots p_{j}\Upsilon q)$.

Then express the anonymity atoms with inclusion again
\begin{align*} 
    &\forall z_1\forall z_2\exists\bar{p}\exists q(((z_1=z_2\wedge p_1=q\wedge\bar{p}=\bar{t})\lor(z_1\neq z_2\wedge\bar{p}=\bar{t}'))\\
    &\wedge ((p_{k+1}\wedge ((p_k\wedge\chi_k)\lor(\neg p_{k}\wedge \chi_k))\lor(\neg p_{k+1}\wedge ((p_k\wedge \chi_k)\lor(\neg p_{k}\wedge \chi_k)),
\end{align*}
where $\chi_k:=\exists u(u\neq q\wedge p_1\dots p_{k-1}u\subseteq p_1\dots p_{k-1}q)$. The formula is clearly in $\QPL(\subseteq_k)_4$ as wanted.
\end{proof}

%\subsection{Propositional exclusion atoms}

Using quantifiers, we can also obtain propositional versions of the translations between first-order dependence atoms and exclusion atoms from \cite{GALLIANI201268}.
%
%\begin{align} \label{eq3}
 %   t_1\dots t_{k+1}\mid t'_1\dots t'_{k+1}\equiv \,\, & \forall \bar{p}\exists q_1\exists q_2(\dep(\bar{p};q_1)\wedge \dep(\bar{p};q_2)\\
  %  &\wedge((q_1=q_2\wedge \bar{p}\neq \bar{t})\lor(q_1\neq q_2\wedge \bar{p}\neq \bar{t}'))), \notag\\
    %
%    \dep(\bar{p};q)\equiv\,\, &\forall z(z=q\lor \bar{p}z\mid \bar{p}q).\label{eq4}
%\end{align}
\begin{align} \label{eq3}
    \bar{t}\mid \bar{t}'\equiv \,\, & \forall \bar{p}\exists q(\dep(\bar{p};q)\wedge((q\wedge \bar{p}\neq \bar{t})\lor(\neg q\wedge \bar{p}\neq \bar{t}'))),\\
    \dep(\bar{p};q)\equiv\,\, &\forall z(z=q\lor \bar{p}z\mid \bar{p}q).\label{eq4}
\end{align}
%\ToDo{Mention that (3) is slightly simpler than the analogous formula in the first-order setting.}
The translation of \Cref{eq3} is slightly simpler than the analogous formula in the first-order setting. Note that the simplified formula uses the propositional atoms $q$ and $\neg q$, hence it does not immediately provide a similar simplification in the first-order setting. Next we show that this simpler formula is indeed equivalent to a propositional exclusion atom.
\begin{proposition}
    \[
    \bar{t}\mid \bar{t}'\equiv \,\,  \forall \bar{p}\exists q(\dep(\bar{p};q)\wedge((q\wedge \bar{p}\neq \bar{t})\lor(\neg q\wedge \bar{p}\neq \bar{t}')))
    \]
\end{proposition}
\begin{proof}
Let $\phi:=\forall \bar{p}\exists q(\dep(\bar{p};q)\wedge((q\wedge \bar{p}\neq \bar{t})\lor(\neg q\wedge \bar{p}\neq \bar{t}')))$ with $\bar{p}:=p_1\dots p_n$. Denote by $T(\bar{t})$ the set $\{s(\bar{t})\mid s\in T\}$ and note that $T\models \bar{t}\mid \bar{t}'$ if and only if $T(\bar{t})\cap T(\bar{t}')=\emptyset$. We now prove the equivalence.

Suppose first that $T\models \bar{t}\mid \bar{t}'$ and define the duplicate team $T':=T[\{0,1\}/{p_1}\}]\dots[\{0,1\}/{p_n}\}]$. Define the function $F:T'\rightarrow P(\{0,1\})$ by 
  $$F(s):=
   \begin{cases}
      \{0\}, \text{ if } s(\bar{p})\in  T'(\bar{t})\\
      \{1\}, \text{ otherwise.} %s(\bar{p})\not\in T'(\bar{t}).
   \end{cases}$$
 Let $T'':=T'[F/q]$. Clearly $T''\models \dep(\bar{p};q)$, and we find subteams $T_0,T_1\subseteq T''$ defined by $T_0=\{s\in T''\mid s(q)=0\}$ and $T_1=\{s\in T''\mid s(q)=1\}$ for which $T_0\cup T_1=T''$. By the assumption and locality, $T''(\bar{t})\cap T''(\bar{t}')=\emptyset$, thus $T_0\models \neg q\wedge \bar{p}\neq \bar{t}'$ and $T_1\models  q\wedge \bar{p}\neq \bar{t}$. Hence $T''\models(q\wedge \bar{p}\neq \bar{t})\lor(\neg q\wedge \bar{p}\neq \bar{t}')$ and we conclude that $T\models \phi$. 

For the other direction, suppose that  $T\not\models \bar{t}\mid \bar{t}'$ and, for a contradiction, that $T\models \phi$. Consider the duplicate team $T'$ as above. Then there is a function $F:T'\rightarrow P(\{0,1\})$ such that $T''\models\ \dep(\bar{p};q)\wedge((q\wedge \bar{p}\neq \bar{t})\lor(\neg q\wedge \bar{p}\neq \bar{t}'))$, where $T'':=T'[F/q]$. By assumption, there are $s_1,s_2\in T''$ such that $s_1(\bar{p})=s_1(\bar{t})=s_2(\bar{t}')=s_2(\bar{p})$. Since $\{s_1\}\not\models \bar{p}\neq \bar{t}$, it must be that $\{s_1\}\models \neg q \land \bar{p}\neq \bar{t}'$, so $s_1(q)=0$. By a similar argument, $s_2(q)=1$, contradicting $T''\models \dep(\bar{p};q)$.
\end{proof}
%    Let $\phi:=\forall \bar{p}\exists q(\dep(\bar{p};q)\wedge((q\wedge \bar{p}\neq \bar{t})\lor(\neg q\wedge \bar{p}\neq \bar{t}')))$. Now $T\models\phi$ iff the function $f\colon\{0,1\}^{k+1}\to\{0,1\}$, where 
%    $f(x):=
%    \begin{cases}
%        0, \text{ if } x\in T(\bar{t}')\cup X\setminus T(\bar{t})\\
%        1, \text{ if } x\in T(\bar{t})
%    \end{cases}$
%    is well-defined. Note that the above function is well-defined iff $T(\bar{t})\cap T(\bar{t}')=\emptyset$. 
% \end{proof}
%where $\bar{z}\neq\bar{t}$ is used as a shorthand notation for $\bigvee_{1\leq i\leq k+1}z_i\neq t_i$. We will also use an analogous shorthand notation $\bar{q}=\bar{t}$ for $\bigwedge_{1\leq i\leq k+1} q_i=t_i$. 
% Similarly, we can use the result concerning the first-order exclusion logic formula introduced in \cite{GALLIANI201268} to express a dependence atom with propositional exclusion atoms as follows
% $\dep(\bar{p};q)\equiv \forall z(z=q\lor \bar{p}z\mid \bar{p}q)$.

We now apply these equivalences to construct reduction formulas for exclusion atoms in $\QPL(\mid)$.
\begin{proposition}\label[proposition]{qexclu}
    Every $k+1$-ary propositional exclusion atom has an equivalent formula  $\phi\in \QPL(\mid_k)_4$.
\end{proposition}

\begin{proof}
We apply the translation in \Cref{eq3} to an exclusion atom $t_1\dots t_{k+1}\mid t'_1\dots t'_{k+1}$, together
with the reduction formula 
\[
\dep(p_1\dots p_{k+1};q)\equiv ((p_{k+1}\wedge \dep(p_1\dots p_{k}; q))\lor(\neg p_{k+1}\wedge \dep(p_1\dots p_{k}; q)))
\]
for propositional dependence atoms twice. First, to obtain
\begin{align*}
    t_1\dots t_{k+1}\mid t'_1\dots t'_{k+1}\equiv & \forall p_1\dots\forall p_{k+1}\exists q((\xi_{k}^+\lor\xi_{k}^-)\\
    &\wedge((q\wedge \bar{p}\neq \bar{t})\lor(\neg q\wedge \bar{p}\neq \bar{t}'))),
\end{align*}
and then
\begin{align*}
    t_1\dots t_{k+1}\mid t'_1\dots t'_{k+1}\equiv & \forall p_1\dots\forall p_{k+1}\exists q(((p_{k+1}\wedge(\xi_{k-1}^+\lor\xi_{k-1}^-))\\
    &\lor(\neg p_{k+1}\wedge(\xi_{k-1}^+\lor\xi_{k-1}^-)))\\
    &\wedge((q\wedge \bar{p}\neq \bar{t})\lor(\neg q\wedge \bar{p}\neq \bar{t}'))),
\end{align*}
where $\xi_{j}^+:=p_{j+1}\wedge\dep(p_1\dots p_j;q)$ and $\xi_{j}^-:=\neg p_{j+1}\wedge\dep(p_1\dots p_j;q)$.
Lastly, we can translate back to a formula with propositional exclusion atoms by using the equivalence
$\dep(\bar{p};q)\equiv \forall z(z=q\lor \bar{p}z\mid \bar{p}q)$. The reduction formula is then
\begin{align*}
    t_1\dots t_{k+1}\mid t'_1\dots t'_{k+1}\equiv \,\,& \forall p_1\dots\forall p_{k+1}\exists q(((p_{k+1}
    \wedge((p_k\wedge\phi)\lor(\neg p_k\wedge\phi))\\
    &\lor(\neg p_{k+1}\wedge(p_k\wedge\phi)\lor(\neg p_k\wedge\phi)))\\
    &\wedge(q\wedge \bar{p}\neq \bar{t})\lor(\neg q\wedge \bar{p}\neq \bar{t}')),
\end{align*}
where $\phi:=\forall z(z=q\lor p_1\dots p_{k-1}z\mid p_1\dots p_{k-1}q))$.
The formula is clearly in $\QPL(\mid_k)_4$.
\end{proof}

% \subsection{Relativized propositional inclusion and exclusion atoms}

%\ToDo{Add paragraph about expressivity of the quantified inc/exc logics.}

%Since $\PL(\dep(\dots))$ is expressively complete for all downward closed team properties with the empty team, and 

\subsection{Relativized inclusion and exclusion atoms} \label{section rel}

We introduce \emph{relativized} variants of inclusion and exclusion atoms, each having simple reduction formulas. We also show that propositional logic extended with each relativized atom is expressively complete. We end the section by discussing the lack of such expressive completeness results for the logics obtained from $\PL$ extended with the usual propositional inclusion and exclusion atoms.

Let $\bar{p}=p_1\dots p_k$ and $\bar{q}=q_1\dots q_k$ be sequences of propositional symbols and let $\alpha,\beta$ be formulas of propositional logic. Then $(\bar{p};\alpha)\subseteq(\bar{q};\beta)$ and $(\bar{p};\alpha)\mid(\bar{q};\beta)$ are relativized propositional inclusion and exclusion atoms with the following semantics: 
\begin{align*}
X\models(\bar{p};\alpha)\subseteq(\bar{q};\beta)      \, \text{ iff }& 
\text{for all } s\in X, \text{ if } s\models\alpha, \text{ then there is } s'\in X \\
&\text{  such that }s'(\bar{q})=s(\bar{p})  \text{ and }s'\models\beta.\\
X \models (\bar{p};\alpha)\mid(\bar{q};\beta)
 \text{ iff }& 
\text{for all } s,s' \in X,\ 
\text{if } s \models \alpha \text{ and } s' \models \beta,\ \\
&\text{then } s(\bar{p}) \neq s'(\bar{q}).
\end{align*}
%
%
% $X\models(p_1\dots p_k;\alpha)\subseteq(q_1\dots q_k;\beta)$ iff for all $s\in X$, if $s\models\alpha$, then there is $s'\in X$ such that $s'(q_i)=s(p_i)$ for $1\leq i\leq k$ and $s'\models\beta$.
%

We note that the relativized inclusion atoms are union closed and the relativized exclusion atoms are downward closed. We provide an example for the relativized inclusion atoms.

\begin{example}\label[example]{relincex}
%\matilda{Double check the text.}\minna{I think that this looks good, but should we try to make the table narrower so that it would fit the page better?} \matilda{made the table smaller}\minna{Thanks, it looks good now!} 
Consider \Cref{relincex_table} that describes for each role at a company whether it is part-time, requires previous experience in the team, and who is assigned to that role in the first and second year.

\begin{table}[h]   \centering\resizebox{1\columnwidth}{!}{
$\begin{array}{ccccc}
\toprule	
 \textit{role} & \textit{part-time} & \textit{experience} & \textit{year 1} & \textit{year 2}   \\ 
 \midrule 
\text{back office} &1 & 0 &  \text{Bob Jones} &     \text{Bob Jones}  \\ 
\text{customer service (phone)} &1 &0 & \text{Alice Wilson}&     \text{Alice Wilson}   \\ 
\text{customer service (chat)} &0 &0 & \text{Jane Smith}&     \text{Mary Brown}   \\ 
\text{team leader} &0 & 1 &  \text{John Williams} &     \text{Jane Smith}  \\ 
\text{customer experience advisor} &1 &1 & \text{Emma Baker} &  \text{Emma Baker}     \\ 
\bottomrule 
\end{array}$}
   \caption{The team used in \Cref{relincex}.
   } \label{relincex_table}
\end{table}

Let the team $X$ be defined over the variables $q_{prt},q_{exp},\bar{p}_1, \bar{p}_2$, with the roles of the table being its assignments. Based on the table, the propositional variable $q_{prt}$ states whether the role is part-time or not, and the propositional variable $q_{exp}$ states whether the role requires experience or not. Furthermore, $\bar{p}_1$ and $\bar{p}_2$ are tuples of propositional variables that encode the employee names in binary based on columns year 1 and year 2, respectively.

Consider the formula $\phi:=(\bar{p}_2;q_{exp})\subseteq(\bar{p}_1;\neg q_{prt}\lor q_{exp})$. The formula $\phi$ states that in the second year, the employees working in roles that require previous experience in the team must have worked in the first year either in a full-time role or in a role that requires experience.

To determine whether $X\models\phi$, it suffices to check the claim for the year 2 team leader Jane Smith and the year 2 customer experience advisor Emma Baker. Since in the first year, Jane Smith worked in a full-time role, and Emma Baker had a job that required experience, we conclude $X\models\phi$. 
% \bigbreak 
% \matilda{But this can be expressed with a usual excl atom right: Consider the formula $\psi:=(\bar{p}_1,q_{exp})\mid(\bar{p}_2,\neg q_{exp})$. ?? The formula $\psi$ states that employees in who worked a job requiring experience in their first year cannot work a job not requiring experience in their second year. Hence $X\models\psi$}
\end{example}

We state the reduction formulas for the relativized inclusion and exclusion atoms next, where each iteration reduces the arity by one.

\begin{proposition} Let $\alpha,\beta\in\PL$.
    \begin{enumerate}[label=(\roman*)]
        \item  $\begin{aligned}[t] 
(p_1\dots p_{k+1};\alpha)\subseteq\,&(q_1\dots q_{k+1};\beta)\equiv\\ &(p_1\dots p_{k};\alpha\wedge p_{k+1})\subseteq(q_1\dots q_{k};\beta\wedge q_{k+1})\\
&\wedge(p_1\dots p_{k};\alpha\wedge \neg p_{k+1})\subseteq(q_1\dots q_{k};\beta\wedge \neg q_{k+1}).
\end{aligned}$
        \item $\begin{aligned}[t]  
(p_1\dots p_{k+1};\alpha)\mid\,&(q_1\dots q_{k+1};\beta)\equiv\\ 
&(p_1\dots p_{k};\alpha\wedge p_{k+1})\mid(q_1\dots q_{k};\beta\wedge q_{k+1})\\
&\wedge(p_1\dots p_{k};\alpha\wedge \neg p_{k+1})\mid(q_1\dots q_{k};\beta\wedge \neg q_{k+1}).
\end{aligned}$
    \end{enumerate}
\end{proposition}

\begin{proof}Let $\bar{p}=p_1\dots p_k$ and $\bar{q}=q_1,\dots,q_k$.
    \begin{enumerate}[label=(\roman*)]
        \item  $\begin{aligned}[t] 
&T\models(\bar{p} p_{k+1};\alpha)\subseteq(\bar{q} q_{k+1};\beta)\\ 
%
%&\iff \forall s\in T, \text{ if } s\models \alpha \text{ then }  \exists  s'\in T \text{ s.t. } s'\models\beta  \text{ and } s(\bar{p} p_{k+1})=s'(\bar{q} q_{k+1}).\\
%
&\iff \forall s\in T, \text{ if } s\models \alpha \text{ then }  \exists s'\in T \text{ s.t. } s'\models\beta \text{ and } s(p_{k+1})=s'(q_{k+1})\\
&\quad\quad\quad\text{ and }  s(\bar{p})=s'(\bar{q}).\\
&\iff \forall s\in T, \text{ if } s\models \alpha\land p_{k+1} \text{ then }  \exists  s'\in T \text{ s.t. } s'\models\beta\land q_{k+1} \text{ and } \\
&\quad\quad\quad s(\bar{p})=s'(\bar{q}); \text{ and if } s\models \alpha\land \neg p_{k+1} \text{ then }  \exists  s'\in T \text{ s.t.}\\
&\quad\quad\quad s'\models\beta\land \neg q_{k+1} \text{ and } s(\bar{p})=s'(\bar{q}).\\
&\iff T\models (\bar{p};\alpha\wedge p_{k+1})\subseteq(\bar{q};\beta\wedge q_{k+1})\wedge(\bar{p};\alpha\wedge \neg p_{k+1})\subseteq(\bar{q};\beta\wedge \neg q_{k+1}).
\end{aligned}$

        \item $\begin{aligned}[t]  
&T\models (\bar{p} p_{k+1};\alpha)\mid(\bar{q}q_{k+1};\beta)\\ 
&\iff \forall s,s' \in T,\ 
\text{if } s \models \alpha \text{ and } s' \models \beta, \text{then } s(p_i) \neq s'(q_i) \text{ for some}\\
&\quad\quad\quad  1 \leq i \leq k+1. \\
&\iff \forall s,s' \in T, 
\text{ if } s \models \alpha\land p_{k+1}\text{ and } s' \models \beta\land q_{k+1}, \text{ then } s(\bar{p}) \neq s'(\bar{q}); \\
&\quad\quad\quad \text{ and if } s \models \alpha\land\neg p_{k+1}\text{ and } s' \models \beta\land\neg q_{k+1}, \text{then } s(\bar{p}) \neq s'(\bar{q}).\\
&\iff T\models(\bar{p};\alpha\wedge p_{k+1})\mid(\bar{q};\beta\wedge q_{k+1})\wedge(\bar{p};\alpha\wedge \neg p_{k+1})\mid(\bar{q};\beta\wedge \neg q_{k+1}).
\end{aligned}$
    \end{enumerate}    
\end{proof}

We have the following noteworthy equivalences for the fully reduced atoms: 
$(\emptyset;\alpha)\subseteq(\emptyset;\beta)\equiv \neg \alpha\lor\might\beta\equiv \neg \alpha\glor\might\beta$; and $(\emptyset;\alpha)\mid(\emptyset;\beta)\equiv \neg\alpha\glor \neg\beta$, where 
\begin{align*}
 &T \models \might \phi \text{ iff }  T = \emptyset \text{ or there is a nonempty } S \subseteq T \text{ such that } S \models \phi, \\
 &T\models\phi\glor\psi \text{ iff } T\models \phi \text{ or } T\models\psi.
\end{align*}
%

%{Consider relativized versions of dependence and anonymity, in particular with a conjunction of literals as the relativizing formulas.}

By extending $\PL$ with relativized inclusion and exclusion atoms, we obtain the logics $\PL(\subseteq^\mathsf{r})$ and  $\PL(\mid^\mathsf{r})$, respectively. We show that these logics are expressively equivalent with  $PL(\subseteq^{\top\bot})$ and ${\PL(\dep(\dots))}$, and thus expressively complete for all union closed properties with the empty team, and all nonempty downward closed properties, respectively. Recall from \cite{yang2022} that $PL(\subseteq^{\top\bot})$  is $\PL$ extended with the quasi upward closed  \emph{primitive inclusion atoms} ($\subseteq^{\top\bot}$)
of the form $\bar{x}\subseteq\bar{p}$, where $\bar{x}$ is a sequence of constants $\top,\bot$ and $\bar{p}$ is a sequence of distinct propositional symbols. Similarly, we introduce $\PL(\mid^{\top\bot})$ in which the \emph{extended exclusion atoms} can have both propositional symbols and the constants $\top,\bot$ within its `variable' sequences.

\begin{theorem} \label{expr compl exc inc}
 %Let $\PL(\subseteq^*)$ be the logic that extends $\PL$ with relativized inclusion atoms, and $\PL(\mid^*)$ the logic that extends $\PL$ with relativized exclusion atoms. Then
The logics $\PL(\subseteq^\mathsf{r})$ and $\QPL(\subseteq)$ are expressively complete for all union closed properties with the empty team. In particular, 
$$\PL(\subseteq^{\top\bot})\equiv \PL(\subseteq^\mathsf{r})\equiv\QPL(\subseteq)\equiv\PL(\Upsilon).$$

The logics $\PL(\mid^\mathsf{r})$ and $\QPL(\mid)$ are expressively complete for all downward closed properties with the empty team. In particular, 
$$ \PL(\mid^{\top\bot})\equiv \PL(\mid^\mathsf{r})\equiv\QPL(\mid)\equiv  \PL(\dep(\dots)).$$
\end{theorem}

\begin{proof}
It is easy to confirm that the logics $\PL(\subseteq^\mathsf{r})$ and $\QPL(\subseteq)$ are union closed and have the empty team property. We know from \cite{yang2022} that $\PL(\subseteq^{\top\bot})$ and $\PL(\Upsilon)$ are expressively complete for all union closed properties with the empty team. We obtain  $\PL(\Upsilon)\leq \QPL(\subseteq)$ by the translation in \Cref{eq2}, showing that propositional anonymity atoms can be expressed in $\QPL(\subseteq)$. 

Next, we show that $\PL(\subseteq^{\top\bot})\leq \PL(\subseteq^\mathsf{r})$. For $x\in\{\top,\bot\}$, define $q^{\top}=q$ and $q^{\bot}=\neg q$. By \cite{yang2022}, $\PL(\subseteq^{\top\bot})\equiv \PL(\subseteq^{\top\bot}_1)$, hence it suffices to express unary primitive inclusion atoms using relativized inclusion atoms, which we do by 
$(\emptyset;\top)\subseteq(\emptyset;q^x)\equiv \neg \top\lor\might q^x\equiv \might q^x$, where $\might q^\top\equiv\top\subseteq q$ and $\might q^\bot\equiv\bot\subseteq q$. 

Similarly, it is straightforward to check that the logics $\PL(\mid^{\top\bot})$, $\PL(\mid^\mathsf{r})$ and $\QPL(\mid)$ are downward closed and have the empty team property. Thus, it remains to show that the three logics are at least as expressive as the expressively complete logics for nonempty downward closed properties, ${\PL(\dep(\dots))}$ and ${\PL(\dep(\dots)_0)}$ \cite{Yang2017}. Using the translation in \Cref{eq4}, we obtain ${\PL(\dep(\dots))}\leq\QPL(\mid)$. Since 
$(\emptyset;p)\mid(\emptyset;\neg p)\equiv \neg p\glor \neg\neg p \equiv   \dep(p)\equiv \top p\mid p\bot$, we also have that
$\PL(\dep(\dots)_0)\leq\PL(\mid^\mathsf{r})$ and $\PL(\dep(\dots)_0)\leq\PL(\mid^{\top\bot})$, completing the proof. 
\end{proof}

We recall from \cite{yang2022} that $\PL(\subseteq)$ is not expressively complete for all union closed team properties with the empty team, since over a singleton set of propositional symbols, $\PL(\subseteq)\equiv\PL$, meaning only flat properties can be expressed. For $\PL(\mid)$, an analogous argument can be made. 
We conclude the section by giving an alternative proof without relying on locality to show that the inequalities in $\PL(\mid)\lneq \PL(\dep(\dots)_0)$ and $\PL(\subseteq)\lneq\PL(\subseteq^{\top\bot})$ are strict, by identifying a union closure-like property for $\PL(\mid)$, and a downward closure-like property for $\PL(\subseteq)$.

\begin{proposition} Let $\PL(\mid)$ and $\PL(\subseteq)$  be $\PL$ extended with exclusion/inclusion atoms over propositional symbols. \label[proposition]{inexpr inc exc}
    \begin{enumerate}[label=(\roman*)]
        \item $\PL(\mid)\lneq\PL(\dep(\dots)_0)$.
        \item $\PL(\subseteq) \lneq\PL(\subseteq^{\top\bot})$.\footnote{The proof of the second item is joint work with Pietro Galliani, and is presented here with his permission.}
    \end{enumerate}
\end{proposition}
\begin{proof} Let $\Prop'$ be a set of propositional symbols.
    \begin{enumerate}[label=(\roman*)]
    \item 
    Let $T=\{s_1,s_2\}$ be a team over $\Prop'$ with $s_1(r)=0$ for all $r\in \Prop'$, and $s_2(p)=1$ and $s_2(r)=0$ for all other $r\in \Prop'$.
\begin{table}[h]
    \centering
$\begin{array}{cccc ccc}
\toprule	
& p & q_1 & q_2 & \dots &q_i & \dots  \\ 
 \midrule
s_1&0 & 0 &0&     \dots &0& \dots \\ 
s_2&1 &0 &0&     \dots &0& \dots  \\ 
\bottomrule 
\end{array}$
   %\caption{} \label{inex_exclusion_table}
\end{table}

We claim that all $\phi\in\ \PL(\mid)$ have the following property: If $\{s_1\}\models \phi$ and $\{s_2\}\models \phi$, then $\{s_1,s_2\}\models \phi$. 

\begin{enumerate}
    \item If $\phi$ is some $r\in \Prop'$ or $\neg\alpha$, then the claim holds by union closure.
    \item If $\phi=u|v$, then clearly $\{s_1\}\not\models \phi$, so the claim vacuously holds. 
\item If $\phi=\psi\land\chi$ and we assume $\{s_1\}\models \psi\land\chi$ and $\{s_2\}\models \psi\land\chi$, then  $\{s_1\}\models \psi$ and $\{s_2\}\models \psi$, so by the induction hypothesis, $\{s_1,s_2\}\models \psi$. By the same argument $\{s_1,s_2\}\models \chi$, hence $\{s_1,s_2\}\models \phi$. 

\item If $\phi=\psi\lor\chi$ and we assume $\{s_1\}\models \psi\lor\chi$ and $\{s_2\}\models \psi\lor\chi$, then if both singleton teams satisfy the same disjunct, say $\{s_1\}\models \psi$ and $\{s_2\}\models \psi$, we have by the induction hypothesis that $\{s_1,s_2\}\models \psi$ and thus $\{s_1,s_2\}\models \psi\lor\chi$. Suppose instead that the singletons satisfy different disjuncts, say $\{s_1\}\models \psi$ and $\{s_2\}\models \chi$. Then by the semantic clause for disjunction $\{s_1,s_2\}\models \psi\lor\chi$.
\end{enumerate}

This property implies in particular that no formula in $\PL(\mid)$ can express $\dep(p)$, since $\{s_1\}\models\ \dep(p)$ and $\{s_2\}\models\ \dep(p)$, but $\{s_1,s_2\}\not\models\ \dep(p)$.
   
    \item 
     Let $T$ be a team over $\Prop'$ and let $s_1(r)=0$ for all $r\in \Prop'$.
We claim that all formulas $\phi\in\ \PL(\subseteq)$ have the following property: If $s_1\in T$ and $T\models \phi$, then $\{s_1\}\models \phi$.

    \begin{enumerate}
    \item If $\phi$ is some $r\in \Prop'$ the claim holds vacuously by $s_1\in T$. If $\phi=\neg\alpha$, then the claim holds by downward closure.
    \item If $\phi=u\subseteq v$, then clearly $\{s_1\}\models \phi$, so the claim trivially holds. 
\item If $\phi=\psi\land\chi$ and $T\models \psi\land\chi$, then  $T\models \psi$ and $T\models \psi$, so by the induction hypothesis, $\{s_1\}\models \psi$ and $\{s_1\}\models\chi$, hence $\{s_1\}\models \psi\land\chi$. 

\item If $\phi=\psi\lor\chi$ and $T\models \psi\lor\chi$, then we find $T_1,T_2\subseteq T$ such that $T_1\cup T_2=T$ and $T_1\models \phi$ and $T_2\models \chi$. W.l.o.g., suppose $s_1\in T_1$. Then by the induction hypothesis, $\{s_1\}\models \psi$ and we conclude $\{s_1\}\models \psi\lor\chi$.
\end{enumerate}
We conclude that no formula in $\PL(\subseteq)$ can express $\top\subseteq p$, since $\{s_1,s_2\}\models \top\subseteq p$ but $\{s_1\}\not\models \top\subseteq p$. 
\end{enumerate}
\end{proof}

\section{Inexpressibility results} \label{section 5}

%\minna{Was this supposed to be changed to "inexpressibility results" or what was it that Jouko suggested?}
We examine some translations using the upper dimension, and show that they are optimal in the following sense:  fewer occurrences of atoms of a lower arity are not sufficient in expressing the original formula. We obtain sharp results for dependence and anonymity atoms but only estimates for inclusion and exclusion atoms.
%\ToDo{Define optimal here, if needed?}\jouko{I have added a little more explanation above.} \minna{OK}

%For dependence atoms and anonymity, the translations are succinct, with the number of atoms only increasing by a factor of $2$. 
%
%We will later in \Cref{subsec inexpressibility} use dimensions to prove that this is optimal.
%, in the sense that a $k+1$-ary dependence or anonymity atom cannot be expressed with a formula that has fewer than two $k$-ary dependence or anonymity atoms, respectively. 

%\q{Should we mention the following, or some other similar results?}
%$\PL(\dep(\dots)_k)\equiv \PL(\dep(\dots)_{k+1})$

%[I think that it is natural to define the "length" as the number of dependence atom appearances, because $\text{D}(\lVert\phi\rVert^{m})=1$ for all $\phi\in\PL$.]

First, we obtain an upper bound for the upper dimension of a formula in $\phi\in\QPL(\circledast_k)_n$ by considering the upper dimension of the $\circledast$-atom and its number of occurrences in $\phi$.

\begin{proposition}\label[proposition]{aritydim} If $\phi\in\QPL(\circledast_k)_n$, then $\Dim(\lVert\phi\rVert^{\bar{p}})\leq (\Dim(\lVert{\circledast}\rVert^{\bar{p}}))^n$ where $\circledast$ is a $k$-ary atom from either $\dep(\dots)$, $\mid$, $\subseteq$, $\subseteq^{\top\bot}$ or $\Upsilon$.
\end{proposition}
\begin{proof}
     We show by induction on the structure of the formula that for all $\ell\in\mathbb{N}$, if a formula $\psi\in\QPL(\circledast_k)$ contains $\ell$ appearances of $\circledast$-atoms, then 
$\Dim(\lVert\psi\rVert^{\bar{p}})\leq(\Dim(\lVert\circledast\rVert^{\bar{p}}))^\ell$.
     \begin{itemize}
         \item If $\psi\in\QPL$, then the formula $\psi$ does not contain any $\circledast$-atoms, and $\Dim(\lVert\psi\rVert^{\bar{p}})=1 \leq(\Dim(\lVert\circledast\rVert^{\bar{p}}))^0$. 
         \item If %$\psi=\dep(p_{i_1}\dots p_{i_k},p_{i_m})$
         $\psi$ is a $\circledast$ atom, then the formula $\psi$ contains one $\circledast$ atom, and $\Dim(\lVert\psi\rVert^{\bar{p}})=(\Dim(\lVert\circledast\rVert^{\bar{p}})) \leq(\Dim(\lVert\circledast\rVert^{\bar{p}}))^1$.
        \item Suppose that $\psi=\psi_1\circ\psi_2$, where $\circ\in\{\wedge,\lor\}$, and each formula $\psi_j$ contains $\ell_j$ appearances of $\circledast$-atoms and $\Dim(\lVert\psi_j\rVert^{\bar{p}})\leq (\Dim(\lVert\circledast\rVert^{\bar{p}}))^{\ell_j}$ for $j\in\{1,2\}$. Then the formula $\psi$ contains $\ell_1+\ell_2$ appearances of $\circledast$-atoms and by \Cref{kripkeop}, $\Dim(\lVert\psi\rVert^{\bar{p}})\leq (\Dim(\lVert\circledast\rVert^{\bar{p}}))^{\ell_1}\cdot(\Dim(\lVert\circledast\rVert^{\bar{p}}))^{\ell_2}=(\Dim(\lVert\circledast\rVert^{\bar{p}}))^{\ell_1+\ell_2}$.  %\matilda{Do we have that it is an equality whenever the propositional symbols within the atoms form disjoint sets (I think it is used in the collapse results)?} \q{Minna: In thrm 4.5 we have a specific formula, so it is not required. I added more explanation to the thrm to make this clearer.}
         \item Suppose that $\psi=Qp_i\psi_1$, where $Q\in\{\exists,\forall\}$, and the formula $\psi_1$ contains $\ell$ appearances of $\circledast$-atoms and by \Cref{kripkeop}, $\Dim(\lVert\psi_1\rVert^{\bar{p}})\leq (\Dim(\lVert\circledast\rVert^{\bar{p}}))^{\ell}$. Then the formula $\psi$ contains $\ell$ appearances of $\circledast$-atoms and $\Dim(\lVert\psi\rVert^{\bar{p}})\leq\Dim(\lVert\psi_1\rVert^{\bar{p}})\leq (\Dim(\lVert\circledast\rVert^{\bar{p}}))^{\ell}$. 
     \end{itemize}
     Since every formula $\phi\in\QPL(\circledast_k)_n$ contains at most $n$ appearances of at most $k$-ary dependence atoms, the original claim follows from the claim that was proved by induction.
\end{proof}

% \ToDo{Some text here.}
Next, we show that the translations in \Cref{reductionfmlas} %\cref{reductionfmlas item dep,reductionfmlas item anon} 
for dependence and anonymity atoms are optimal in the sense that a $k+1$-ary dependence or anonymity atom cannot be expressed with a formula that has fewer than two $k$-ary dependence or anonymity atoms, respectively. 
% \matilda{Does it make more sense to have the statements of thm 5.2 and 5.3 for $\PL$, and have $\QPL$ as the `analogous claim' ? }
% \minna{Yeah, I think that would be better, I changed it. For the negative results we also have the stronger results: $\PL(\circledast_{k+1})_{n}\not\leq \QPL(\circledast_{k})_m$ for all $n,m$ such that $0<n\leq m <2n$, and $\PL(\circledast_{k})_{n}\not\leq \QPL(\circledast_{0})_{m}$ for all $n,m$ such that $0<n\leq m <n\cdot2^k$. Would it make sense to separate this thrm into the positive and negative results and have the stronger version for the neg results?}\matilda{Perhaps those mixed Pl-QPL statements could just be written in the text after the theorems, since they follow easily (also for 5.4 as you mentioned).}\minna{Added these in the text after the theorems.}

\begin{theorem}\label{opti}
   For all $k\geq 0$ and $\circledast\in \{\dep(\dots),\Upsilon\}$, 
   \begin{enumerate}[label=(\roman*)]
       \item $\PL(\circledast_{k+1})_{n}\leq \PL(\circledast_{k})_{2n}$ for all $n$,
       \item  \label{opti item a}$\PL(\circledast_{k+1})_{n}\not\leq \PL(\circledast_{k})_m$ for all $n,m$ such that $0<n\leq m <2n$,
       \item $\PL(\circledast_{k})_{n}\leq \PL(\circledast_{0})_{n\cdot2^k}$ for all $n$,
       \item \label{opti item b}$\PL(\circledast_{k})_{n}\not\leq \PL(\circledast_{0})_{m}$ for all $n,m$ such that $0<n\leq m <n\cdot2^k$.
       %for any $m<n\cdot2^k$. 
       %\ToDo{separate the cases and write instead QPL}
   \end{enumerate}
   Analogous claims also hold if quantification is allowed.
\end{theorem}
\begin{proof}
  %Since the reduction formulas in \Cref{reductionfmlas}  for dependence and anonymity atoms do not contain any quantifiers, the same result holds when restricted to the fragments, where quantification is disallowed. 
  We prove the first two claims for dependence atoms; the proof for anonymity atoms is analogous, using \Cref{reductionfmlas} \cref{reductionfmlas item anon}. The last two items follow straightforwardly from the proofs of the first two items.
\begin{enumerate} [label=(\roman*)]
    \item  Let $\phi\in\PL(\circledast_{k+1})_{n}$. By  \cref{reductionfmlas item dep} of \Cref{reductionfmlas}, each appearance of a $k+1$-ary dependence atom in $\phi$ can be replaced by a formula that contains two appearances of $k$-ary dependence atoms. Let  $\phi^*$ be a formula that is obtained by replacing the appearance in such a way. Since $\phi$ can contain at most $n$ appearances of $k+1$-ary dependence atoms, the formula $\phi^*$ can contain at most $2n$ appearances of $k$-ary dependence atoms, and hence $\phi^*\in\PL(\dep(\dots)_{k})_{2n}$.
    
    \item Let $n,m$ be as in \cref{opti item a} of the theorem. Define $\phi:=\bigwedge_{1\leq i\leq n}{\dep(p_{i_1}\dots p_{i_{k+1}};p_{i_j})}$, where none of the propositional variables appear more than once in the formula. Now $\phi\in\PL(\dep(\dots)_{k+1})_{n}$, and since $\PL(\dep(\dots)_{k+1})_{n}$ is downward closed, the upper dimension $\Dim(\lVert\phi\rVert^{\bar{p}})=|\Max(\lVert\phi\rVert^{\bar{p}})|$, i.e., the number of ways to pick an $n$-length sequence of functions from $\{0,1\}^{k+1}$ to $\{0,1\}$. Therefore, $\Dim(\lVert\phi\rVert^{\bar{p}})=(2^{2^{k+1}})^n$. Let $\psi\in\PL(\dep(\dots)_{k})_m$. Then by Theorem \ref{aritydim}, $\Dim(\lVert\psi\rVert^{\bar{p}})\leq(2^{2^{k}})^m<(2^{2^{k}})^{2n}=(2^{2^{k+1}})^n=\Dim(\lVert\phi\rVert^{\bar{p}})$, so it is impossible to express $\phi$ in $\PL(\dep(\dots)_{k})_m$. 
   % \item[(iii)\&(iv)] Follows from the previous item.
\end{enumerate}
Since Theorem \ref{aritydim} holds when quantifiers are allowed, it is easy to see that analogous claims also hold in that case.
\end{proof}
%
%Note that for \cref{opti item a} and \cref{opti item b}, we also get the following stronger statements: for all $k\geq 0$ and $\circledast\in \{\dep(\dots),\Upsilon\}$, we have $\PL(\circledast_{k+1})_{n}\not\leq \QPL(\circledast_{k})_m$ for all $n,m$ such that $0<n\leq m <2n$, and $\PL(\circledast_{k})_{n}\not\leq \QPL(\circledast_{0})_{m}$ for all $n,m$ such that $0<n\leq m <n\cdot2^k$.

Consequently, the below translation (\ref{translation anonymity}) from \cite{yang2022} for anonymity atoms is optimal.
The translation uses the quasi upward closed unary \emph{non-constancy atoms} $\not\dep(q)$, which are semantically equivalent to $0$-ary anonymity atoms $\langle\rangle\Upsilon q$.  
Another translation (\ref{translation inclusion}) from \cite{yang2022} for propositional inclusion atoms uses the quasi upward closed  primitive inclusion atoms. 
%
%($\subseteq^{\top\bot}$)
%of the form $\bar{x}\subseteq\bar{p}$, where $\bar{x}$ is a sequence of constants $\top,\bot$ and $\bar{p}$ is a sequence of distinct propositional symbols. 
%
%We recall the semantic clause for $n$-ary non-constancy atoms next. 
% %
% \begin{align*}
% T \models \not\dep(p_1\dots p_n) \text{ iff } &  T = \emptyset \text{ or there are } s_1,s_2\in T \text{ such that } s_1(p_i)\neq s_2(p_i) \text{ for} \\
% &\text{some $1\leq i\leq n$}.
% \end{align*}
%
%
For a sequence  $\bar{x}=x_1\dots x_k$, $x_i\in\{\top,\bot\}$ and $\bar{p}=p_1\dots p_k$, let $\bar{p}^{\bar{x}}$ be shorthand for $\bigwedge_{1\leq i\leq k} p_i^{x_i}$, where $p_i^{\top}=p_i$ and $p_i^{\bot}=\neg p_i$. We recall from \cite{yang2022},
\begin{align}\label{translation anonymity}
\bar{p}\ \Upsilon\ q\equiv &\bigvee_{\bar{x}\in\{\top,\bot\}^k} (\bar{p}^{\bar{x}}\land \not\dep(q)),\\
\label{translation inclusion}
    \bar{p}\subseteq \bar{q}\equiv &\bigwedge_{\bar{x}\in\{\top,\bot\}^{\bar{p}}}(\neg \bar{p}^{\bar{x}}\lor \bar{x}\subseteq \bar{q}).
\end{align}
We show that translation (\ref{translation inclusion}) is `optimal' for propositional inclusion atoms of arity $\geq 2$, in the sense that we cannot express a $k$-ary inclusion atom by a formula containing fewer than $2^k$ occurrences of $k$-ary primitive inclusion atoms. Note that in (\ref{translation inclusion}) the arity does not decrease but the primitive inclusion atoms on the right hand side of (\ref{translation inclusion}) are in principle simpler, being quasi upward closed, than the inclusion atoms on the left hand side of the equivalence.

%{Define uniform translation and note its impact in the result of thm 4.7.}
%\minna{The way that the proof of Theorem \ref{opti} is formulated could also give us $\PL(\Upsilon_k)_1\nleq\PL(\Upsilon_0)_m$ for any $m<2^k$, where $\Upsilon_0$ is the same as the non-constancy atom. An analogous point is mentioned for dependency/constancy right below the proof of Theorem \ref{opti}. Maybe we should explain this a little bit more there and also should mention this result for non-constancy?} 

%\matilda{I checked the result for the primitive inclusion translation and it really is only optimal when we consider uniform translations $(2^{(-1+n*2^k)} <(2^{(2^k)}-2^k)^n$ when k=2 and n=3 does not hold for instance).}

\begin{theorem}\label{collapse prim inc}
   For all $k,n\geq 1$, 
   \begin{enumerate}[label=(\roman*)]
 
       \item $\PL(\subseteq_{k})_{n}\leq \PL(\subseteq^{\top\bot}_{k})_{n\cdot 2^{{k}}}$.
       \item \label{collaps prim item} $\PL(\subseteq_{k})_{1}\not\leq \PL(\subseteq^{\top\bot}_{k})_{m}$ for any $m< 2^k$ whenever $k\geq 2$. %\matilda{I removed the `general' case, since I don't think it actually says any more than this one.}

   \end{enumerate}
\end{theorem}
\begin{proof}
  \begin{enumerate}[label=(\roman*)]
    \item Let $\phi\in\PL(\subseteq_{k})_{n}$ and replace each occurrence of a $k$-ary inclusion atom by  $2^{k}$ appearances of $k$-ary primitive inclusion atoms using the translation in (\ref{translation inclusion}) to obtain the formula $\phi^*$. Now $\phi^*$ contains at most $n\cdot 2^{k}$ appearances of $k$-ary primitive inclusion atoms, hence $\phi^*\in\PL(\subseteq^{\top\bot}_{k})_{n\cdot 2^{{k}}}$. %Since the reduction formula does not contain any quantifiers, the same result holds when restricted to the fragments, where quantification is disallowed. 

    \item %L $p_{1}\dots p_{{k}} \subseteq q_{1}\dots q_{{k}}$, where none of the propositional variables appear more than once in the atom. Now 
    Let $\phi\in\PL(\subseteq_{k})_{1}$ be a $k$-ary inclusion atom $\bar{p}\subseteq\bar{q}$ with distinct propositional symbols such that $\FV(\phi)\subseteq\Var(\bar{p}\bar{q})$. Now
    $\Dim(\lVert\phi\rVert^{\bar{p}\bar{q}})=(2^{2^{k}}-2^k)$. By \Cref{trivial dim} \cref{trivial dim quasi}, the upper dimension of quasi upward closed formulas, such as primitive inclusion atoms, is $2$. Suppose that there is a translation of the inclusion atom using $m< {2^{k}}$ occurrences of primitive inclusion atoms of arity at most $k$, and let $\phi^*$ be obtained by applying this translation to the inclusion atom in $\phi$. Now for all $k\geq 2$, $\Dim(\lVert\phi^*\rVert^{\bar{p}\bar{q}})\leq 2^m\leq2^{2^{k}-1}<2^{2^{k}}-2^k=\Dim(\lVert\phi\rVert^{\bar{p}\bar{q}})$. Hence, there is no such translation. % of inclusion atoms to a formula with fewer than $2^{2^{k}}$ appearances of primitive inclusion atoms of at most arity $k$.   
    %

        % \item Let $\phi:=\bigwedge_{1\leq i\leq n} p_{i_1}\dots p_{i_{k}} \Upsilon q_{i}$ be such that none of the propositional variables appear more than once in the formula. Now $\phi\in\PL(\Upsilon_{k})_{n}$ and $\Dim(\lVert\phi\rVert^{\bar{p}\bar{q}})=(2^{2^{k}})^n$. Suppose that there is a uniform translation of the anonymity atoms using $m< {2^{k}}$ occurrences of primitive anonymity atoms of arity at most $k$, and obtain $\phi^*$ by applying this translation to $\phi$. Then for all $k,n$, $\Dim(\lVert\phi^*\rVert^{\bar{p}\bar{q}})\leq (2^m)^n\leq(2^{2^{k}-1})^n<(2^{2^{k}})^n=\Dim(\lVert\phi\rVert^{\bar{p}\bar{q}})$, a contradiction. % Hence, there is no such uniform translation. 
\end{enumerate}
\end{proof}

% Let $k\geq 2$ and $\phi\in\PL(\subseteq_{k})_{n}$. Then there is no formula $\phi'$ obtained from translating each inclusion atom within $\phi$ using fewer than $2^{{k}}$ occurrences of primitive inclusion atoms of arity at most $k$ such that $\phi'\equiv \phi$. In particular, 

    % \item[] $\PL(\Upsilon_{k})_{n}\leq \PL(\Upsilon_{0})_{n\cdot2^k}$ and $\PL(\Upsilon_{k})_{n}\not\leq \PL(\Upsilon_{0})_{m}$ for any $m<n\cdot2^k$. 

%$QPL(\dep(\dots)_k)\equiv QPL(\dep(\dots)_{k+1})$

%\q{For anonymity, the results are the same. Add a theorem.}
%\ToDo{Mention the fact that we need the quantifiers for the positive results.}
%\ask{About thrm 5.4. We also have the stronger results $\PL(\circledast_{2})_n\not\leq\QPL(\circledast_1)_m$ for all $n,m$, s.t., $0<n\leq m< 4n$, and $\PL(\circledast_{k+1})_n\not\leq\QPL(\circledast_k)_m$ for all $k\geq 2$ and $n,m$, s.t., $0<n\leq m< 3n$, but would it make sense to write this version or does it looks funny?}
From \Cref{theoremk+1tok4,qexclu}, it follows that $\QPL(\subseteq_{k+1})_n\leq\QPL(\subseteq_k)_{4n}$  and $\QPL(\mid_{k+1})_n\leq\QPL(\mid_k)_{4n}$, respectively. Note that our translations use quantifiers, so we do not obtain analogous results for quantifier-free fragments. From the first item of the following theorem, it follows that a reduction formula for a binary inclusion/exclusion atom
%, i.e., $\subseteq_2$,
%, a reduction formula similar to that of Theorem \ref{theoremk+1tok4}
must contain at least four unary inclusion/exclusion atoms, so in the binary case, the reduction formulas in \Cref{theoremk+1tok4,qexclu} are optimal with respect to the number of unary inclusion/exclusion atoms. For the other arities of inclusion atoms and for any exclusion atoms, we do not have such optimality results.

\begin{theorem} \label{opti inc exc}
Let $\circledast\in\{\subseteq,\mid\}$. We then have
\begin{itemize}
    \item[(i)] $\QPL(\circledast_{2})_n\not\leq\QPL(\circledast_1)_m$ for all $n,m$, s.t., $0<n\leq m\leq 3n$,
    \item[(ii)] $\QPL(\circledast_{k+1})_n\not\leq\QPL(\circledast_k)_m$ for all $k\geq 2$ and $n,m$, s.t., $0<n\leq m \leq 2n$.
\end{itemize}  
\end{theorem}
\begin{proof}
We prove the first item for inclusion and the second for exclusion. The remaining cases for the items are analogous.
    \begin{itemize}
        \item[(i)] Let $n,m$ be as in item (i) of the theorem. Define $\phi:=\bigwedge_{1\leq i\leq n}{p_{i_1} p_{i_2}\subseteq p_{i_3}p_{i_4}}$, where none of the propositional variables appears more than once in the formula. Now $\phi\in\QPL(\subseteq_{2})_{n}$, and the upper dimension $\Dim(\lVert\phi\rVert^{\bar{p}})$ can be calculated using the observation that since each propositional symbol appears only once, the smallest family that dominates $\lVert\phi\rVert^{\bar{p}}$ is obtained by taking the family of Cartesian products $\mathcal{G}:=\{\prod_{1\leq i\leq n} G_i: G_i\in\mathcal{G}_i\text{ for all } i\}$, where each $\mathcal{G}_i$ is the smallest family dominating $\lVert{p_{i_1} p_{i_2}\subseteq p_{i_3}p_{i_4}}\rVert^{\bar{p}}$. The fact that this is the smallest dominating family can be checked analogously the case of $\Finc$. Therefore, $\Dim(\lVert\phi\rVert^{\bar{p}})=(2^{2^{2}}-2^2)^n=12^n$. Let $\psi\in\QPL(\subseteq_1)_m$. Then by Theorem \ref{aritydim}, $\Dim(\lVert\psi\rVert^{\bar{p}})\leq(2^{2^{1}}-2^1)^m=2^{m}\leq 2^{3n}=8^n<12^n=\Dim(\lVert\phi\rVert^{\bar{p}})$, so it is impossible to express $\phi$ in $\QPL(\subseteq_1)_m$.
        \item[(ii)] Let $k,n,m$ be as in item (ii) of the theorem. Define $\phi:=\bigwedge_{1\leq i\leq n}{p_{i_1}\dots p_{i_{k+1}}\mid p_{i_1}'\dots p'_{i_{k+1}}}$, where none of the propositional variables appear more than once in the formula. Now $\phi\in\QPL(\mid_{k+1})_{n}$, and the upper dimension $\Dim(\lVert\phi\rVert^{\bar{p}})=|\Max(\lVert\phi\rVert^{\bar{p}})|$, i.e., the number of ways to pick an $n$-length sequence of sets from $\mathcal{P}(\{0,1\}^{k+1})\setminus\{\emptyset,\{0,1\}^{k+1}\}$. Therefore, $\Dim(\lVert\phi\rVert^{\bar{p}})=(2^{2^{k+1}}-2)^n$. Let $\psi\in\QPL(\mid_k)_m$. Then by Theorem \ref{aritydim}, $\Dim(\lVert\psi\rVert^{\bar{p}})\leq(2^{2^{k}}-2)^m\leq(2^{2^{k}}-2)^{2n}=(2^{2^{k+1}}-4(2^{2^k}-1))^{n}<(2^{2^{k+1}}-2)^n=\Dim(\lVert\phi\rVert^{\bar{p}})$, so it is impossible to express $\phi$ in $\QPL(\mid_k)_m$.
    \end{itemize}
\end{proof}

% \ToDo{Write the following properly:}

Note that in \Cref{opti} \cref{opti item a,opti item b}, as well as in \Cref{collapse prim inc} \cref{collaps prim item}, the stronger statements with $\PL$ replaced by $\QPL$ on the right-hand side of the inequality hold. Similarly, in both items of \Cref{opti inc exc}, the stronger statements obtained by replacing $\QPL$ with $\PL$ on the left-hand side of the inequality also hold.

\section{Additional dimension calculations}
\label{section 6}

In \Cref{genatoms}, we present calculations for the upper dimension of various extended propositional atoms. In \Cref{subsec: trs quasi}, we focus on the dual upper dimension of some (quasi) upward closed atoms from the literature.

\subsection{Extended atoms}\label{genatoms}

Extended atoms, introduced in \cite{Ebbing2013}, allow formulas in the sequences of `variables', and can, in certain settings, increase the expressivity further than their propositional counterparts. This is the case for extended inclusion atoms, used in \cite{yang2022} and \cite{Anttila}, and for exclusion atoms of the form $\top p\mid p\bot$ as observed in \Cref{expr compl exc inc}. % and \Cref{inexpr inc exc}.
%
%For exclusion, we observe $\PL(\mid)$ over one propositional symbol is expressively equivalent to $\PL$, while extended exclusion atoms of the form $\top p\mid p\bot$ are semantically equivalent to $\not\dep(p)$, thus increasing the expressivity. 
In the case of ${\PL(\dep(\dots))}$ and $\PL(\Upsilon)$, extending the syntax with the respective extended variants do not increase the expressivity, since propositional dependence and anonymity logic are already expressively complete logics \cite{MR3488885,yang2022}.

%for all downward closed team properties with the empty team \cite{MR3488885}, and for all union closed team properties with the empty team \cite{yang2022}, respectively.

% {Similar result for inclusion?}

Let $\bar{\alpha}$ and $\bar{\beta}$ be sequences of $\PL$-formulas. The semantic clauses for the extended dependence, anonymity, inclusion and exclusion atoms are presented below. 
\begin{align*}
    T\models\ \dep(\bar{\alpha};\beta) \text{ iff}& \text{ for all } s,s'\in T  \text{ such that } s(\bar{\alpha})=s'(\bar{\alpha}),\text{ we have } s(\beta)=s'(\beta). \\
T\models\bar{\alpha}\Upsilon\beta \text{ iff}& \text{ for all }  s\in T, \text{ there is } s'\in T  \text{ such that } s(\bar{\alpha})= s'(\bar{\alpha}) \text{ and } \\
&s(\beta)\neq s'(\beta). \\
T\models\bar{\alpha}\subseteq\bar{\beta}\text{ iff}& \text{ for all }  s\in T,  \text{ there is }  s'\in T \text{ such that }  s(\bar{\alpha})= s'(\bar{\beta}). \\
T\models\bar{\alpha}\mid\bar{\beta}\text{ iff}& \text{ for all } s,s'\in T, s(\bar{\alpha})\neq s'(\bar{\beta}). 
\end{align*}
    
Note that allowing extended $\circledast$-atoms in $\QPL(\circledast)$ for any $\circledast\in\{{\dep(\dots),}{\Upsilon,} {\subseteq,} {\mid\}}$ does not increase the expressivity of the logic, since Theorem \ref{expr compl exc inc} shows that allowing quantification makes the corresponding logic expressively complete also in the cases of inclusion and exclusion. 
%\matilda{Add a comment that this point can also be made using the expressive completeness results of each quantified logic.}\minna{Added now something like this.}\matilda{Good!}
This can also be seen by a direct translation constructed as follows.
Let $\alpha,\beta\in\PL$. We write $\alpha\leftrightarrow\beta$ as an abbreviation for the formula $(\alpha\wedge\beta)\lor(\neg\alpha\wedge\neg\beta)$. Then an extended $\circledast$ atom $\phi(\alpha_1,\dots,\alpha_k)$ can be expressed as $\exists q_1\dots\exists q_m((\bigwedge_{1\leq i\leq m}(q_i\leftrightarrow\alpha_i)) \wedge\phi(q_1,\dots,q_m))$, where $q_1,\dots,q_m$ are propositional variables that do not appear in $\phi(\alpha_1,\dots,\alpha_k)$. For example, the extended dependency atom  $\dep(\alpha_1\dots \alpha_k;\beta)$ can be exressed as $\exists q_1\dots\exists q_k\exists u((\bigwedge_{1\leq i\leq k}(q_i\leftrightarrow\alpha_i))\wedge (u\leftrightarrow\beta)\wedge\dep(q_1\dots q_k;u))$.

Recall that for a normal $k$-ary dependence atom (with propositional variables instead of general formulas), the upper dimension, i.e., the number of maximal teams that satisfy the atom, is the number of all functions from $\{0,1\}^k$ to $\{0,1\}$. 
Consider then 
the extended atom $\theta:=\ \dep(\alpha_1\dots \alpha_k;\beta)$
. As the formula $\theta$ is downward closed, the upper dimension is again the number of maximal teams that satisfy the atom. Since $\alpha_1\dots \alpha_k,\beta$ are formulas instead of propositional variables, 
%it might be that 
some of the functions might not be available because the evaluations that would correspond to them are not possible for the formulas. If e.g. $\beta=p_1\wedge\neg p_1$, then clearly $s(\beta)=0$ for all $s\in\Eval$, and therefore every assignment evaluates the formulas $\alpha_1,\dots,\alpha_k,\beta$ such that the maximal satisfying team corresponds to the constant function 0. (Extended) anonymity atoms are not downward closed, but analogous observations hold for the critical sets. We formulate these observations in the theorem below.

\begin{theorem}\label{genFD}
    Let $\theta$ be the atom $\dep(\alpha_1\dots \alpha_k;\beta)$ or the atom $\alpha_1\dots \alpha_k\Upsilon\beta$, where $\alpha_1\dots \alpha_k,\beta\in\PL$, and $\Var(\alpha_1)\cup\dots\cup\Var(\alpha_k)\cup\Var(\beta)=
    %\{p_1\dots p_l\}
    \Var(\bar{p})$.
    Define 
%     \[\begin{array}{l}
%     A=\{(x_1,\dots,x_k)\in\{0,1\}^{k}\mid \exists x_{k+1}\in\{0,1\}\text{ s.t. }\\
%     \qquad\forall s\in\Eval, 
% %s\not\models\alpha_1^{x_1}\wedge\dots\wedge\alpha_k^{x_k}\wedge\beta^{x_{k+1}}
% s(\alpha_1^{x_1}\wedge\dots\wedge\alpha_k^{x_k}\wedge\beta^{x_{k+1}})=0
% \},\end{array}\]
%
%
  \[
  A = \left\{ (x_1,\dots,x_k)\in\{0,1\}^{k} \ \middle\vert \begin{array}{l}
      \exists x_{k+1}\in\{0,1\}\text{ s.t. } \forall s\in\Eval\\ s(\alpha_1^{x_1}\wedge\dots\wedge\alpha_k^{x_k}\wedge\beta^{x_{k+1}})=0  
  \end{array}\right\},
\]
    where $\alpha^{1}=\alpha$ and $\alpha^{0}=\neg\alpha$. Let $m=2^k-|A|.$ Then $\Dim(\lVert\theta\rVert^{\bar{p}})= 2^m$.
\end{theorem}
\begin{proof}
  Consider first the case that $\theta$ is the atom $\dep(\alpha_1\dots \alpha_k;\beta)$. The formula $\theta$ is downward closed, so $\Dim(\lVert\theta\rVert^{\bar{p}})=|\Max(\lVert\theta\rVert^{\bar{p}})|$. Define the set $A$ as in the theorem, and let $(x_1,\dots,x_k)\in A$. If there are no $s\in\Eval$ such that $s(\alpha_i)=x_i$ for all $1\leq i\leq k$, then $(x_1,\dots,x_k)$ cannot be included in the domain of the functions that we are counting. Suppose then that $s\in\Eval$ is such that $s(\alpha_i)=x_i$ for all $1\leq i\leq k$, and let $s(\beta)=x_{k+1}$. Then $s(\alpha_1^{x_1}\wedge\dots\wedge\alpha_k^{x_k}\wedge\beta^{x_{k+1}})=1$, and by the definition of $A$, it must be that $s'(\alpha_1^{x_1}\wedge\dots\wedge\alpha_k^{x_k}\wedge\beta^{1-x_{k+1}})=0$ for all $s\in\Eval$. Let $s'\in\Eval$ be such that $s'(x_i)=s(x_i)$ for all $1\leq i\leq k$. Then $s'(\alpha_1^{x_1}\wedge\dots\wedge\alpha_k^{x_k})=1$. Since $s'(\alpha_1^{x_1}\wedge\dots\wedge\alpha_k^{x_k}\wedge\beta^{1-x_{k+1}})=0$, it must be that $s'(\beta^{1-x_{k+1}})=0$, i.e., $s'(\beta)=x_{k+1}=s(\beta)$. This means that any function must always evaluate $(x_1,\dots,x_k)$ as $x_{k+1}$, and we can ignore the tuple $(x_1,\dots,x_k)$ when we count the number of different functions. Hence $m=2^k-|A|$ represents the number of $k$-tuples for which both 0 or 1 are options for values of a function. Then the number of possible functions is $2^m$, i.e., $\Dim(\lVert\theta\rVert^{\bar{p}})= 2^m$ as wanted.

The case for the atom $\alpha_1,\dots,\alpha_k\Upsilon\beta$, is similar, but instead of $\Max(\lVert\theta\rVert^{\bar{p}})$, we consider the set $\Crit(\lVert\theta\rVert^{\bar{p}})$. The proof for the claim $\Dim(\lVert\theta\rVert^{\bar{p}})=|\Crit(\lVert\theta\rVert^{\bar{p}})|$ can then be done as in Theorem \ref{dim_families} Item (ii) by letting $X=\{0,2\}^k\setminus A$ and $Y=\{0,1\}$. 
\end{proof}
The following claims are straightforward consequences of the above theorem.
\begin{corollary}
   
   Let $\theta$ be as in Theorem \ref{genFD}. Then the following hold.
      \begin{enumerate}[label=(\roman*)] 
     \item  $\Dim(\lVert\theta\rVert^{\bar{p}})\leq 2^{2^k}$.
    \item If each $\alpha_1,\dots,\alpha_k,\beta$ is contingent, and $\Var(\alpha_1),\dots,\Var(\alpha_k),\Var(\beta)$ are pair-wise disjoint, then $\Dim(\lVert\theta\rVert^{\bar{p}})=2^{2^k}$.
    %\ToDo{The $\alpha_i$ should be removed from the atom, fix this.}
    \item If $\alpha_i$ is a tautology or a contradiction, then $\Dim(\lVert\theta\rVert^{\bar{p}})$ is $\Dim(\lVert{\dep(\alpha_1\dots\alpha_{i-1}\alpha_{i+1}\dots \alpha_k;\beta)}\rVert^{\bar{p}})$ or $\Dim(\lVert{\alpha_1\dots\alpha_{i-1}\alpha_{i+1}\dots \alpha_k\Upsilon\beta}\rVert^{\bar{p}})$, depending on whether $\theta$ is $\dep(\alpha_1\dots \alpha_k;\beta)$ or $\alpha_1\dots \alpha_k\Upsilon\beta$.
    \item[(iv)] If $\beta$ is a tautology or a contradiction, then $\Dim(\lVert\theta\rVert^{\bar{p}})=1$.
\end{enumerate} 
\end{corollary}
%\q{The case of anonymity is the same as FD, so maybe we want to combine these?}
%\begin{theorem}\label{genANO}
 %   Let $\theta$ be the atom $\alpha_1\dots \alpha_k\Upsilon\beta$, where $\alpha_1\dots \alpha_k,\beta\in\PL$, and $\Var(\alpha_1)\cup\dots\cup\Var(\alpha_k)\cup\Var(\beta)=
   %  \Var(\bar{p})$.
  %  Define 
 % \[
 % A = \left\{ (x_1,\dots,x_k)\in\{0,1\}^{k} \ \middle\vert \begin{array}{l}
     % \exists x_{k+1}\in\{0,1\}\text{ s.t. } \forall s\in\Eval\\ s(\alpha_1^{x_1}\wedge\dots\wedge\alpha_k^{x_k}\wedge\beta^{x_{k+1}})=0  
 % \end{array}\right\},
%\]
  %  where $\alpha^{1}=\alpha$ and $\alpha^{0}=\nnf(\neg\alpha)$. Let $m=2^k-|A|.$ Then $\Dim(\lVert\theta\rVert^{\bar{p}})= 2^m$.
%\end{theorem}

The situation is similar for extended exclusion atoms, which are downward closed like extended dependence atoms. For the normal $k$-ary exclusion atom the maximal satisfying teams are of the form $A\times(\{0,1\}^k\setminus A)$, where $\emptyset\neq A\subsetneq \{0,1\}^k$. For an extended atom, some of these maximal teams might be impossible to construct because some truth values are not possible for the formulas.
For example, consider the extended atom $p_1\mid  p_2\wedge\neg p_2$ and the set $A=\{0\}\subseteq\{0,1\}$. Now $A\times(\{0,1\}\setminus A)=\{(0,1)\}$. The tuple $(0,1)$ can be viewed as the evaluation where the first element of the tuple corresponds to the truth evaluation of the formula $p_1$ and the second element to that of the formula $p_2\wedge\neg p_2$. This evaluation does not satisfy the extended exclusion atom, because the truth value of the formula $p_2\wedge\neg p_2$ can only be 0. We formulate these observations in the theorem below.
\begin{theorem}
    Let $\theta:=\alpha_1\dots \alpha_k\mid\beta_1\dots \beta_k$, where $\alpha_1,\dots,\alpha_k,\beta_1,\dots,\beta_k\in\PL$, and $\bigcup_{1\leq i\leq k}\Var(\alpha_i)\cup\bigcup_{1\leq i\leq k}\Var(\beta_i)=
    %\{p_1\dots p_l\}
     \Var(\bar{p})$. Let $\{P_j\mid j\in\{1,\dots,2^{2^k}-2\}\}=\mathcal{P}(\{0,1\}^k)\setminus\{\emptyset,\{0,1\}^k\}$, and define for each $j$,
    %the sets $A_j^l=P_j$, $A_j^r=\{0,1\}^k\setminus P_j$, and  $A_j=A_j^l\times A_j^r$, and let
    the set $A_j=P_j\times(\{0,1\}^k\setminus P_j)$. Let
\begin{align*}
     B &= \Max\left(\left\{ D\subseteq A_j \ \middle\vert \begin{array}{l}
    j\in\{1,\dots,2^{2^k}-2\},
     \forall\bar{x}\bar{y}\in A_j, \bar{x}\bar{y}\in D \text { iff }\exists s\in\Eval, \\ s(\bigwedge_{1\leq i\leq k}\alpha_i^{x_i}\wedge\bigwedge_{1\leq i\leq k}\beta_i^{y_i})=1  
  \end{array}\right\}\right)
\end{align*}
     where $\alpha^{1}=\alpha=$ and $\alpha^{0}=\neg\alpha$. 
   Then %$\Dim(\lVert\theta\rVert^{\bar{p}})=\max(\{1,|B|\})$.
$\Dim(\lVert\theta\rVert^{\bar{p}})=|B|$.
\end{theorem}

\begin{proof}
    The formula $\theta$ is downward closed, so again $\Dim(\lVert\theta\rVert^{\bar{p}})=|\Max(\lVert\theta\rVert^{\bar{p}})|$. Define the set $B$ as in the theorem. The idea of the set $B$ is that we go through all the sets $A_j$ and remove from each $A_j$ exactly those tuples that do not consist of a possible truth evaluation for the formulas $\alpha_1,\dots,\alpha_k,\beta_1,\dots,\beta_k$. Note that although $A_j\nsubseteq A_{j'}$ for any $j\neq j'$, after removal of the tuples, we might have $D$ and $D'$ such that $D\subsetneq D'$, so it is necessary to remove the sets that are not maximal.
\end{proof}

We give an example of an extended inclusion atom and examine it through its upper dimension. 

\begin{proposition}
Let $x_i, \tilde{x_i}\in\{0,1\}$ and $x_i\neq \tilde{x_i}$, $1\leq i\leq k$, then
   $\Dim(\lVert p_{1}^{x_1}\dots p_{k}^{x_k}\subseteq p_{1}^{\tilde{x}_1}\dots p_{k}^{\tilde{x}_k}\rVert^{\bar{p}})=2^{2^{k-1}}$. 
\end{proposition}
\begin{proof}
For all assignments $s$ and propositional symbols $p_i$ such that $s(p_i)=s(x_i)$, define $\tilde{s}(p_i)=s(\tilde{x_i})$. Now
    $\lVert p_{1}^{x_1}\dots p_{k}^{x_k}\subseteq p_{1}^{\tilde{x}_1}\dots p_{k}^{\tilde{x}_k}\rVert^{\bar{p}}=\{T\subseteq \Eval^{\bar{p}}\mid T=\emptyset \text{ or if } s\in T \text{ then } \tilde{s}\in T\}$. All convex subfamilies consist of exactly one team, hence the upper dimension is the size of the family, $2^{2^{k-1}}$, where $2^{k-1}$ is the number of pairs $(s,\tilde{s})$ over $\bar{p}$.
\end{proof}

%\ToDo{Write a short comment about the lack of inclusion result (should we say anything about independence?).}

The general case of the upper dimension for extended inclusion atoms is more complicated because when some evaluations are not possible for the formulas appearing in the atom, it affects the satisfaction of the atom. Therefore, some critical sets for the usual propositional inclusion atom of the corresponding arity might not satisfy the extended version of the atom. However, Section \ref{subsec: trs quasi} contains some results for the dual upper dimension of some fragments of extended inclusion atoms.

\subsection{(Quasi) upward closed formulas} \label{subsec: trs quasi}

%We have so far focused the concrete dimension calculations of formulas on their upper dimension, but as recognized in \Cref{trivial dim} \cref{trivial dim quasi}, this trivializes the dimension of certain properties, e.g., quasi upward closed properties. 

We examine some (quasi) upward closed atoms from the literature and calculate their dual upper dimension, while recalling from \Cref{trivial dim} that the upper dimension is $2$ for quasi upward closed formulas, and $1$ for upward closed formulas.

We have already covered some quasi upward closed formulas from the literature, such as the $0$-ary anonymity atom corresponding to a unary non-constancy atom. We further recall the semantic clause for $n$-ary non-constancy atoms. 
\begin{align*}
T \models\ \not\dep(p_1\dots p_n) \text{ iff } &  T = \emptyset \text{ or there are } s_1,s_2\in T \text{ such that } s_1(p_i)\neq s_2(p_i) \text{ for} \\
&\text{some $1\leq i\leq n$}.
\end{align*}

We have seen primitive inclusion atoms of the form $\bar{x}\subseteq\bar{p}$, and consider \emph{extended} primitive inclusion atoms of the form $\bar{x}\subseteq\bar{\alpha}$ where $\bar{x}$ is a sequence of constants from $\{\top,\bot\}$, and $\bar{\alpha}$ is a sequence of $\PL$-formulas.
%A particular type of primitive inclusion atom is $\top\subseteq\alpha$. 
Related notions include the \emph{singleton might} operator $\singlemight$ from \cite{Anttila}, and the \emph{might} operator $\might$ from \cite{HS}.
%(where it is called the nonemptiness operator). 
We recall their semantic clauses.
\begin{align*}
T \models \singlemight \phi &\text{ iff } T = \emptyset \text{ or there is } s \in T \text{ such that } \{s\} \models \phi, \\
T \models \might \phi &\text{ iff }  T = \emptyset \text{ or there is a nonempty } S \subseteq T \text{ such that } S \models \phi.
\end{align*}

Might and singleton might formulas are thus quasi upward closed. Over flat formulas, the three notions coincide: $\top\subseteq\alpha\equiv\singlemight\alpha\equiv\might\alpha$. 
%However, we do allow general formulas in the scope of each might operator. 

%We calculate the dual upper dimensions of some quasi upward closed formulas, where, in particular, primitive inclusion atoms and unary non-constancy atoms have the dual upper dimension $2$. 
To obtain meaningful dimension calculations, we assume that the formulas are contingent.

\begin{theorem}\label{uc thm} Let $\bar{p}=p_1\dots p_n$ and $n=|\bar{p}|$. We state the dual upper dimension of $\Dim^d(\mathcal{C})$ for quasi upward closed team properties $\mathcal{C}\subsetneq P(\Eval^{\bar{p}})$. 
   \begin{enumerate}[label=(\roman*)]
  \item $\Dim^d(\lVert{\not\dep(p_1\dots p_n)}\rVert^{\bar{p}})=\frac{2^n(2^n-1)}{2}+1$. In particular, $\Dim^d(\lVert{\not\dep(p_1)}\rVert^{{p_1}})=2$. \label{uc item nonconstancy}
    \item $\Dim^d(\lVert x_{1}\dots x_{k}\subseteq p_{1}\dots p_{k}\rVert^{\bar{p}})= 2$.\label{uc item primitive inc} %=|\Min((\lVert x_{1}\dots x_{k}\subseteq p_{1}\dots p_{k}\rVert^{\bar{p}})|+1
    
      \item $\Dim^d(\lVert x_{1}\dots x_{k}\subseteq \alpha_{1}\dots \alpha_{k}\rVert^{\bar{p}})\leq 2^{n}$. \label{uc item 1} %=|\Min(\lVert x_{1}\dots x_{k}\subseteq \alpha_{1}\dots \alpha_{k}\rVert^{\bar{p}})|+1
    
    \item $\Dim^d(\lVert \singlemight \phi\rVert^{\bar{p}}) \leq 2^{n}$. %= |\Min(\lVert \singlemight \phi\rVert^{\bar{p}})|+1
    \label{uc item 3}

     \item $\Dim^d(\lVert \might \phi\rVert^{\bar{p}})\leq \frac{2^{n}}{2^{{n}-1}!(2^{n}-2^{{n}-1})!}+1.$ \label{uc item 4} %= |\Min(\lVert \might \phi\rVert^{\bar{p}})|+1 
\end{enumerate}
\end{theorem}
\begin{proof}
By quasi upward closure and \Cref{quasi dc and uc prop}, it suffices to calculate the number of quasi minimal sets for the family of teams satisfying the formula.
\begin{enumerate}[label=(\roman*)]
\item By definition,
    $\lVert{\neq(p_1\dots p_n)}\rVert^{\bar{p}}=\{T\subseteq \Eval^{\bar{p}}\mid T=\emptyset \text{ or } |T|\geq2\}.$ There are $\frac{2^n(2^n-1)}{2}$ such sets of cardinality $2$, hence $\Dim^d(\lVert \neq(p_1\dots p_n)\rVert^{\bar{p}})=|\Min^q(\lVert \neq(p_1\dots p_n)\rVert^{\bar{p}})| = \frac{2^n(2^n-1)}{2}+1$.

      \item Denote $x_{1}\dots x_{k}\subseteq p_{1}\dots p_{k}$ by $\bar{x}\subseteq \bar{p}$. There is a unique assignment $s$ for which $\{s\}\models \bar{x}\subseteq \bar{p}$, and all nonempty $T\in \lVert{\bar{x}\subseteq \bar{p}}\rVert^{\bar{p}}$ are such that $s\in T$. Hence $\Min^q(\lVert{\bar{x}\subseteq \bar{p}}\rVert^{\bar{p}})=\{\emptyset,\{s\}\}$ and the dual upper dimension is $\Dim^d(\lVert \bar{x}\subseteq \bar{p}\rVert^{\bar{p}})=|\Min^q(\lVert \bar{x}\subseteq \bar{p}\rVert^{\bar{p}})|=2$.
      
     % This is a special case of \cref{uc item 1}, when there is only one unique assignment $s$ such that $\{s\}\models\bigwedge_{1\leq i\leq k}\p_i^{x_i}$. 
    
   \item   %Let $M$ be the full team over $\bar{p}$.
Denote $x_{1}\dots x_{k}\subseteq \alpha_{1}\dots \alpha_{k}$ by $\bar{x}\subseteq \bar{\alpha}$.   We have 
    $$\lVert{\bar{x}\subseteq \bar{\alpha}}\rVert^{\bar{p}}=\{T\subseteq \Eval^{\bar{p}}\mid T=\emptyset \text{ or }\text{there is }s\in T\text{ with } \{s\}\models\bigwedge_{1\leq i\leq k}\alpha_i^{x_i}\},$$
 where $\alpha_i^\top=\alpha_i$ and $\alpha_i^\bot=\neg\alpha_i$. There are $2^{n}$ such singleton sets, and by the assumption that $\bar{x}\subseteq\bar{\alpha}$ is contingent, the dual upper dimension is $\Dim^d(\lVert \bar{x}\subseteq \bar{\alpha}\rVert^{\bar{p}})=|\Min^q(\lVert \bar{x}\subseteq \bar{\alpha}\rVert^{\bar{p}})|\leq 2^{n}$.

%$|\Min^q(\lVert x\subseteq \alpha\rVert^{\bar{p}})|=2^{\bar{p}}-1$. 

    %
  \item This case is analogous to \cref{uc item 1} in general, and to \cref{uc item primitive inc} if $\phi$ is a conjunction of literals, due to the translation $\singlemight (p_1^{x_1}\land\dots\land p_n^{x_n})\equiv x_1\dots x_n\subseteq p_1\dots p_n$. 
  
   \item We calculate $\Dim^d(\lVert \might\phi\rVert^{\bar{p}})=|\Min^q(\lVert \might\phi\rVert^{\bar{p}})| \leq \frac{2^{n}}{2^{{n}-1}!(2^{n}-2^{{n}-1})!}+1$, where the upper bound is the largest number of different teams of the same size, together with the empty team.
     \end{enumerate}
\end{proof}

We end the section by calculating the dual upper dimension of some upward closed formulas. First, recall the semantics of the nonemptiness atom $\NE$ (see, e.g., \cite{YANG20171406}): $T\models \NE$ if and only if $T\neq\emptyset$. Clearly, the nonemptiness atom is upward closed. Furthermore the dual upper dimension $\Dim^d(\lVert\NE\rVert^{\bar{p}})$ is calculated in \cite{Hella_Luosto_Vaananen_2024} to be the number of singleton teams, i.e., $2^{|\bar{p}|}$.

Expressions such as `might $p$' can linguistically be interpreted as true only when there is a witness to $p$ being true. This conflicts with the empty team property of $\might p$. To avoid this vacuous satisfaction of quasi upward closed formulas, we \emph{pragmatically enrich} them by taking a conjunction with the non-emptiness atom $\NE$, as in \cite{AloniSP2022}. The pragmatically enriched formula is then upward closed.  In particular, we obtain the formula $\might\phi\land \NE$, equivalent to $\blackdiamond\phi$, where $\blackdiamond$ is the \emph{epistemic might} operator used in \cite{anttila2025convexteamlogics,Hornung}. 

\begin{corollary} Let  $\bar{x}=x_1\dots x_n$, $\bar{\alpha}=\alpha_1\dots \alpha_n$, $\bar{p}=p_1\dots p_n$ and $n=|\bar{p}|$. We state the dual upper dimension of $\Dim^d(\mathcal{C})$ for upward closed team properties $\mathcal{C}$ without the empty team.

\begin{align*}
&\Dim^d(\lVert{\not\dep(\bar{p})\land \NE}\rVert^{\bar{p}})=\frac{2^n(2^n-1)}{2}, \quad \Dim^d(\lVert{\not\dep(p_1)\land \NE}\rVert^{{p_1}})=1,\\
   &\Dim^d(\lVert \bar{x}\subseteq \bar{p}\land \NE\rVert^{\bar{p}})= 1,\quad \Dim^d(\lVert \bar{x}\subseteq \bar{\alpha}\land \NE\rVert^{\bar{p}})\leq 2^{n},\\
   &    \Dim^d(\lVert \singlemight \phi\land \NE\rVert^{\bar{p}}) \leq 2^{n},  \quad\text{and}\quad \Dim^d(\lVert \might \phi\land \NE\rVert^{\bar{p}})\leq \frac{2^{n}}{2^{{n}-1}!(2^{n}-2^{{n}-1})!}.
\end{align*}
\end{corollary}
\begin{proof}
    The dual upper dimensions are calculated as in \Cref{uc thm}, with the difference that the empty team is not in any of the families defined by the formulas, and that the non-contingency assumption in \Cref{uc thm} is now trivially satisfied. 
\end{proof}

\section{Conclusion and future work}
\label{section 7}

%\ToDo{Write the conclusion.}\jouko{I will write the conclusion.}

We set out wanting to prove inexpressibility results in propositional team logic by means of  dimension calculations. The situation with propositional logic turns out to be different than with first-order logic. Team theoretical atoms \emph{can} actually be defined in terms of atoms of smaller arity but complexity of the formula increases in the sense that the defining formula gets longer with more and more occurrences of the atoms of lower arity. We used the dimension concept to quantify this increase of complexity leading us to a sequence of succinctness results. At the same time we ended up bringing into the picture some features that go beyond the strict confines of propositional logic, such as extended propositional logic and quantified propositional logic.

%\bigbreak

Next, we provide some directions for future work.

The upper and dual upper dimensions are both concerned with convex subfamilies, leading to natural ways of calculating the dimension of (quasi) convex families. In particular, we observe a duality between the upper dimension for (quasi) downward closed properties and the dual upper dimension for (quasi) upward closed properties. In the literature, we find expressively complete logics for all nonempty downward closed team properties \cite{MR3488885}. To enrich this picture, we could introduce propositional team-based logics that are expressively complete for (quasi) upward closed, and (quasi) downward closed team properties. It would be of particular interest from a conceptual standpoint if this can be done in a way that, in some sense, mirrors the duality observed by considering the corresponding team properties through the lens of the dimensions.

Outside convexity, other natural closure properties include union and intersection closure. Potentially, optimality results for, e.g., union closed atoms, could more easily be studied under an alternative definition of dimension that does not rely on convex subfamilies. This motivates us to define a natural dimension for union closed families $\mathcal{A}$: $\Dim^u(\mathcal{A})$ is the smallest cardinality of a subfamily $\mathcal{G}\subseteq\mathcal{A}$ such that $\bar{\cup}\mathcal{G}=\mathcal{A}$, where $\bar{\cup}\mathcal{G}=\{A\mid A=\bigcup\mathcal{G'} \text{ for some } \mathcal{G}'\subseteq\mathcal{G}\}$; and for intersection closed families $\mathcal{B}$: $\Dim^i(\mathcal{B})$ is the smallest cardinality of a subfamily $\mathcal{H}\subseteq\mathcal{B}$ such that $\underline{\cap}\mathcal{H}=\mathcal{B}$, where $\underline{\cap}\mathcal{H}=\{B\mid B=\bigcap\mathcal{H'} \text{ for some } \mathcal{H}'\subseteq\mathcal{H}\}$.

 % Results like this: \matilda{Note that the strict inequality at the end of the optimality proof for inclusion becomes $=$ whenever $k=1$, meaning that the unary case is not covered. Can we prove that such a translation is not possible in some other way?}

Using the upper dimension, we obtained sharp inexpressibility results for dependence and anonymity atoms concerning the arity of the atoms and the number of their occurrences. Similar optimality proofs were also established for certain arities of inclusion and exclusion atoms, where the translations made use of propositional quantifiers. We conjecture that the quantifiers are necessary, and that there are no reduction formulas for propositional inclusion/exclusion logic, i.e., $\PL(\subseteq_k)_n\not\leq \PL(\subseteq_{k-1})_m$ and $\PL(\mid_k)_n\not\leq \PL(\subseteq_{k-1})_m$ for any natural number $m$. 

% Furthermore, as observed in \cite{Hagg24,Hagg25}, the complete systems for first-order exclusion atoms \cite{Hagg24} and anonymity atoms \cite{Vaananen2022} would need to be extended to axiomatize their propositional counterparts. On the other hand, the completeness proofs for the systems for dependence and non-conditional independence \cite{Galliani2014}, as well as inclusion atoms \cite{Hagg25}, only use two values, making the systems complete in the propositional setting. 

%The logics $\PL(\Upsilon)$ and $\PL(\mid^{\top\bot})$ have not yet been axiomatized. 

Furthermore, the logics $\PL(\subseteq)$ and $\PL(\mid)$ are strictly less expressive than their quantified counterparts $\QPL(\subseteq)$ and $\QPL(\mid)$ and the exact expressive powers of $\PL(\subseteq)$ and $\PL(\mid)$ are unknown. One interesting feature, that might have a connection to expressive limitations of the logics, is that the semantics of the inclusion and exclusion atoms have a \emph{global} nature in contrast with dependence and anonymity atoms. Namely, we have that $\dep(\bar{p};q)\equiv \bigvee_{\bar{{x}} \in 2^{\bar{p}}} (\bar{p}^{\bar{{x}}}\land \dep(\bar{p};q))$ and $\bar{p}\Upsilon q\equiv \bigvee_{\bar{{x}}\in 2^{\bar{p}}}(\bar{p}^{\bar{{x}}}\land \bar{p}\Upsilon q)$, while $\bigvee_{\bar{{x}} \in 2^{\bar{p}}} (\bar{p}^{\bar{{x}}}\land \bar{p}\mid\bar{q})\not\models \bar{p}\mid\bar{q}$ and $\bar{p}\subseteq\bar{q}\not\models \bigvee_{\bar{{x}}\in 2^{\bar{p}}}(\bar{p}^{\bar{{x}}}\land \bar{p}\subseteq\bar{q})$.

Lastly, the upper dimension of the independence atom and the upper bounds given by the number of occurrences do not allow us to obtain the corresponding results for independence atoms in relation to their reduction formula; thus, a more careful consideration of inexpressibility results for independence atoms is left as future work.

\bigbreak

%\section{Notes and to-do list (to be deleted)}

%\ToDo{Consider "lower dimension" instead of dual upper dimension. Timon: Light cone for dual convex shadows. }

%\matilda{We now have $\alpha,\beta$ used for $\PL$ formulas in the sections where that's relevant. Notation mostly unified.}

\noindent\textbf{Acknowledgements.} We thank Pietro Galliani, Lauri Hella, \AA sa Hirvonen, Kerkko Luosto,  Marius Tritschler, and  Fan Yang for helpful conversations on the topics of this paper. This project has  received  funding from the European Research Council (ERC) under the
European Union’s Horizon 2020 research and innovation programme (grant agreement No
101020762) and from the Research Council of Finland, grant number 368671. The first author was partially supported by the Research Council of Finland under grant number 359650. The second author is supported by the Magnus Ehrnrooth Foundation.% was given some suggestions by Timon that were implemented, also discussed some results about the atoms in more detail with Nicolas, used Marius' `quasi' convexity (check if published), discussed the main definitions with Lauri and Kerkko...}

\bibliographystyle{plain}
\bibliography{mybib}

@article {MR2480819,
    AUTHOR = {Abramsky, Samson and V\"a\"an\"anen, Jouko},
     TITLE = {From {IF} to {BI}: a tale of dependence and separation},
   JOURNAL = {Synthese},
  FJOURNAL = {Synthese. An International Journal for Epistemology,
              Methodology and Philosophy of Science},
    VOLUME = {167},
      YEAR = {2009},
    NUMBER = {2},
     PAGES = {207--230},
      ISSN = {0039-7857,1573-0964},
   MRCLASS = {03B60 (03B47 03C80 03G25)},
  MRNUMBER = {2480819},
MRREVIEWER = {Micha\l\ Krynicki},
       DOI = {10.1007/s11229-008-9415-6},
}

@article{Hella_Luosto_Vaananen_2024, title={Dimension in team semantics}, volume={34}, DOI={10.1017/S0960129524000021}, number={5}, journal={Mathematical Structures in Computer Science}, author={Hella, Lauri and Luosto, Kerkko and Väänänen, Jouko}, year={2024}, pages={410–454}}

@incollection {HLSV,
    AUTHOR = {Hella, Lauri and Luosto, Kerkko and Sano, Katsuhiko and
              Virtema, Jonni},
     TITLE = {The expressive power of modal dependence logic},
 BOOKTITLE = {Advances in modal logic. {V}ol. 10},
     PAGES = {294--312},
 PUBLISHER = {Coll. Publ., London},
      YEAR = {2014},
   MRCLASS = {03B45},
}

@inproceedings {HS,
    author = {Hella, Lauri and Stumpf, Johanna},
     title = {The expressive power of modal logic with inclusion atoms},
 booktitle = {Proceedings {S}ixth {I}nternational {S}ymposium on {G}ames,
              {A}utomata, {L}ogics and {F}ormal {V}erification},
    series = {Electron. Proc. Theor. Comput. Sci. (EPTCS)},
    volume = {193},
     pages = {129--143},
      year = {2015},
       doi = {10.4204/EPTCS.193.10},
}

@article{yang2022,
	title = {Propositional union closed team logics},
	volume = {173},
	issn = {01680072},
	doi = {10.1016/j.apal.2022.103102},
	language = {en},
	number = {6},
	
	journal = {Annals of Pure and Applied Logic},
	author = {Yang, Fan},
	month = jun,
	year = {2022},
	pages = {103102},
}

@misc{anttila2025convexteamlogics,
      title={Convex Team Logics}, 
      author={Anttila, Aleksi and Knudstorp, Søren},
      year={2025},
      eprint={2503.21850},
      archivePrefix={arXiv},
      primaryClass={math.LO}, 
note={Manuscript}
}

@mastersthesis{Hornung,
    author = {Lorenz Hornung},
    title = {Reasoning About Legal Concepts with Propositional Dependence Logic},
    school = {University of Amsterdam},
    year = {2025}
}

@mastersthesis{ciardelli09,
    author = {Ivano  Ciardelli},
    title = {Inquisitive Semantics and Intermediate Logics},
    school = {University of Amsterdam},
    year = {2009},
}

@article{YANG20171406,
title = {Propositional team logics},
journal = {Annals of Pure and Applied Logic},
volume = {168},
number = {7},
pages = {1406-1441},
year = {2017},
issn = {0168-0072},
doi = {10.1016/j.apal.2017.01.007},
author = {Fan Yang and Jouko Väänänen}
}

@article{AloniSP2022,
  author = {Aloni, Maria},
  title = {Logic and conversation: {The} case of free choice},
  journal = {Semantics and Pragmatics},
  volume = {15},
  number = {5},
  month = jun,
  pages = {1--60},
  doi = {10.3765/sp.15.5},
  year = {2022}
}

@article {MR3488885,
    AUTHOR = {Yang, Fan and V\"a\"an\"anen, Jouko},
     TITLE = {Propositional logics of dependence},
   JOURNAL = {Ann. Pure Appl. Logic},
  FJOURNAL = {Annals of Pure and Applied Logic},
    VOLUME = {167},
      YEAR = {2016},
    NUMBER = {7},
     PAGES = {557--589},
      ISSN = {0168-0072,1873-2461},
   MRCLASS = {03B60 (03B55 03B65 03B70)},
  MRNUMBER = {3488885},
MRREVIEWER = {Fredrik\ Engstr\"om},
       DOI = {10.1016/j.apal.2016.03.003},
}

@article{Anttila,
title = "Axiomatizing modal inclusion logic and its variants",
author = "Anttila, Aleksi and H{\"a}ggblom, Matilda and Fan Yang",
year = "2025",
month = jul,
doi = "10.1007/s00153-024-00957-y",
language = "English",
volume = "64",
pages = "755--793",
journal = "Archive for Mathematical Logic",
issn = "0933-5846",
publisher = "Springer",
number = "5-6",
}

@article{Hannula16,
author = {Hannula, Miika and Kontinen, Juha and Lück, Martin and Virtema, Jonni},
year = {2016},
month = {09},
pages = {198-212},
title = {On Quantified Propositional Logics and the Exponential Time Hierarchy},
volume = {226},
journal = {Electronic Proceedings in Theoretical Computer Science},
doi = {10.4204/EPTCS.226.14}
}

@article{Yang2017,
    author = {Yang, Fan},
    title = {Modal dependence logics: axiomatizations and model-theoretic properties},
    journal = {Logic Journal of the IGPL},
    volume = {25},
    number = {5},
    pages = {773-805},
    year = {2017},
    month = {08},
    issn = {1367-0751},
    doi = {10.1093/jigpal/jzx023},
}

@InProceedings{Ebbing2013,
author="Ebbing, Johannes
and Hella, Lauri
and Meier, Arne
and M{\"u}ller, Julian-Steffen
and Virtema, Jonni
and Vollmer, Heribert",
editor="Libkin, Leonid
and Kohlenbach, Ulrich
and de Queiroz, Ruy",
title="Extended Modal Dependence Logic",
booktitle="Logic, Language, Information, and Computation",
year="2013",
address="Berlin, Heidelberg",
pages="126--137",
doi= {10.1007/978-3-642-39992-3_13},
}

@article{Galliani2015,
title = {Upwards closed dependencies in team semantics},
journal = {Information and Computation},
volume = {245},
pages = {124-135},
year = {2015},
issn = {0890-5401},
doi = {10.1016/j.ic.2015.06.008},
author = {Pietro Galliani},
}

@book{ronnholm18,
title = "Arity Fragments of Logics with Team Semantics",
author = "Raine R{\"o}nnholm",
year = "2018",
language = "English",
isbn = "978-952-03-0912-1",
series = "Acta Electronica Universitatis Tamperensis",
publisher = "Tampere University Press",
number = "1955",
}

@book{galliani12,
title = "The Dynamics of Imperfect Information.",
author = "Pietro Galliani",
year = "2012",
language = "English",
isbn = "9789090269320",
series = "ILLC Dissertation
Series",
publisher = "Amsterdam: Institute for Logic, Language and Computation",
}

@article{GALLIANI201268,
title = {Inclusion and exclusion dependencies in team semantics — On some logics of imperfect information},
journal = {Annals of Pure and Applied Logic},
volume = {163},
number = {1},
pages = {68-84},
year = {2012},
issn = {0168-0072},
doi = {10.1016/j.apal.2011.08.005},
author = {Pietro Galliani}
}

@InProceedings{hannula17,
  author =	{Hannula, Miika},
  title =	{{Validity and Entailment in Modal and Propositional Dependence Logics}},
  booktitle =	{26th EACSL Annual Conference on Computer Science Logic (CSL 2017)},
  pages =	{28:1--28:17},
  series =	{Leibniz International Proceedings in Informatics (LIPIcs)},
  ISBN =	{978-3-95977-045-3},
  ISSN =	{1868-8969},
  year =	{2017},
  volume =	{82},
  editor =	{Goranko, Valentin and Dam, Mads},
  address =	{Dagstuhl, Germany},
  URN =		{urn:nbn:de:0030-drops-76691},
  doi =		{10.4230/LIPIcs.CSL.2017.28}
}

@article{DBLP:journals/lmcs/LuckV19,
  author       = {Martin L{\"{u}}ck and
                  Miikka Vilander},
  title        = {On the Succinctness of Atoms of Dependency},
  journal      = {Log. Methods Comput. Sci.},
  volume       = {15},
  number       = {3},
  year         = {2019},
  doi          = {10.23638/LMCS-15(3:17)2019},
  bibsource    = {dblp computer science bibliography, https://dblp.org}
}

\section{Appendix} \label{app1}

For the convenience of the reader, we include here some proofs omitted from the main part of the paper. These proofs are essentially  from \cite{Hella_Luosto_Vaananen_2024}, where they were considered in the setting of first-order logic.\\

\noindent\textbf{\Cref{dim and locality prop}.}[\cite{Hella_Luosto_Vaananen_2024}]
Let $\phi\in\QPL(\circledast)$ be such that $\FV(\phi)\subseteq\Var(\bar{p})\subseteq\Var(\bar{p}')$.
\begin{enumerate}[label=(\roman*)]
    \item   Then $\Dim(\lVert\phi\rVert^{\bar{p}})=\Dim(\lVert\phi\rVert^{\bar{p}'})$.

    \item  %Let $\phi\in\QPL(*)$ be such that $\FV(\phi)\subseteq\Var(\bar{p})\subseteq\Var(\bar{p}')$. 
    Then $\Dim^d(\lVert\phi\rVert^{\bar{p}})\leq2^{n}\cdot\Dim^d(\lVert\phi\rVert^{\bar{p}'})$, where 
    \[
    n=\begin{cases}
        0, \text{ if } \lVert\phi\rVert^{\bar{p}} \text{ is convex,}\\
        |\Var(\bar{p}')\setminus\Var(\bar{p})|\cdot\max\{|T|\mid T\in\lVert\phi\rVert^{\bar{p}}\}, \text{ otherwise}.
    \end{cases}
    \]
   % and $\mathcal{H}$ is minimal supporting subfamily of $\lVert\phi\rVert^{\bar{p}}$
\end{enumerate}

\begin{proof}
  \begin{enumerate}[label=(\roman*)]
    \item   It suffices to show the claim for any $\bar{p}$ and $\bar{p}'$ such that $\FV(\phi)\subseteq\Var(\bar{p})\subseteq\Var(\bar{p}')$ and $\Var(\bar{p}')\setminus\Var(\bar{p})=\{p_i\}$. First, note that by locality of $\QPL(\circledast)$, $\lVert\phi\rVert^{\bar{p}}=\{T\restriction\Var(\bar{p})\mid T\in\lVert\phi\rVert^{\bar{p}'}\}$ and $\lVert\phi\rVert^{\bar{p}'}=\{T[F/p_i]]\mid F\colon T\to\{\{0\},\{1\},\{0,1\}\}, T\in\lVert\phi\rVert^{\bar{p}}\}$. 
    
    Suppose that $\mathcal{G}\subseteq\lVert\phi\rVert^{\bar{p}}$ is a minimal subfamily dominating $\lVert\phi\rVert^{\bar{p}}$. We show that then $\mathcal{G}':=\{G[\{0,1\}/p_i]\mid G\in\mathcal{G}\}\subseteq\lVert\phi\rVert^{\bar{p}'}$ is a minimal subfamily dominating $\lVert\phi\rVert^{\bar{p}'}$. Let $\mathcal{D}_{G'}=\partial_{G'}(\lVert\phi\rVert^{\bar{p}'})$ for all $G'\in\mathcal{G}'$. We show that $\lVert\phi\rVert^{\bar{p}'}\subseteq\bigcup_{G'\in\mathcal{G}'}\mathcal{D}_{G'}$. If $T\in\lVert\phi\rVert^{\bar{p}'}$, then $T=T'[F/p_i]$ for some  $F\colon T'\to\{\{0\},\{1\},\{0,1\}\}$ and  $T'\in\lVert\phi\rVert^{\bar{p}}$. Since $\mathcal{G}$ dominates $\lVert\phi\rVert^{\bar{p}}$, there is $G\in\mathcal{G}$ such that $T'\subseteq G$ and $[T',G]\subseteq\lVert\phi\rVert^{\bar{p}}$. By locality, we have $[T,G[\{0,1\}/p_i]]\subseteq\lVert\phi\rVert^{\bar{p}'}$, so $T\in\mathcal{D}_{G'}$ for $G'=G[\{0,1\}/p_i]\in\mathcal{G}'$. It is straightforward to check the rest of the conditions for dominating subfamily, so we conclude that $\mathcal{G}'$ dominates $\lVert\phi\rVert^{\bar{p}'}$.
    
    It remains to show that $\mathcal{G}'$ is minimal among the dominating subfamilies. Suppose for a contradiction that there is $\mathcal{G}_{min}'\subsetneq\mathcal{G}'$ that dominates $\lVert\phi\rVert^{\bar{p}'}$. Let $\mathcal{G}_{min}:=\{G\restriction\Var(\bar{p})\mid G\in\mathcal{G}_{min}'\}$. Then $\mathcal{G}_{min}\subsetneq\mathcal{G}$, and it is straightforward to show that $\mathcal{G}_{min}$ dominates $\lVert\phi\rVert^{\bar{p}}$, contradicting the assumption that $\mathcal{G}$ is a minimal dominating subfamily. By the definition of $\mathcal{G}'$, we clearly have $|\mathcal{G}'|=|\mathcal{G}|$.

    Note that we could analogously show that if $\mathcal{G}\subseteq\lVert\phi\rVert^{\bar{p}'}$ is a minimal subfamily dominating $\lVert\phi\rVert^{\bar{p}'}$, then $\mathcal{G}':=\{G\restriction\Var(\bar{p})\mid G\in\mathcal{G}\}$ is a minimal subfamily dominating $\lVert\phi\rVert^{\bar{p}}$. We would also need to show that $|\mathcal{G}'|=|\mathcal{G}|$, but this is easy: if $T\restriction\Var(\bar{p})=T'\restriction\Var(\bar{p})$ for some $T,T'\in\mathcal{G}$, then $T,T'\subseteq(T\restriction\Var(\bar{p}))[\{0,1\}/p_i]$, and by locality and $\mathcal{G}\subseteq\Crit(\lVert\phi\rVert^{\bar{p}'})$, we have $T=(T\restriction\Var(\bar{p}))[\{0,1\}/p_i]=T'$.

    \item If $\lVert\phi\rVert^{\bar{p}}$ is convex, then also $\lVert\phi\rVert^{\bar{p}'}$ is convex, and by the empty team property of $\QPL(\circledast)$ and Lemma \ref{minmax and crit dim lemma} item (iii), $\Dim^d(\lVert\phi\rVert^{\bar{p}})=|\Min(\lVert\phi\rVert^{\bar{p}})|=|\{\emptyset\}|=|\Min(\lVert\phi\rVert^{\bar{p}'})|=\Dim^d(\lVert\phi\rVert^{\bar{p}'})$.
    Suppose then that $\lVert\phi\rVert^{\bar{p}}$ is not convex. Let $\bar{p}$, $\bar{p}'$, and $\bar{p}''$ be such that $\FV(\phi)\subseteq\Var(\bar{p})\subseteq\Var(\bar{p}')$, $\Var(\bar{p}')\setminus\Var(\bar{p})=\Var(p)''$, and $|\bar{p}''|=k$.
    Let $\mathcal{G}\subseteq\lVert\phi\rVert^{\bar{p}}$ is a minimal subfamily supporting $\lVert\phi\rVert^{\bar{p}}$. Then a proof analogous to the proof of item (i) shows that
    \[
    \mathcal{G}':=\{G[F/\bar{p}'']\mid F\colon G\to\{\{a\}\mid a\in\{0,1\}^k\},G\in\mathcal{G}\}\subseteq\lVert\phi\rVert^{\bar{p}'}
    \]
    is a subfamily supporting $\lVert\phi\rVert^{\bar{p}'}$. Since $|\mathcal{G}'|\leq 2^{k\cdot\max\{|T| \ \mid \ T\in\lVert\phi\rVert^{\bar{p}}\}}$, the claim follows.
    \end{enumerate}
\end{proof}

\noindent\textbf{\Cref{dim_families}.}[\cite{Hella_Luosto_Vaananen_2024}]
Let $X$ and $Y$ be as in Definition \ref{sets}, and assume that $|X|=\ell\geq 2$ and $|Y|=n\geq 2$. Then the following claims hold for the upper dimensions of the families. 
% \begin{align*}
%    \Dim(\Fdep)=n^\ell, \quad&\Dim(\Fano)=2^\ell,\quad\Dim(\Finc)=2^\ell-\ell,\\\quad\Dim(\Fexc)=2^\ell-2,\quad\text{and}\quad
%    &\Dim(\Find)=(2^\ell-\ell-1)(2^n-n-1)+\ell+n.
% \end{align*}
   \begin{itemize} 
       \item[(i)] $\Dim(\Fdep)=n^\ell$, 
       \item[(ii)] $\Dim(\Fano)=2^\ell$,
       \item[(iii)] $\Dim(\Finc)=2^\ell-\ell$,
       \item[(iv)] $\Dim(\Fexc)=2^\ell-2$,
       \item[(v)] $\Dim(\Find)=(2^\ell-\ell-1)(2^n-n-1)+\ell+n$.
\end{itemize}

\begin{proof}
    \begin{itemize}
        \item[(i)] First note that  $\Fdep$ is downward closed, %and \textcolor{red}{finite}, 
        so by \Cref{dc and uc prop}, $\Dim(\Fdep)=|\Max(\Fdep)|$. The maximal sets in $\Fdep$ are total functions $f\colon X\to Y$, so $\Dim(\Fdep)=|Y|^{|X|}=n^\ell$. 
        \item[(ii)] %We show that $\Crit(\Fano)=\{A\times Y\mid A\subseteq X\}$ the smallest subfamily dominating $\Fano$, so $\Dim(\Fano)=|\{A\times Y\mid A\subseteq X\}|=2^\ell$.

We show that $\Crit(\Fano)=\{A\times Y\mid A\subseteq X\}$ the smallest subfamily dominating $\Fano$, so $\Dim(\Fano)=|\{A\times Y\mid A\subseteq X\}|=2^\ell$.
        
        Let $R\subseteq R'\subseteq X\times Y$. Then by the definition of $\Fano$, $[R,R']\subseteq\Fano$ if and only if $R,R'\in\Fano$ and $\Dom(R)=\Dom(R')$. 
         Since $\partial_{R'}(\Fano)=\{S\subseteq R'\mid [S,R']\subseteq\Fano\}$, this means that for any $R,R'\in\Fano$, $R\in \partial_{R'}(\Fano)$ if and only if $R\subseteq R'$ and $\Dom(R)=\Dom(R')$. Since $R\in \partial_{R'}(\Fano)$ if and only if $\partial_R(\Fano)\subseteq\partial_{R'}(\Fano)$, the critical sets of $\Fano$ are of the form $A\times Y$ for some $A\subseteq X$, so $\Crit(\Fano)=\{A\times Y\mid A\subseteq X\}$.

          We now show that $\Crit(\Fano)$ is the smallest dominating subfamily of $\Fano$ Let $\mathcal{G}\subseteq\Fano$ be such that it dominates $\Fano$. 
Since $\Crit(\Fano)\subseteq\Fano=\bigcup_{G\in\mathcal{G}}\mathcal{D}_{G}$ for some convex families $\mathcal{D}_{G},G\in\mathcal{G}$, we have that for every $A\in\Crit(\Fano)$ there is some $G\in\mathcal{G}$ such that $A\in\mathcal{D}_{G}$. Since $\mathcal{D}_G$ is convex, $[A,G]\subseteq\Fano$, and hence $A\in\partial_G(\Fano)$.
Then 
$\partial_A(\Fano)\subseteq \partial_G(\Fano)$ and since $A$ is critical, 
we have $A=G$. Thus $A\in\mathcal{G}$, and $\Crit(\mathcal{A})\subseteq\mathcal{G}$.
         
        \item[(iii)] Let $\operatorname{id}_X=\{(a,a)\mid a\in X\}$. We observe that \( \Crit (\Finc) = \{ R_A \mid A \subseteq X \} \), where for \( A \subseteq X \),  \( R_A = (A \times X) \cup \operatorname{id}_X \), and refer to \cite{Hella_Luosto_Vaananen_2024} for details of this claim.
        Clearly \( R_A \in \Finc \) and
    \[
    \partial_{R_A}(\Finc) = \left\{ T \in \Finc \mid \Dom (T\setminus\operatorname{id}_X)\subseteq A\subseteq \Ran (T) \right\}.
    \]

We show that  $\Crit (\Finc)\setminus\{ R_{\{a\}} \mid a \in X \}$ is a family of smallest size that dominates $\Finc$ and hence $\Dim(\Finc)=2^l-l$. We first note that for any $A \subseteq X$, $|A|\geq 2$, there is a set in $\partial_{R_A}(\Finc)$ that is not in the shadow $\partial_{R_B}(\Finc)$ for any $B\neq A$. Let $f=A\times A\setminus id_X$. Then $\Dom (f\setminus id_X )= \Dom (f)=A = \Ran (f)$, so $f\in \partial_{R_A}(\Finc)\setminus \partial_{R_B}(\Finc)$ by definition. Similarly for $A=\emptyset\neq B$, $\emptyset\in \partial_{R_{\emptyset}}\setminus \partial_{R_B}$. Hence $\Crit (\Finc) \setminus \{R_{A}\}$ does not dominate $\Finc$ for any  $A \subseteq X$ which is not a singleton. 

For singletons $A=\{a\}$, we show that each member of $\partial_{R_{\{a\}}}(\Finc)$ is in some shadow  $\partial_{R_B}(\Finc)$ with $B\neq {\{a\}}$. %, i.e.,  the members of their shadows are not wholly contained in another set's shadow, but distributed over multiple of them. 
Let $T\in\partial_{R_{\{a\}}}(\Finc)$. Either $T\subseteq id_X$ and thus $T\in\partial_{\emptyset}(\Finc)$, or there is a nonempty set $B=\{b\in X\mid (a,b)\in T\}$ for which $\Dom (T\setminus\operatorname{id}_X)=\{a\}\subseteq B\subseteq \Ran (T)$, hence $T\in\partial_{R_B}$. 

        \item[(iv)] Again, note that $\Fexc$ is downward closed,% and \textcolor{red}{finite}, 
        so by \Cref{dc and uc prop}, $\Dim(\Fexc)=|\Max(\Fexc)|$. The maximal sets in $\Fexc$ are of the form $A\times(X\setminus A)$, where $\emptyset\neq A\subsetneq X$. This means that the upper dimension of $\Fexc$ is the number of nonempty proper subsets of $X$, i.e., $\Dim(\Fexc)=2^\ell-2$.
        \item[(v)] 
        In order to find the critical sets, we will first find the convex shadow  $\partial_{A\times B}(\Find)$. 
        Suppose first that $|A|,|B|\geq 2$. Then $\partial_{A\times B}(\Find)=\{A\times B\}$, because removing or adding any element $(x,y)\in X\times Y$ from or to $A\times B$ would make the set no longer a Cartesian product. Suppose then that $|A|\leq 1$ or $|B|\leq 1$. We show that then $\partial_{A\times B}(\Find)=\mathcal{P}(A\times B)$. If either $A$ or $B$ is empty, the claim is trivial. So by symmetry, we may assume that $A=\{a\}$. Now every subset of $A\times B$ is of the form $A\times B'$, where $B'\subseteq B$, and therefore $\partial_{A\times B}(\Find)=\mathcal{P}(A\times B)$.

        Then by the definition of critical sets, we have $\Crit(\Find)=\{A\times B\mid |A|\geq 2,|B|\geq 2\}\cup\{\{a\}\times Y\mid a\in X\}\cup\{Y\times \{b\}\mid b\in Y\}$. 
        For any $\mathcal{G}\subsetneq\Crit(\Find)$, $\bigcup_{G\in\mathcal{G}}\partial_{G}(\Find)\neq\Find$, so by \Cref{Zorn lemma} \cref{Zorn item}, $\Dim(\Find)=|\Crit(\Find)|=(2^\ell-\ell-1)(2^n-n-1)+\ell+n$.
    \end{itemize}
\end{proof}

\end{document}